\documentclass[addpoints,12pt,letter]{article}

\usepackage{amscd}
\usepackage{amsfonts}
\usepackage{amsmath}
\usepackage{amssymb}
\usepackage{amsthm}
\usepackage{authblk}
\usepackage{array}
\usepackage{booktabs}
\usepackage{chngpage}
\usepackage{color}
\usepackage{commath}
\usepackage{comment}
\usepackage{diagbox}
\usepackage[shortlabels]{enumitem}
\usepackage{esint}
\usepackage{fge}
\usepackage{float}
\usepackage{graphics}
\usepackage{graphicx}
\usepackage[colorlinks=true, citecolor=blue, linkcolor=blue]{hyperref}
\usepackage{latexsym}
\usepackage{mathrsfs}
\usepackage{mathtools}
\usepackage{multicol}
\usepackage{multirow}
\usepackage{nccmath}
\usepackage{pgfplots}
\usepackage{physics}
\usepackage{subcaption}
\usepackage{textcomp}
\usepackage[sc]{titlesec}
\usepackage{tikz}
\usepackage{tikz-cd}
\usepackage[utf8x]{inputenc}
\usepackage{xcolor}

\DeclareMathOperator{\diag}{diag}
\DeclareMathOperator{\dist}{dist}

\DeclareMathOperator{\e}{e}

\DeclareMathOperator{\im}{im}
\DeclareMathOperator{\ind}{ind}

\DeclareMathOperator{\spn}{span}

\DeclareMathOperator{\wind}{wind}

\DeclareMathOperator{\mesh}{mesh} 
\DeclareMathOperator{\mi}{i} 
\newcommand{\bx}{\boldsymbol{x}}

\makeatletter
\@for\@letter:=A,B,C,D,E,F,G,H,I,J,K,L,M,N,O,P,Q,R,S,T,U,V,W,X,Y,Z\do{%
\expandafter\edef\csname b\@letter\endcsname{\noexpand\mathbb{\@letter}}}
\@for\@letter:=A,B,C,D,E,F,G,H,I,J,K,L,M,N,O,P,Q,R,S,T,U,V,W,X,Y,Z\do{%
\expandafter\edef\csname c\@letter\endcsname{\noexpand\mathcal{\@letter}}}
\makeatother

\newcommand{\al}{\alpha}
\newcommand{\be}{\beta}
\newcommand{\de}{\delta}
\newcommand{\De}{\Delta}
\newcommand{\ds}{\displaystyle}

\newcommand{\eps}{\varepsilon}

\newcommand{\ga}{\gamma}

\newcommand{\ka}{\kappa}

\newcommand{\la}{\lambda}
\newcommand{\La}{\Lambda}
\newcommand{\nf}{\infty}
\newcommand{\ol}{\overline}

\newcommand{\Om}{\Omega}
\newcommand{\ph}{\varphi}
\newcommand{\Ph}{\Phi}

\newcommand{\sm}{\setminus}
\newcommand{\tht}{\theta}
\newcommand{\ti}{\tilde}

\newcommand{\ze}{\zeta}

\newcommand{\SR}{S}
\newcommand{\TF}{T}

\renewcommand{\ge}{\geqslant}
\renewcommand{\d}{\dif}
\renewcommand{\le}{\leqslant}

\renewcommand{\sp}{\mathop{\mathrm{sp}}}

\makeatletter
\renewcommand*\env@matrix[1][c]{\hskip-\arraycolsep
\let\@ifnextchar\new@ifnextchar
\array{*\c@MaxMatrixCols #1}}
\makeatother

\usetikzlibrary{arrows,calc,shapes,positioning}
\usetikzlibrary{decorations.markings}
\tikzstyle arrowstyle=[scale=1]
\tikzstyle directed=[postaction={decorate,decoration={markings,mark=at position .65 with {\arrow[arrowstyle]{stealth}}}}]
\tikzstyle reverse directed=[postaction={decorate,decoration={markings,mark=at position .65 with {\arrowreversed[arrowstyle]{stealth};}}}]

\numberwithin{equation}{section}

\newtheorem{lemma}{Lemma}[section]
\newtheorem{theorem}[lemma]{Theorem}
\newtheorem{proposition}[lemma]{Proposition}
\newtheorem{corollary}[lemma]{Corollary}

\theoremstyle{definition}

\newtheorem{remark}[lemma]{Remark}

\begin{document}

\title{Condition numbers of block Toeplitz matrices and stability of space-time IgA approximations for the wave and Schrödinger equations}
\date{\today}
\author[1]{M. Bogoya\thanks{johan.bogoya@correounivalle.edu.co}}
\author[2]{A. B\"ottcher\thanks{aboettch@mathematik.tu-chemnitz.de}}
\author[3]{M. Ferrari\thanks{m.ferrari@unipv.it}}
\author[4,5]{S.M. Grudsky\thanks{grudsky@math.cinvestav.mx}}
\author[6,7]{S.~Serra-Capizzano\thanks{s.serracapizzano@uninsubria.it}}

\affil[1]{\footnotesize Departamento de Matem\'aticas, Universidad del Valle, Cali, Colombia}
\affil[2]{\footnotesize Fakult\"at f\"ur Mathematik, Technische Universit\"at Chemnitz, Chemnitz, Germany}
\affil[3]{\footnotesize Dipartimento di Matematica, Universit\`a di Pavia, Pavia, Italy}
\affil[4]{\footnotesize Departamento de Matem\'aticas, CINVESTAV del IPN, CDMX, Mexico}
\affil[5]{\footnotesize Southern Federal University, Regional Mathematical Center, Rostov-on-Don, Russia}
\affil[6]{\footnotesize Dipartimento di Scienza e Alta Tecnologia, Universit\`a degli Studi dell'Insubria, Como, Italy}
\affil[7]{\footnotesize Division of Scientific Computing, Department of Information Technology, Uppsala University, Uppsala, Sweden}
\date{\footnotesize\today}

\maketitle

\begin{abstract}
In previous work by several authors, the behavior of the condition numbers of banded Toeplitz  matrices was studied as the matrix size tends to infinity.
In the present contribution, two main directions are pursued.
As a first step, we extend this study to block Toeplitz matrices with blocks of fixed size $N$.
As in the scalar case, we show that even when the symbol generates a Fredholm infinite Toeplitz operator, the condition numbers of the finite matrices may grow at least exponentially.
Upper and lower bounds for the condition numbers are obtained, and examples showing that they may grow arbitrarily fast are presented.
Then, as a second step, we apply the developed theory to the stability analysis of space-time Galerkin methods, where in time an Isogeometric approach is used with regularity $r$, $1\le r\le p-1$, $p$ being the employed polynomial degree.
These stability issues are related exactly to the conditioning of block Toeplitz-like matrices with blocks of size $N=p-r$.
Specific examples are treated in detail and related numerical experiments are presented and critically discussed.
We finally present a short list of relevant open problems.
\end{abstract}

\section{Introduction}

In the present work we are interested in the stability of numerical methods when approximating hyperbolic partial differential equations (PDEs) using isogeometric analysis (IgA)~\cite{CoHu09,HuCoBa05} in time with polynomial degree $p\ge 1$ and intermediate regularity $1\le r\le p-1$.
The PDEs considered here are the linear wave equation and the linear Schrödinger equation.
More precisely, we consider the linear acoustic wave equation
\begin{equation}\label{eq:45}
\begin{cases}
\partial_{t}^{2} U(\bx,t)-\De_{\bx} U(\bx,t)=F(\bx,t), & (\bx,t)\in Q_{\TF}=\Om\times (0,\TF),\\
U(\bx, t)=0, & (\bx, t)\in\partial\Om\times (0, \TF),\\
U(\bx, 0)=0,\quad\partial_{t} U(\bx, 0)=0, &\bx\in\Om,
\end{cases}
\end{equation}
and the linear Schrödinger equation
\begin{equation}\label{eq:426}
\begin{cases}
\mi\partial_{t}\psi(\bx,t)-\De_{\bx}\psi(\bx,t)=F(\bx,t), & (\bx,t)\in Q_{\TF},\\[1ex]
\psi(\bx, t)=0, & (\bx, t)\in\partial\Om\times (0, \TF ),\\[1ex]
\psi(\bx, 0)=0, & \bx\in\Om,
\end{cases}
\end{equation}
with $\Om\subset\bR^{d}$ ($d=1,2,3$), $\TF>0$, and $F$ being a given source term(real-valued for \eqref{eq:45}, complex-valued for \eqref{eq:426}) in $L^{2}$ over $Q_{\TF}$.
After considering space-time weak formulations, both problems are approximated in time by the IgA method on an equispaced mesh with $p\ge 1$ and regularity $1\le r\le p-1$.
For the case of maximal regularity $r=p-1$, the stability of the corresponding numerical approximations has been investigated in~\cite{FeFr26,FeFrLoPe25,FeGo25} by using tools such as the asymptotic conditioning of real banded nonsymmetric Toeplitz matrix-sequences.
More in detail, we employ the results of~\cite{AmoBru96}, which relate the asymptotic behavior of the condition number to the number and position of the zeros of the Laurent polynomial generating the considered banded Toeplitz matrix-sequence.
In this direction, an exhaustive study of the asymptotic conditioning of real banded nonsymmetric Toeplitz matrix-sequences from the point of view of operator theory is presented in~\cite{BoGr99}.

On the other hand, when considering an intermediate regularity such that $N=p-r\ge 2$, the structures are again of banded Toeplitz type, but with blocks of fixed size $N\ge 2$ and possibly with some low rank correction term, i.e., they are of block Toeplitz-like nature.
For block Toeplitz matrix-sequences distribution and extremal results for singular values and eigenvalues exist both in the operator theory community~\cite{Goh1,Goh2,Goh3,WiBlo3,WiBlo2,WiBlo1} and in the numerical analysis community~\cite{BaGa20b,BaGa20a,DoNeySe12,Se99e,Se99f,SeTi99,Ti98a};
see also~\cite{GaSe15,GaSp19} for specific spectral results in the context of high order finite element methods and IgA with intermediate regularity.

However, despite the rich literature on the subject, a block equivalent of the systematic works as in~\cite{AmoBru96,BoGr99} does not exist so far for the asymptotic conditioning, and to fill this gap is our first task in the current work.
Our findings are derived following the notation in~\cite{BoGr99} and then they are interpreted with the terminology in~\cite{AmoBru96}.
As a final part, we apply them to the stability analysis of the considered space-time discretizations with various choices of the parameters $r, p$ so that $N=p-r\ge 2$, by considering the pure block Toeplitz structures and ignoring the low rank correction terms.
Numerical tests are performed and discussed, followed by open problems and proposals for future work.

\section*{Main novelties}

The main novelties of the present work are listed below.
The first of them is more of operator theory nature, the third of numerical analysis nature, and the second consists in bridging the two.
\begin{itemize}
\item As a step towards generalizing the work in~\cite{AmoBru96,BoGr99}, we obtain results on the conditioning of block Toeplitz matrix-sequences, by describing three regimes, uniformly bounded, polynomially unbounded, exponentially unbounded.
\item We rewrite the results in the language of the work~\cite{AmoBru96}, whose results have been used in the recent numerical analysis literature; see~\cite{FeFr26, FeFrLoPe25, FeGo25} and references therein.
\item As a step in the numerical direction, we apply the theoretical results to stability issues of the matrix-sequences arising from the IgA approximation with intermediate regularity of the linear wave and linear Schrödinger equations, focusing on a selection of representative cases rather than aiming at an exhaustive treatment. This generalizes the results in~\cite{FeFr26,FeGo25} which have been developed for the case of maximal regularity, that is, for $N=p-r=1$, and which corresponds to the standard Toeplitz setting.
\end{itemize}

The current work is organized as follows.
Section~\ref{sec:prelim} is devoted to preliminary results concerning block Toeplitz matrices of infinite and finite order and to the related matrix-sequences.
Section~\ref{sec:main} contains the main theoretical results regarding the asymptotic conditioning and distinguishing between the case where the generating function is nonsingular on the unit circle and the case where it is singular.
Section~\ref{sec:connection} represents a link between the results of the previous section and those used in the numerical analysis community.
In Section~\ref{sec:application} we deal with the stability issues of proper IgA approximations with intermediate regularity of the wave and Schrödinger equations: the section is accompanied by numerical experiments and by related comments.
Finally, in Section~\ref{sec:end}, we draw conclusions and we report a concise list of open problems.

\section{Preliminaries}\label{sec:prelim}

The present section is divided into three parts and contains preliminary results concerning block Toeplitz matrices and matrix-sequences.
In Subsections~\ref{ssec:prelim1} and~\ref{ssec:prelim2} the cases of infinite order and finite order are considered, respectively.
In Subsection~\ref{ssec:prelim3} the rational symbols setting is treated in detail.
Everything in this section is well known and can be found in the books~\cite{BoSi99,GoFe}, for example.

\subsection{Infinite block Toeplitz matrices}\label{ssec:prelim1}

For a positive integer $N$, a matrix $\mathbf{M}=(M_{j-k})_{j,k=0}^{\nf}$, where each $M_{k}$ is itself an $N\times N$ complex matrix, is called a {\em block Toeplitz matrix}.
See Figure~\ref{fg:Block}.

\begin{figure}[h]
\centering
\begin{tikzpicture}[scale=1]
\def\n{5}
\def\gap{0.1}
\def\size{1}
\foreach \i in {0,...,4} {
\foreach \j in {0,...,4} {
\pgfmathtruncatemacro{\d}{\i-\j}
\pgfmathsetmacro{\x}{\j*(\size+\gap)}
\pgfmathsetmacro{\y}{-\i*(\size+\gap)}
\draw (\x,\y) rectangle++(\size,-\size);
\pgfmathtruncatemacro{\last}{(\i==4)||(\j==4)}
\ifnum\last=1
\node at (\x+0.5*\size,\y-0.5*\size) {$\ddots$};
\else
\node at (\x+0.5*\size,\y-0.5*\size) {$M_{\d}$};
\fi}}
\node at (-0.6,-2.7) {$\mathbf{M}=$};
\end{tikzpicture}
\caption{The block Toeplitz matrix $\mathbf{M}$.
Each block $M_{k}$ belongs to $\bC^{N\times N}$ for some positive integer $N$.}
\label{fg:Block}
\end{figure}

Each block $M_{k}$ can be seen as an operator acting on the finite-dimensional space $\bC^{N}$.
Therefore, $\textbf{M}$ can be viewed as an operator acting on the $\bC^{N}$-valued $\ell^{2}$ space, that is, the space $\ell_{N}^{2}$ consisting of all sequences $(x_{0},x_{1},\ldots)$ with $x_{j}\in\bC^{N}$, equipped with the usual $\ell^{2}$ norm.

For a real number $1\le p\le\nf$, let $L_{N\times N}^{p}(\bT)$ denote the space of all matrix-valued functions $A:\bT\to\bC^{N\times N}$ whose entries belong to $L^{p}(\bT)$.
If there is a matrix-valued function $A\in L_{N\times N}^{1}(\bT)$ whose Fourier coefficients $\hat{A}_{k}$ coincide with $M_{k}$ for all $k\in\bZ$, that is, if
\[
M_{k}=\hat{A}_{k}:=\frac{1}{2\pi}\int_{0}^{2\pi}A(\e^{\mi\tht})\e^{-\mi k\tht}\d\tht\quad(k\in\bZ),
\]
then this matrix function is uniquely defined, it is called the {\em symbol} of $\textbf{M}$, and the matrix $\textbf{M}$ is denoted by $T(A)$.
It is well known that $T(A)$ is bounded on $\ell_{N}^{2}$ if and only if $A\in L_{N\times N}^{\nf}(\bT)$.
In that case
\[
\|T(A)\|=\|A\|_{\nf},
\]
where $\|A\|_{\nf}$ denotes the $C^{\ast}$-norm on $L_{N \times N}^{\nf}(\bT)$, that is, the essential supremum over $t\in\bT$ of the square root of the largest eigenvalue of the $N \times N$ matrix $A(t)^{\ast}A(t)$.

Let $R_{N\times N}(\bT)$ denote the set of all rational matrix-valued functions $A:\bT\to\bC^{N\times N}$ having no poles on~$\bT$.
An $A\in R_{N\times N}(\bT)$ can be viewed as a matrix-valued function with rational entries but also as the quotient $A(t)={S(t)}/{q(t)}$, where $S$ is a matrix-valued function with polynomial entries and $q$ is a scalar polynomial.
Hence, the condition of having no poles on~$\bT$ means that $q$ has no zeros there.
We denote by $GR_{N\times N}(\bT)$ the matrix-functions $A\in R_{N\times N}(\bT)$ for which $\det A(t)\ne 0$ on~$\bT$.
Clearly, if $A\in R_{N\times N}(\bT)$, then $A^{-1}: t \mapsto A^{-1}(t)$ also belongs to $R_{N\times N}(\bT)$.

A bounded linear operator $B$ acting on a Hilbert space $H$, $B\in\cB(H)$, is called {\em Fredholm} if its image $\im(B)$ is closed and
\[
\dim(\ker(B))=:\al<\nf,\qquad
\dim(H/\im(B))=:\be<\nf.
\]
In such a case, the number $\ind(B):=\al-\be$ is called the {\em index} of $B$.
The following is a classical result.
It was explicitly stated by Simonenko in~\cite{Si64} for the first time.

\begin{theorem}\label{th:Sim}
If $A$ belongs to $L_{N\times N}^{\nf}(\bT)$ and $T(A)$ is Fredholm on $\ell_{N}^{2}$, then $A$ is invertible in $L_{N\times N}^{\nf}(\bT)$.
\end{theorem}

It follows that if $A\in R_{N\times N}(\bT)$ and $T(A)$ is Fredholm, then necessarily $A\in GR_{N\times N}(\bT)$.
It turns out that belonging to $GR_{N\times N}(\bT)$ is even sufficient for $T(A)$ to be Fredholm.
To determine the index and to decide whether the operator is invertible we need more information.
In this connection the following theorem, which appears, for example, as Theorem 2.9 in~\cite{LiSp87}, is of importance.

\begin{theorem}\label{th:WH}
Suppose $A\in GR_{N\times N}(\bT)$.
Then $A$ admits a Wiener-Hopf (WH) factorization, that is, there exist integers
\[
\ka_{1}\le\ka_{2}\le\cdots\le\ka_{N}
\]
such that
\[
A(t)=A_{-}(t)\diag(t^{\ka_{1}},\ldots,t^{\ka_{N}})A_{+}(t),
\]
where $A_{+}$ and $A_{-}$ are matrix-valued functions in $GR_{N \times N}(\bT)$, analytic and invertible inside and outside~$\bT$ (including the point at infinity), respectively.
\end{theorem}

The numbers $\ka_{j}$ are called {\em partial indices}.
The conditions on $A_\pm$ imply power series representations
\[
A_{-}(z)=\sum_{k=-\nf}^{0} B_{k} z^{k} \;\:(|z|\ge 1), \quad A_{+}(z)=\sum_{k=0}^{\nf} C_{k} z^{k} \quad (|z|\le 1)
\]
with invertible $B_{0},C_{0}$ and the inverses are given by analogous power series,
\[
A_{-}^{-1}(z)=\sum_{k=-\nf}^{0} D_{k} z^{k} \;\:(|z|\ge 1), \quad A_{+}^{-1}(z)=\sum_{k=0}^{\nf} E_{k} z^{k} \quad (|z|\le 1).
\]

\begin{theorem}\label{th:inv}
If $A\in GR_{N\times N}(\bT)$, then $T(A)$ is Fredholm with
\[
\ind(T(A))=-(\ka_{1}+\cdots+\ka_N),
\]
and $T(A)$ is invertible if and only if $\ka_{1}=\cdots=\ka_N=0$.
\end{theorem}

The connection between the previous two theorems is that a WH-factorization of the symbol $A$ yields the factorization
\begin{equation}\label{WHO}
T(A)=T(A_{-})T(D)T(A_{+})\;\:\mbox{with}\;\:
D=\diag(t^{\ka_{1}},\ldots,t^{\ka_{N}})
\end{equation}
and that the (block-triangular) outer operators $T(A_\pm)$ are invertible, the inverses being $T(A_\pm^{-1})$.
We obtain in particular that
\[
\ind(T(A))=\ind(T(D))=-\sum_{j=1}^{N}\ka_{j},
\]
and that $T(A)$ is invertible if and only if $T(D)$ is invertible, which in turn happens if and only if $\ka_{j}=0$ for all $j$.

The following well known result provides an expression for the index of $T(A)$ that does not need knowledge of a WH-factorization.
Given a continuous function $\ph: \bT\to\bC \sm \{0\}$, we denote by $\wind(\ph)$ the winding number of this function about the origin.

\begin{theorem}\label{th:Fredholm}
If $A\in GR_{N\times N}(\bT)$, then $T(A)$ is Fredholm with
\[
\ind(T(A))=-\wind(\det A).
\]
\end{theorem}

For $A\in R_{N\times N}(\bT)$ we introduce the {\em associate symbol} $\ti{A}$ and the {\em conjugate transpose} symbol $A^{\ast}$ by
\[
\ti{A}(t):=A(1/t)=\sum_{k=-\nf}^{\nf} \hat{A}_{k} t^{-k}, \quad A^{\ast}(t)=A(t)^{\ast}=\sum_{k=-\nf}^{\nf} \hat{A}_{k}^{\ast}t^{-k}\quad (t\in\bT).
\]
In the scalar case, $N=1$, the matrix $T(\ti{A})$ is the transpose of $T(A)$ and hence invertible if and only if so is $T(A)$.
This is no longer true in general for $N>1$.
The partial indices of $\ti{A}$ are in no obvious way related to those of $A$.
However, it is easy to see that $T(A^{\ast})=T^{\ast}(A)$, that (therefore) $T(A)$ is invertible if and only if so is $T(A^{\ast})$, and that the partial indices of $A^{\ast}$ are the negatives of those of $A$.
Here is an example.
Let $A(t)=\begin{pmatrix} t & 1\\ 0 & t^{-1}\end{pmatrix}$ for $t\in\bT$.
We then have the following WH-factorizations:
\begin{eqnarray*}
& & A(t)=\begin{pmatrix} t & 1\\ 0 & t^{-1}\end{pmatrix}=
\begin{pmatrix} 1 & 0\\ t^{-1} &-1\end{pmatrix}\begin{pmatrix} 1 & 0\\ 0 & 1\end{pmatrix}\begin{pmatrix} t & 1\\ 1 & 0\end{pmatrix},\\
& & \ti{A}(t)=\begin{pmatrix} t^{-1} & 1\\ 0 & t\end{pmatrix}=
\begin{pmatrix} 1 & 0\\ 0 & 1\end{pmatrix}\begin{pmatrix} t^{-1} & 0\\ 0 & t\end{pmatrix}\begin{pmatrix} 1 & t\\ 0 & 1\end{pmatrix},\\
& & A^{\ast}(t)=\begin{pmatrix} t^{-1} & 0\\ 1 & t\end{pmatrix}=
\begin{pmatrix} t^{-1} & 1\\ 1 & 0\end{pmatrix}\begin{pmatrix} 1 & 0\\ 0 & 1\end{pmatrix}\begin{pmatrix} 1 & t\\ 0 &-1\end{pmatrix}.
\end{eqnarray*}
Thus, the partial indices of $A$ and $A^{\ast}$ are $0,0$, while those of $\ti{A}$ are $-1,1$.

\subsection{Finite block Toeplitz matrices}\label{ssec:prelim2}

We define the projections $P_{n}:\ell_{N}^{2}\to\ell_{N}^{2}$ by
\[
P_{n}(x_{0},x_{1},\ldots)=(x_{0},\ldots,x_{n-1},0,0,\ldots).
\]
A bounded linear operator $B$ acting on $\ell_{N}^{2}$ has a block-matrix representation $(B_{j,k})_{j,k=0}^{\nf}$ in the natural way.
The operator $P_{n}BP_{n}$, acting on $P_{n}\ell_{N}^{2}$, is then given by an $nN\times nN$ matrix, and we freely identify this operator and the matrix.
In the case where $B=T(A)$, we denote the $nN\times nN$ matrix $P_{n}T(A)P_{n}$ by $T_{n}(A)$.
This is a {\em finite block Toeplitz matrix}.

Our goal is to obtain asymptotic estimates for the condition numbers
\[
\ka(T_{n}(A)):=\|T_{n}(A)\|\,\|T_{n}^{-1}(A)\|
\]
of the matrices $T_{n}(A)$ with $A\in R_{N\times N}(\bT)$ as $n\to\nf$.
The norm on the right is the spectral norm ($=$ operator norm) on $\bC^{nN\times nN}$, we prefer writing $T_{n}^{-1}(A)$ instead of $T_{n}(A)^{-1}$ for the inverse, and we put $\ka(T_{n}(A))=\nf$ if $T_{n}(A)$ is not invertible.

The block-matrix sequence $\{T_{n}(A)\}_{n=1}^{\nf}$ is called {\em stable} if there exists a natural number $n_{0}$ such that all $T_{n}(A)$ with $n\ge n_{0}$ are invertible and the norms $\|T_{n}^{-1}(A)\|$ are uniformly bounded for $n\ge n_{0}$.

\begin{theorem}[Theorem VIII.5.3 in~\cite{GoFe}, Theorem 6.9 in~\cite{BoSi99}]\label{th:Stable}
Suppose that $A$ is a continuous matrix-valued function.
Then $\{T_{n}(A)\}_{n=1}^{\nf}$ is stable if and only if $T(A)$ and $T(\ti{A})$ are invertible.
\end{theorem}

\subsection{Exponential decay of the Fourier coefficients}\label{ssec:prelim3}

A matrix-function $A\in R_{N \times N}(\bT)$ has no poles on~$\bT$ and hence also no poles in an open annular neighborhood of~$\bT$.
Thus, there exist numbers $0<r<1<R$ such that $A$ has no poles in the annulus $\{z\in\bC: r\le|z|\le R\}$ and, consequently, is analytic there.
Cauchy's theorem therefore gives
\[
\hat{A}_{j}=\frac{1}{2\pi \mi}\int_{|z|=R}\!A(z)\,\frac{\d z}{z^{j+1}}, \quad
\hat{A}_{-j}=\frac{1}{2\pi \mi}\int_{|z|=r}\! A(z)\,\frac{\d z}{z^{-j+1}}=\frac{1}{2\pi \mi}\int_{|z|=r}\! A(z)z^{j-1}\,\d z
\]
for $j\ge 0$.
Taking spectral norms in these equalities we get the following.

\begin{lemma}\label{lm:FourierBound}
Suppose that $A\in R_{N\times N}(\bT)$.
Then, for some $0<r<1<R$, all its entries have no poles in the annulus $\{z\in\bC: r\le|z|\le R\}$ and there exist constants $c_{1},c_{2}$ such that
\[
\|\hat{A}_{j}\|<\frac{c_{1}}{R^{j}}\quad\mbox{and}\quad
\|\hat{A}_{-j}\|<c_{2}r^{j},
\]
for all $j\ge 0$.
\end{lemma}

\section{Theoretical results}\label{sec:main}

This section contains the main theoretical results.

\subsection{Determinants without zeros on~$\bT$}\label{ssec:main1}

\begin{theorem}\label{th:low}
Suppose that $A\in GR_{N\times N}(\bT)$.
Then the condition numbers $\ka(T_{n}(A))$ are uniformly bounded as $n\to\nf$ if and only if $T(A)$ and $T(\ti{A})$ are invertible.
If at least one of the operators $T(A)$ or $T(\ti{A})$ fails to be invertible, then there exist positive constants $c,\al$, independent of $n$, such that
\[
\ka(T_{n}(A))\ge c\e^{\al n},
\]
for all $n\ge 1$.
\end{theorem}

\begin{proof}
Let both $T(A)$ and $T(\ti{A})$ be invertible.
Since $A$ is continuous, Theorem~\ref{th:Stable} tells us that there exist positive constants $c,n_{0}$ such that $\|T_{n}^{-1}(A)\|\le c$ for all $n\ge n_{0}$.
Since the finite sections $T_{n}(A)$ are compressions of $T(A)$, we have
\[
\|T_{n}(A)\|\le\|T(A)\|=\|A\|_{\nf}.
\]
Consequently, $\ka(T_{n}(A))\le c\|A\|_{\nf}$ for every $n\ge n_{0}$, which proves the first part.

We continue with the second part.
Assume first that $T(A)$ is not invertible.
From Theorems~\ref{th:WH} and~\ref{th:inv} we know that $A$ admits a WH-factorization and that at least one of the partial indices $\ka_{1}\le\ka_{2}\le\cdots\le\ka_{N}$ is nonzero.
Suppose that $\ka_{1}<0$.
Let $e_{1}\in\ell_{N}^{2}$ be the sequence $\{x_{0},x_{1},\ldots\}\in\ell_{N}^{2}$ given by $x_{0}=(1,0,\ldots,0)^\top$ and $x_{k}=0$ for $k\ge 1$.
The vector $\xi=T(A_{+}^{-1})e_{1}$ belongs to the kernel of $T(A)$.
Indeed, from~(\ref{WHO}) we infer that $T(A)\xi=T(A_{-})T(D)e_{1}$.
The Fourier series of $D$ is
\[
D(t)=D_{1} t^{\ka_{1}}+D_{2} t^{\ka_{2}}+\cdots+D_N t^{\ka_N} \quad (t\in\bT)
\]
where $D_{j}$ is the diagonal matrix whose $j,j$ entry is $1$ and the other entries of which are zero.
The nonzero entries of the first column of the block Toeplitz matrix $T(D)$ come from the nonzero entries of the first columns of the matrices $D_{j}$ with $\ka_{j}\ge 0$.
Since $\ka_{1}<0$, we can exclude $D_{1}$.
As the first columns of the remaining matrices $D_{2}, \ldots, D_N$
are all zero, it follows that the first column of $T(D)$ is zero as well, which gives $T(D)e_{1}=0$, as claimed.

Let $Q_{n}:=I-P_{n}$, i.e., $Q_{n}: \{x_{0},x_{1},\ldots\}\mapsto \{0, \ldots,0, x_{n}, x_{n+1}, \ldots\}$.
We have
\[
0=P_{n}T(A)\xi=P_{n}T(A)P_{n}\xi+P_{n}T(A)Q_{n}\xi,
\]
or equivalently, $T_{n}(A)\xi_{n}=-P_{n}T(A)Q_{n}\xi$ for $\xi_{n}=P_{n}\xi$.
From Lemma~\ref{lm:FourierBound} we infer that there is some $R>1$ such that
\[
\|Q_{n}\xi\|^{2}=\sum_{k\ge n}\|\xi_{k}\|^{2}\le\sum_{k\ge n}\frac{c_{1}^{2}}{R^{2k}}=\frac{c_{1}^{2}R^{2}}{R^{2}-1}\,\frac{1}{R^{2n}},
\]
implying that $\|Q_{n}\xi\|\le C\e^{-\al n}$ with some constant $C$ and with $\al=\log R>0$.
Thus,
\[
\|T_{n}(A)\xi_{n}\|=\|P_{n} T(A)Q_{n}\xi\|\le\|P_{n} T(A)\|\,\|Q_{n}\xi\|\le C\|A\|_{\nf} \e^{-\al n}.
\]
Since
\[
\|T_{n}^{-1}(A)\|=
\sup\{|T_{n}^{-1}(A)x|:|x|=1\}=
\frac{1}{\inf\{|T_{n}(A)x|:|x|=1\}},
\]
it follows that
\[
\|T_{n}^{-1}(A)\|\ge\frac{1}{\|T_{n}(A)\xi_{n}\|/\|\xi_{n}\|}\ge\frac{\|\xi_{n}\|}{C\|A\|_{\nf}}\e^{\al n}.
\]
As $\xi_{n}=P_{n} \xi\to\xi$ and $T_{n}(A)=P_{n}T(A)P_{n}$ converges strongly to $T(A)$, there is an $n_{0}$ such that $\|\xi_{n}\|\ge\|\xi\|/2$ and $\|T_{n}(A)\|\ge\|A\|_{\nf}/2$.
Consequently, for $n\ge n_{0}$ we obtain
\[
\ka(T_{n}(A))=\|T_{n}(A)\|\,\|T_{n}^{-1}(A)\|\ge\frac{\|\xi\|}{4C}\e^{\al n}.
\]
The restriction $n\ge n_{0}$ can be replaced with $n\ge 1$ by adjusting the constants.

To prove the case in which all partial indices of $A$ are non-negative, we work with $A^{\ast}$ instead of $A$, obtaining an operator $T_{n}(A^{\ast})$ with the same condition numbers as $T_{n}(A)$ and whose symbol $A^{\ast}$ has at least one negative partial index.

Finally, the remaining case in which $T(A)$ is invertible while $T(\ti{A})$ is not invertible can be reduced to the previous case simply by using the identity
\[
W_{n} T_{n}(A)W_{n}=T_{n}(\ti{A}),
\]
where $W_{n}$ is given by $W_{n}: \{x_{0}, x_{1}, \ldots\}\mapsto \{x_{n-1}, \ldots ,x_{0},0,0, \ldots\}$.
This identity implies that $\ka(T_{n}(A))=\ka(T_{n}(\ti{A}))$.
\end{proof}

The following result provides an upper bound for $\|T_{n}^{-1}(A)\|$.
Its proof is analogous to that of Theorem 1.3 in~\cite{BoGr99}, and we therefore present it without proof.

\begin{theorem}\label{th:up}
Suppose that $A\in GR_{N\times N}(\bT)$ and that at least one of the operators $T(A)$ or $T(\ti{A})$ is not invertible.
Then
\[
\|T_{n}^{-1}(A)\|\le\frac{e^{\beta n}}{|\det T_{n}(A)|}
\]
with some constant $\beta>0$.
\end{theorem}

The upper bound in Theorem~\ref{th:up} depends on $\det T_{n}(A)$.
If all eigenvalues of the matrix $T_{n}(A)$ are ``small'', then its determinant will also be small, making this upper bound large.
In fact, in~\cite{BoSi99} scalar examples are presented in which the condition numbers grow faster than any predetermined growth rate.
On the other hand, if all eigenvalues are uniformly bounded away from zero, then this upper bound becomes exponential.
To show this, we need to introduce the concept of the {\em limit set}.

Let $\sp T_{n}(A)$ denote the spectrum of $T_{n}(A)$, that is, the set of its eigenvalues.
The set
\[
\La(A):=\limsup_{n\to\nf}\sp T_{n}(A),
\]
which is the collection of all partial limits of the sequence $\{\sp T_{n}(A)\}_{n\ge1}$, is called the {\em limit set} of $A$.

\begin{theorem}\label{th:double}
Suppose that $A\in GR_{N\times N}(\bT)$ and that at least one of the operators $T(A)$ or $T(\ti{A})$ is not invertible.
Then, if $0$ does not belong to $\La(A)$, there exist constants $d_{1},d_{2}>0$ and $\al,\be>0$ such that
\[
d_{1}\e^{\al n}\le\ka(T_{n}(A))\le d_{2}\e^{\be n}
\]
for sufficiently large $n$.
\end{theorem}

\begin{proof}
The lower bound follows from Theorem~\ref{th:low}.

Let $\eps:=\dist(0,\La(A))>0$ and let $\{\la_{n,j}\}_{j=1}^{nN}$ be the eigenvalues of $T_{n}(A)$.
It follows that there exists an~$n_{0}$ such that $|\la_{n,j}|\ge\eps/2$ for $n>n_{0}$.
Since $\det T_{n}(A)=\prod_{j=1}^{nN}\la_{n,j},$ we obtain $|\det T_{n}(A)|\ge({\eps}/{2})^{nN}$.
Combining this with Theorem~\ref{th:up} and using that $\|T_{n}(A)\|\le\|T(A)\|=\|A\|_{\nf}$, we get the upper bound.
\end{proof}

In the scalar case $N=1$, Day's papers~\cite{Da75a,Da75b} provide a constructive description of the limit set $\La(A)$.
Unfortunately, such a description is not available in the block case $N>1$, which somewhat reduces the significance of Theorem~\ref{th:double}.
However, in special cases, such as diagonal block Toeplitz matrices $T_{n}(A)$ with $A\in R_{N\times N}(\bT)$, such a description can be derived.

\subsection{Determinants with zeros on~$\bT$}\label{ssec:main2}

Theorems~\ref{th:Sim} and~\ref{th:Stable} imply the following, which we want to record at the very beginning.

\begin{theorem}\label{th:limsup}
If $A\in R_{N\times N}(\bT)$ and $\det A$ has zeros on~$\bT$, then
\[
\limsup_{n\to\nf}\ka(T_{n}(A))=\nf.
\]
\end{theorem}

Thus, the only question that remains in the situation at hand is the question about the rate of the growth of the condition numbers.
This question is difficult and our insights are accordingly modest.
To get an idea of what happens, we cite the following scalar case result, which was established in~\cite{AmoBru96} for banded Toeplitz matrices and in~\cite{BoGr99} as stated here.
Given two sequences $\{u_{n}\}, \{v_{n}\}$ of positive real numbers, we write $u_{n} \asymp v_{n}$ if there are positive constants $c,d$ such that $cu_{n}\le v_{n}\le du_{n}$ for all~$n$.

\begin{theorem}\label{th:RFact}
Let $a$ be a rational function without poles and with only one zero $t_{0}$ on~$\bT$.
Suppose that we can write
\[
a(t)=(t-t_{0})^{\be}t^{k}\tht(t)\quad(t\in\bT),
\]
where $\be\in\bN$, $k\in\bZ$, and $\tht$ is a rational function without poles or zeros on~$\bT$ whose winding number about the origin is zero.
Consider the interval $J=[-\be,0]$.
Then
\begin{eqnarray*}
\ka(T_{n}(a))&\asymp&n^{\be}\quad\mbox{if}\quad k\in J,\\
\ka(T_{n}(a))&\ge&c\e^{\al n}\quad\mbox{if}\quad k\notin J,
\end{eqnarray*}
for some positive constants $c,\al>0$.
\end{theorem}

For the ``pure zeros'' $(t-t_{0})^\be$ even more is known.

\begin{theorem}\label{th:pure}
Let $a(t)=(t-t_{0})^\be$ with $t_{0}\in\bT$ and $\be\in\bN$.
Then
\[
\ka(T_{n}(a))=2^\be\|K_\be\|n^\be (1+o(1))
\]
as $n\to\nf$, where $\|K_\be\|$ is the norm of the integral operator on $L^{2}(0,1)$ given by
\[
(K_\be f)(x)=\frac{1}{(\be-1)!}\int_{0}^{x} (x-y)^{\be-1}f(y)\,\d y.
\]
\end{theorem}

This was established in~\cite{BoDo09}.
Note that the matrices $T_{n}((t-t_{0})^\be)$ are lower triangular.
It is well known that $\|K_{1}\|=2/\pi$.
More about the norms $\|K_\be\|$ can be found in~\cite{BoDo09}.
Paper~\cite{BoWi06} studies the case $(\be,k)=(2\nu,-\nu)$ of Theorem~\ref{th:RFact} in detail;
note that $(t-t_{0})^{2\nu}t^{-\nu}=(-t_{0})^\nu|t-t_{0}|^{2\nu}$.

The case of symbols with several zeros on~$\bT$ is more intricate.
The simplest example is the one where the symbol has two complex conjugate zeros $\al, \ol{\al}\in\bT$:
\[
a(t)=t^{-1}(t-\al)(t-\ol{\al})=t-2\Re\al+|\al|^{2} t^{-1}\quad(t\in\bT).
\]
The matrix $T_{n}(a)$ is tridiagonal.
With the unitary matrix $U=\diag(1,\al, \al^{2}, \ldots, \al^{n-1})$ we get $UT_{n}(a)U^{\ast}=T_{n}(b)$ where $T_{n}(b)$ is the Hermitian tridiagonal matrix with the symbol $b(t)=\al\, t-2\Re\al+\ol{\al}\, t^{-1}$.
The eigenvalues of $T_{n}(b)$ are known to be
\[
\la_{j}=-2\Re\al+2|\al|\cos\frac{\pi j}{n+1}, \quad j=1,2, \ldots,n;
\]
see, e.g.,~\cite[Theorem 2.4]{BoGr05}.
Consequently,
\[
\|T_{n}^{-1}(a)\|=\|T_{n}^{-1}(b)\|=\frac{1}{\min_{1\le j\le n}\big(-2\Re\al+2|\al|\cos\frac{\pi j}{n+1}\big)}.
\]
The minimum in the denominator may be zero or is nonzero but small, telling us that the condition numbers of $T_{n}(a)$ may go to infinity quite irregularly.
If, for example $\al=\mi$, then $\|T_{n}^{-1}(a)\|$ equals $\nf$ if $n$ is odd and is asymptotically equal to $(n+1)/\pi$ if $n$ is even.

In the applications we will consider in the forthcoming sections, we encounter block Toeplitz matrices that are lower block-diagonal.
The symbols of such matrices are matrix-polynomials, i.e., their Fourier series do not contain terms $t^{k}$ with $k<0$.

\begin{theorem}\label{th:mapol}
Let $A(t)=A_{0}+A_{1}t +\cdots+A_{m}t^{m}$ $(t\in\bT)$ with matrices $A_{j}\in\bC^{N \times N}$.
Suppose $\det A_{0} \ne 0$ but $\det A$ has zeros on~$\bT$.
Then the sequence $\{\ka(T_{n}(A))\}$ is unbounded.
If $\det A$ has no zeros inside~$\bT$, then $\ka(T_{n}(A))$ increases at most polynomially, that is, there is a finite number $\nu$ such that $\ka(T_{n}(A))=O(n^\nu)$ as $n\to\nf$.
If $\det A$ has zeros inside~$\bT$, then $\ka(T_{n}(A))$ may grow at most exponentially, i.e., $\ka(T_{n}(A))=O(\e^{\al n})$ with some constant $\al>0$.
\end{theorem}

\begin{proof}
We know from Theorem~\ref{th:limsup} that $\{\ka(T_{n}(A))\}$ must be an unbounded sequence.
The norms $\|T_{n}(A)\|$ converge to $\|A\|_{\nf}$.
Hence, the growth of $\ka(T_{n}(A))$ is solely determined by the norms $\|T_{n}^{-1}(A)\|$.
Note that the assumption $\det A_{0} \ne 0$ guarantees that the matrices $T_{n}(A)$ are all invertible.

To study the growth of the norms of the inverses, we invoke the Smith canonical form of $A(t)$; see, e.g., Theorem 18.1.2 of~\cite{Pra96}.
This is a representation
\begin{equation}\label{Smith}
A(t)=G_{+}(t)\diag(p_{1}(t), \ldots, p_N(t))H_{+}(t)=:G_{+}(t)D(t)H_{+}(t)
\end{equation}
with matrix-polynomials $G_{+},H_{+}$ having constant determinant, $\det G_{+}(t)=c_{1}\ne 0$ and $\det H_{+}(t)=c_{2}\ne 0$ on~$\bT$, and with monic scalar polynomials $p_{j}$ such that $p_{j-1}|p_{j}$ for $j=2, \ldots, N$.
Since all factors in~(\ref{Smith}) are factors of the $+$ type, we have
\[
T_{n}(A)=T_{n}(G_{+})T_{n}(D)T_{n}(H_{+}).
\]
The inverse of $G_{+}(t)$ is $(1/c_{1})\,{\rm adj}\, G_{+}(t)$, which is again a matrix-polynomial.
Analogously, $H_{+}^{-1}(t)$ is a matrix-polynomial.
It follows that $T_{n}(G_{+})$ and $T_{n}(H_{+})$ are invertible for all $n\ge 1$ and that their inverses are $T_{n}(G_{+}^{-1})$ and $T_{n}(H_{+}^{-1})$.
Thus, we obtain
\[
T_{n}^{-1}(A)=T_{n}(H_{+}^{-1}) T_{n}^{-1}(D) T_{n}(G_{+}^{-1})
\]
with $\|T_{n}(H_{+}^{-1})\|\le\|H_{+}^{-1}\|_{\nf}<\nf$ and $\|T_{n}(G_{+}^{-1})\|\le\|G_{+}^{-1}\|_{\nf}<\nf$.
The diagonal entries of $T_{n}(D)$ are of the form
\[
T_{n}((t-t_{1})^{\be_{1}})\cdots T_{n}((t-t_{r})^{\be_{r}})
\]
with $\be_{1}, \ldots, \be_{r}\in\bN \cup\{0\}$.
If $t_{j}\in\bT$, then Theorem~\ref{th:RFact} implies that the norms of the matrices $T_{n}^{-1}((t-t_{j})^{\be_{j}})$ increase at most polynomially.
For $|t_{j}|>1$, the norms of $T_{n}^{-1}((t-t_{j})^{\be_{j}})$ remain bounded.
Finally, if $0<|t_{j}|<1$, then the norms of $T_{n}^{-1}((t-t_{j})^{\be_{j}})$ are easily seen to increase exponentially.
Putting all this together, we arrive at the assertion of the theorem.
\end{proof}

To cover more general settings, take an $N \times N$ matrix polynomial $S$, a scalar polynomial~$q$ without zeros on~$\bT$, and consider $A(t)=S(t)/q(t)$.
Multiplying the Smith canonical form $S=G_{+}D_{s}H_{+}$ by $1/q$ yields the factorization
\begin{equation}\label{AS}
A(t)=G_{+}(t)\diag\bigg(\frac{p_{1}(t)}{q(t)}, \ldots, \frac{p_N(t)}{q(t)}
\bigg)H_{+}(t)=:G_{+}(t)D(t)H_{+}(t)
\end{equation}
with monic scalar polynomials $p_{j}$ such that $p_{j-1}|p_{j}$ for $j=2, \ldots N$.
Let $t_{1}, \ldots, t_{r}$ be the zeros of $\det A(t)$ on~$\bT$.
We then may write
\begin{equation}\label{ASj}
\frac{p_{j}(t)}{q(t)}=(t-t_{1})^{\be_{j,1}}\cdots (t-t_{r})^{\be_{j,r}}\,t^{\ka_{j}}\,r_{j}^{-}(t)r_{j}^{+}(t)
\end{equation}
where the $\ka_{j}$ are integers and $r_{j}^{+}(t)$ and $r_{j}^{-}(t)$ are rational functions that do not vanish for $|t|\le 1$ and $1\le|t|\le\nf$, respectively.
Since $p_{j-1}|p_{j}$, we actually have
\[
\be_{1,1}\le\cdots\le\be_{N,1}, \quad \ldots\,, \quad \be_{1,r}\le\cdots\le\be_{N,r},\quad
\ka_{1}\le\cdots\le\ka_N.
\]
It follows that
\begin{equation}\label{smallest}
\be_{1,1}+\cdots+\be_{1,r}+\ka_{1} \le\cdots\le\be_{N,1}+\cdots+\be_{N,r}+\ka_N
\end{equation}
Thus, abbreviating $(t-t_{1})^{\be_{j,1}}\cdots (t-t_{r})^{\be_{j,r}}$ to $\xi_{j}(t)$, we arrive at the representation
\[
A(t)=G_{+}(t)D_{-}(t)H_{+}'(t)
\]
with
\[
D_{-}(t)=\diag (\xi_{1}(t)t^{\ka_{1}}r_{1}^{-}(t), \ldots, \xi_N(t)t^{\ka_N}r_N^{-}(t)),
\quad H_{+}'(t)=\diag (r_{1}^{+}(t), \ldots, r_N^{+}(t))H_{+}(t).
\]

\begin{theorem}\label{th:G1}
Let $A(t)$ be as just described and suppose the largest number in~\eqref{smallest} is negative, that is,
\begin{equation}\label{sm1bis}
\be_{N,1}+\cdots+\be_{N,r}+\ka_N < 0.
\end{equation}
Then the condition numbers $\ka(T_{n}(A))$ grow at least exponentially, i.e., there are positive constants $c,\ga$ such that $\ka(T_{n}(A))\ge c\e^{\ga n}$ for all $n\ge 1$.
\end{theorem}

\begin{proof}
To make the proof more transparent, we now use that block Toeplitz matrices are unitarily similar to compressions of matrix multiplications on the $\bC^N$-valued space $L^{2}_N(\bT)$.
Throughout this (and the following) proof, we let $P_{n}$ stand for the orthogonal projection of $L^{2}_N(\bT)$ or $L^{2}(\bT)$ onto the subspace $L^{2}_{N,n}(\bT)$ or $L^{2}_{n}(\bT)$ of vector or scalar polynomials of the form $X(t)=X_{0}+X_{1}t+\cdots+X_{n-1}t^{n-1}$.
Thus, $T_{n}(\Ph)$ may be thought of as the operator $X \mapsto P_{n}\Ph X$ acting on $L^{2}_{N,n}(\bT)$.

Our first goal is to show that $T_{n}(G_{+}D_{-})$ is not invertible if $n$ is large enough.
The matrix $T_{n}(D_{-})$ is diagonal and its first entry is $T_{n}(\xi_{1}(t) t^{\ka_{1}}r_{1}^{-}(t))$.
By virtue of~\eqref{sm1bis}, we have
\begin{equation}\label{sm2}
\xi_{1}(t) t^{\ka_{1}}r_{1}^{-}(t)=\sum_{j=m}^{\nf}c_{j}t^{-j}
\end{equation}
with some $m\ge 1$.
This implies that $T_{n}(G_{+}D_{-})=P_{n}G_{+}P_{n}D_{-}P_{n}+P_{n}G_{+}(I-P_{n})D_{-}P_{n}$ equals
\[
P_{n}G_{+}P_{n} D_- P_{n}+P_{n}G_{+}QD_{-}P_{n},
\]
where $Q$ is the orthogonal projection of $L^{2}_N(\bT)$ onto the subspace spanned by $\{t^{-k}\}_{k=1}^{\nf}$.
It follows that
\[
T_{n}^{-1}(G_{+})T_{n}(G_{+}D_{-})=P_{n} D_{-} P_{n}+P_{n} G_{+}^{-1} P_{n} G_{+} QD_{-} P_{n}=: A_{n}+B_{n}.
\]
From~\eqref{sm2} we infer that the highest power of $t$ in
\[
P_{n}\xi_{1}(t) t^{\ka_{1}}r_{1}^{-}(t)(x_{0}+x_{1}t+\cdots+x_{n-1}t^{n-1})
\]
is $n-m-1$, which tells us that scalar polynomials $y_{0}+y_{1}t+\cdots+y_{n-1}t^{n-1}$ do not belong to the range of the first entry of $A_{n}=T_{n}(D_{-})$ if at least one of $y_{n-1}, \ldots, y_{n-m}$ is nonzero.
The maximal degree of a vector polynomial in the range of $B_{n}=P_{n} G_{+}^{-1} P_{n} G_{+} QD_{-} P_{n}$ is $k_{1}+k_{2}-1$ where $k_{1}$ and $k_{2}$ are the degrees of the matrix polynomials $G_{+}$ and $G_{+}^{-1}$, respectively.
In summary, if $n$ is sufficiently large then the first component of the range of $T_{n}^{-1}(G_{+})T_{n}(G_{+}D_{-})$ does not contain all polynomials of degree $n-1$.

At this point we have proved that $T_{n}^{-1}(G_{+})T_{n}(G_{+}D_{-})$ and hence also $T_{n}(G_{+}D_{-})$ is not invertible.
Let $X^{(n)}\in L^{2}_{N,n}(\bT)$ be a vector function in the kernel of $P_{n}G_{+}D_{-}P_{n}$.
We obtain that
\[
0=T_{n}^{-1}(G_{+})T_{n}(G_{+}D_{-})X^{(n)}=(A_{n}+B_{n})X^{(n)}=P_{n}(D_{-}X^{(n)})+B_{n}X^{(n)},
\]
and since the degree of $B_{n}X^{(n)}$ is at most $k_{1}+k_{2}-1$, the degree of $P_{n}(D_{-}X^{(n)})$ cannot exceed $k_{1}+k_{2}-1$, too.
Let $x_{j}^{(n)}$ denote the $j$th component of $X^{(n)}$.
The $j$th component of $P_{n}(D_{-}X^{(n)})$ is $P_{n}(\xi_{j}(t)t^{\ka_{j}}r_{j}^{-}(t) x_{j}^{(n)}(t))$, and from~\eqref{smallest} and~\eqref{sm1bis} we therefore infer that the degree of $P_{n}(D_{-}X^{(n)})$ is at least $\deg X^{(n)}+d$, with
\begin{equation*}
d:=\be_{1,1}+\cdots+\be_{1,r}+\ka_{1}<0.
\end{equation*}
Consequently,
\begin{equation}\label{kk}
\deg X^{(n)}\le k_{1}+k_{2}-1-d.
\end{equation}

Let us now turn to $T_{n}(A)=T_{n}(G_{+}D_{-}H_{+}')$.
Recall that $H_{+}'(t)=R_{+}(t)H_{+}(t)$ with
\[
R_{+}(t)=\diag (r_{1}^{+}(t), \ldots,r_N^{+}(t)).
\]
We approximate $R_{+}$ by the rational matrix function $R_{+}^{\ast}:=(P_{\lfloor n/2 \rfloor} R_{+}^{-1})^{-1}$.
From Section~\ref{ssec:prelim3} we conclude that $\|R_{+}-R_{+}^{\ast}\|_{\nf}=O(\e^{-\ga n})$ with some $\ga>0$.
Letting $H_{+}^{\ast}=R_{+}^{\ast}H_{+}$, we get $\|H_{+}'-H_{+}^{\ast}\|_{\nf}=O(\e^{-\ga n})$.
It results that
\[
T_{n}(A)=T_{n}(G_{+}D_{-}H_{+}^{\ast})+T_{n}(G_{+}D_{-}(H_{+}'-H_{+}^*))=:C_{n}+O_{n}
\]
with $\|O_n\|_{\nf}=O(\e^{-\ga n})$.

We now show that $C_{n}$ is not invertible whenever $n$ is sufficiently large.
Consider the vector function
\[
J^{(n)}(t)=H_{+}^{-1}(t)(P_{\lfloor n/2 \rfloor} R_{+}^{-1})(t)X^{(n)}(t).
\]
Putting $k_{3}=\deg H_{+}^{-1}$ we see from equality~\eqref{kk} that
\[
\deg J^{(n)}\le k_{3}+\lfloor n/2 \rfloor+\deg X^{(n)}\le k_{3}+\lfloor n/2 \rfloor+k_{1}+k_{2}-1-d,
\]
which implies that $J^{(n)}$ is in $L^{2}_{N,n}(\bT)$ for all sufficiently large $n$.
We claim that $C_{n}J^{(n)}=0$.
Indeed,
\begin{eqnarray*}
C_{n}J^{(n)} &=& P_{n}(G_{+}D_{-}H_{+}^{\ast}J^{(n)})\\
&=& P_{n}(G_{+}D_{-}(P_{\lfloor n/2 \rfloor} R_{+}^{-1})^{-1} H_{+}H_{+}^{-1}(P_{\lfloor n/2 \rfloor} R_{+}^{-1}) X^{(n)})\\
&=& P_{n}(G_{+}D_{-}X^{(n)})=0.
\end{eqnarray*}
Thus, we have $T_{n}(A)=C_{n}+O_{n}$ with $\|O_{n}\|_{\nf}=O(\e^{-\ga n})$ and ${\rm rank}\,C_{n}\le nN-1$.
As the smallest singular value of $T_{n}(A)$ is the distance of $T_{n}(A)$ to the matrices of rank at most $nN-1$, we conclude that the smallest singular value of $T_{n}(A)$ is $O(\e^{-\ga n})$, which gives the assertion of the theorem.
\end{proof}

The previous theorem may be regarded as a block case version of the situation $k<-\be$ in Theorem~\ref{th:RFact}, while the following one may be viewed as a block matrix version of the case $k>0$.

\begin{theorem}\label{th:3.9}
Let $A(t)$ be as described before Theorem~\ref{th:G1}, and suppose $\ka_1 \ge 0$ and $\ka_N>0$. Then the condition numbers $\ka(T_{n}(A))$ grow at least exponentially, i.e., there are positive numbers $c,\ga$ such that $\ka(T_{n}(A))\ge c\e^{\ga n}$ for all $n\ge 1$.
\end{theorem}

\begin{proof}
We proceed as in the previous proof.
This time we reorganize the representation for $A(t)$ to $A(t)=G_{+}(t)R_{-}(t)D_{+}(t)H_{+}(t)$ where
\[
D_{+}(t)=\diag(\xi_{1}(t)t^{\ka_{1}}r_{1}^{+}(t),\ldots, \xi_N(t)t^{\ka_N}r_N^{+}(t))
\]
and
\[
R_{-}(t)=\diag(r_{1}^{-}(t),\ldots, r_N^{-}(t)).
\]
Let $R_{-}^{\ast}(t)$ be the approximation of $R_{-}(t)$ given by $R_{-}^{\ast}(t)=(P_{\lfloor n/2 \rfloor}R_{-}^{-1}(t))^{-1}$.
From Section~\ref{ssec:prelim3} we conclude that $\|R_{-}(t)-R_{-}^{\ast}(t)\|_{\nf}=O(\e^{-\ga n})$ with some $\ga>0$, which yields the splitting
\begin{equation}\label{sm3.8}
T_{n}(A)=T_{n}(G_{+}R_{-}^{\ast}D_{+}H_{+})+O_{1,n}
\end{equation}
with $\|O_{1,n}\|_{\nf}=O(\e^{-\ga n})$.

Let $Q_{n}=I-P_{n}$ on $L^{2}_{n}(\bT)$ and consider first the operator
\begin{eqnarray}\label{sm3.9}
T_{n}(R_{-}^{\ast}D_{+}H_{+})& &=P_{n}R_{-}^{\ast}P_{n}D_{+}H_{+}P_{n}+P_{n}R_{-}^{\ast}Q_{n}D_{+}H_{+}P_{n}\nonumber\\
& &=T_{n}(R_{-}^{\ast})T_{n}(D_{+})T_{n}(H_{+})+P_{n}R_{-}^{\ast}Q_{n}D_{+}H_{+}P_{n}\\
& &=:A_{n}+B_{n}\nonumber.
\end{eqnarray}
Multiplying (\ref{sm3.9}) by $T_{n}^{-1}(H_{+})=T_{n}(H_{+}^{-1})$ on the right and by $T_{n}^{-1}(R_{-}^{\ast})=T_{n}((R_{-}^{\ast})^{-1})$ on the left we obtain
\begin{eqnarray*}
I_{n} & &:=T_{n}^{-1}(R_{-}^{\ast})T_{n}(R_{-}^{\ast}D_{+}H_{+})T_{n}^{-1}(H_{+})\\
& &=P_{n}D_{+}P_{n}+P_{n}(R_{-}^{\ast})^{-1}B_{n}P_{n}H_{+}^{-1}P_{n}\\
& &=: P_{n}D_{+}P_{n}+C_{n}.
\end{eqnarray*}
We prove that the operator $I_{n}$ is not invertible.
The operator $P_{n}D_{+}P_{n}$ is diagonal and its last entry is $D_N:=T_{n}(\xi_N(t)t^{\ka_N}r_N^{+}(t))$.
Put $m=\ka_N$.
It is easy to see that (scalar) polynomials of the form $X_{n}(t)=\sum_{j=0}^{n-1}x_{j}t^{j}$ with $x_{j}\in\bC$ do not belong to the range of $D_N$ if at least one of the numbers $x_{0},x_{1},\ldots,x_{m-1}$ is nonzero.

On the other hand, the range of $C_{n}$ is contained in subspace of vector polynomials of the kind
\begin{equation}\label{sm3.10}
X_{n}(t)=\sum_{j=h}^{n-1}x_{j}t^{j},\quad x_{j}\in \bC^N,
\end{equation}
where the number $h$ satisfies
\begin{equation}\label{sm3.11}
h\ge n-1-\lfloor n/2 \rfloor.
\end{equation}
Indeed, this follows from the relations
\begin{eqnarray*}
{\rm Range}(C_{n}) &\subset & {\rm Range}(P_{n}(R_{-}^{\ast})^{-1}P_{n}R_{-}^{\ast}\mid\Im Q_{n})\\
&=& {\rm Range}(P_{n}(R_{-}^{\ast})^{-1}Q_{n}R_{-}^{\ast}\mid\Im Q_{n})\\
& \subset & {\rm Range}(P_{n}(R_{-}^{\ast})^{-1}\mid\Im Q_{n}).
\end{eqnarray*}
Thus, the range of $I_{n}$ is essentially smaller than $L_{N,n}^{2}(\bT)$.

Consequently, the operator $I_{n}$ is not invertible and there exists a function $X^{(n)}$ belonging to the kernel of $I_{n}$.
That is,
\[
P_{n}D_{+}(t)X^{(n)}(t)+C_{n}(t)X^{(n)}(t)=0.
\]
We note that the function $C_{n}(t)X^{(n)}(t)$ is a polynomial of the kind (\ref{sm3.10})-(\ref{sm3.11}).
So the vector polynomial $P_{n}D_{+}(t)X^{(n)}(t)$ also has the form (\ref{sm3.10})-(\ref{sm3.11}).
This means that
\begin{equation}\label{sm3.12}
X^{(n)}(t)=\sum_{j=h_{1}}^{n-1}x_{j}t^{j},\quad x_{j}\in\bC^N
\end{equation}
where $h_{1}\ge n-1-\lfloor n/2 \rfloor+\ka_{1}$. Here we use the hypothesis $\ka_{1}\ge 0$ to guarantee that each diagonal entry $\xi_j(t)t^{\ka_{j}}r_j^+(t)$ of $D_{+}(t)$ is a power series in nonnegative powers of $t$, so that the lowest degree occurring in $P_{n}D_{+}(t)X^{(n)}(t)$ is bounded from below by $\ka_{1}$. It is easy to see that $J^{(n)}(t):=P_{n}H_{+}^{-1}(t)X^{(n)}(t)$ belongs to the kernel of $T_{n}(R_{-}^{\ast}D_{+}H_{+})$ and that $J^{(n)}(t)$ has the form (\ref{sm3.12}), too.

We now return to (\ref{sm3.8}).
We have
\begin{eqnarray}\label{sm3.13}
(T_{n}(A)J^{(n)})(t)& &=(P_{n}G_{+}R_{-}^{\ast}D_{+}H_{+}J^{(n)})(t)+(O_{1,n}J^{(n)})(t)\nonumber\\
& &=P_{n}G_{+}P_{n}(R_{-}^{\ast}D_{+}H_{+}J^{(n)})(t)\nonumber\\
& & \quad+(P_{n}G_{+}Q R_{-}^{\ast}D_{+}H_{+}J^{(n)})(t)+(O_{1,n}J^{(n)})(t)\nonumber\\
& &=0+(P_{n}G_{+}Q(R_{-}^{\ast}-R_{-}^{**})D_{+}H_{+}J^{(n)})(t)\nonumber\\
& & \quad+(P_{n}G_{+}QR_{-}^{**}D_{+}H_{+}J^{(n)})(t)+(O_{1,n}J^{(n)})(t)\nonumber\\
& &=:(O_{2,n}J^{(n)})(t)+(O_{3,n}J^{(n)})(t)+(O_{1,n}J^{(n)})(t),
\end{eqnarray}
where $R_{-}^{**}(t)=(P_{\lfloor n/4 \rfloor}R_{-}^{\ast})(t)$.
It is easily seen that
\[
\|O_{2,n}J^{(n)}\|_{\nf}=O(\e^{-\ga n}), \quad \ga>0.
\]
Finally, $(O_{3,n} J^{(n)})(t)$
 is identically zero.
Indeed, we have
\[
R_{-}^{**}(t)D_{+}(t)H_{+}(t)J^{(n)}(t)=\sum_{j=h_{2}}^{\nf} x_{j}t^{j},
\]
where $h_{2}\ge n-1-\lfloor n/2 \rfloor-\lfloor n/4 \rfloor+\ka_{1}$.
If $n$ large enough, then $h_{2}>0$.
So we conclude that $(Q R_{-}^{**}D_{+}H_{+}J^{(n)})(t)$ is identically zero, implying that so also is $(O_{3,n} J^{(n)})(t)$.

In summary, from (\ref{sm3.13}) it results that $\|(T_{n}(A)J^{(n)})(t)\|_{\nf}=O(\e^{-\ga n})$, and the proof can be finished as the proof of Theorem~\ref{th:G1}.
\end{proof}

\section{Finite block Toeplitz band matrices}\label{sec:connection}

The present section is devoted to translating the results of the previous section into the terminology used in~\cite{AmoBru96}, since this work is one of the key tools employed in the recent numerical literature on stability features of approximation schemes for space-time PDEs~\cite{FeFr26,FeFrLoPe25,FeGo25}.

For non-negative integers $m,k$, we consider $m+k+1$ matrices $M_{-k},\ldots, M_{0},\ldots, M_{m}$ with $M_{j}\in\bC^{N\times N}$ and the $nN \times nN$ block Toeplitz band matrices as in Figure~\ref{fg:3}.
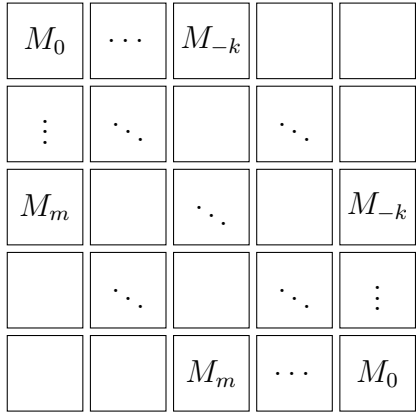
\begin{figure}[h]
\centering
\begin{tikzpicture}[scale=1]
\def\n{5}
\def\gap{0.1}
\def\size{1}
\foreach\mi in {0,...,4} {
\foreach\j in {0,...,4} {
\pgfmathsetmacro{\x}{\j*(\size+\gap)}
\pgfmathsetmacro{\y}{-\mi*(\size+\gap)}
\draw (\x,\y) rectangle++(\size,-\size);}}

\foreach\mi/\j/\testo in {
0/0/M_{0}, 0/1/\dots, 0/2/M_{-k},
1/0/\vdots, 1/1/\ddots, 1/3/\ddots,
2/0/M_{m}, 2/2/\ddots, 2/4/M_{-k},
3/1/\ddots, 3/3/\ddots, 3/4/\vdots,
4/2/M_{m}, 4/3/\dots, 4/4/M_{0}} {
\pgfmathsetmacro{\x}{\j*(\size+\gap)+0.5*\size}
\pgfmathsetmacro{\y}{-\mi*(\size+\gap)-0.5*\size}
\node at (\x,\y) {$\testo$};}
\end{tikzpicture}

\caption{Each block $M_{j}$ occupies the $j$th block diagonal, where $j<0$ corresponds to superdiagonals and $j>0$ to subdiagonals.}\label{fg:3}
\end{figure}

These matrices may be written in the form $T_{n}(t^{-k}\SR(t))$ where $\SR(t)$ is the polynomial of degree $k+m$ with coefficients in $\bC^{N\times N}$ given by
\begin{equation}\label{eq:R}
\SR(t)=\sum_{j=-k}^{m} M_{j} t^{k+j}=M_{-k}+M_{-k+1}t+\cdots+M_{m}t^{k+m}.
\end{equation}
We say that the polynomial $\SR$ is of type $(s,u,\ell)$ if $\det \SR$ has
\begin{itemize}
\item $s$ zeros with modulus smaller than $1$,
\item $u$ zeros with unit modulus,
\item $\ell$ zeros with modulus larger than $1$.
\end{itemize}
Here zeros are counted with their multiplicity.

\begin{theorem}\label{cr:R}
Let $\SR$ be the polynomial~\eqref{eq:R} and put $A(t)=t^{-k}\SR(t)$.
If $\SR$ is of type $(s,0,\ell)$ with $s\ne kN$, then there exist constants $c>0$ and $\al>0$ such that $\ka(T_{n}(A))\ge~c\e^{\al n}$.
\end{theorem}

\begin{proof}
Since $\det \SR$ has no zeros on~$\bT$, we have $A\in GR_{N\times N}(\bT)$.
Moreover, $ \det A(t)=t^{-kN}\det \SR(t), $ from which it follows that
\[
\wind(\det A)=\wind(t^{-kN})+\wind(\det \SR)=-kN+\wind(\det \SR).
\]
By the argument principle, the winding number of the polynomial $\det\SR$ is exactly the number of its roots strictly inside~$\bT$.
Therefore, $\wind(\det \SR)=s$, which leads to the equality $ \wind(\det A)=s-kN. $
According to Theorem~\ref{th:Fredholm}, the operator $T(A)$ is Fredholm and its index is given by
\[
\ind(T(A))=-\wind(\det A)=kN-s.
\]
Since $s\ne kN$ by our hypothesis, we obtain that $\ind(T(A))\ne 0$.
As a bounded linear operator must have index zero in order to be invertible, the operator $T(A)$ fails to be invertible.
Thus, because $A\in GR_{N\times N}(\bT)$ and $T(A)$ is not invertible, Theorem~\ref{th:low} applies directly.
\end{proof}

Here is what we can say in the case of matrix-polynomials, that is, for $k=0$.

\begin{theorem}\label{th:types}
Let $\SR$ be the polynomial~\eqref{eq:R} with $k=0$.
\begin{itemize}
\item If $\SR$ is of type $(s,u,\ell)$ with $u\ge 1$, then $\limsup \ka(T_{n}(\SR))=\nf$.
\item If $\SR$ is of type $(0,u,\ell)$, then $\ka(T_{n}(\SR))$ increases at most polynomially.
\item If $\SR$ is of type $(s,0,\ell)$ with $s\ge 1$, then $\ka(T_{n}(\SR))$ increases exponentially.
\end{itemize}
\end{theorem}

\begin{proof} This follows from Theorems~\ref{th:low},~\ref{th:limsup},~\ref{th:mapol},~\ref{cr:R}.
\end{proof}

\begin{remark} \label{rm:221}
Let $N\ge 2$, let $A\in R_{N \times N}(\bT)$ be of the form
\[
A(t)=t^{-1}\SR(t)=t^{-1}(M_{-1}+M_{0}t+M_{1}t^{2}+M_{2}t^{3}),
\]
and suppose $\SR$ is of type $(2,2,1)$.
This is a case we will repeatedly encounter in the applications and for which numerical evidence suggests that $\ka(T_{n}(A))$ grows at most polynomially. We are, however, not able to establish this polynomial growth at present.

Theorems~\ref{th:G1} and~\ref{th:3.9} only provide sufficient conditions for at least exponential growth, and neither of them applies here.
Indeed, write $\det \SR(t)=(t-\de_{1})(t-\de_{2})(t-t_{1})(t-t_{2})(t-\nu)$ with $|\de_{1}|<1$, $|\de_{2}|<1$, $|t_{1}|=|t_{2}|=1$, $|\nu|>1$; the polynomial $q(t)$ in the factorization~\eqref{AS} is $t$.
Since $p_{j-1}$ must divide $p_{j}$, for $t_{1}\neq t_{2}$ and $\de_{1}\neq\de_{2}$ the only possible diagonal matrix in~\eqref{AS} is
\begin{equation}\label{diff}
\diag(t^{-1}, \ldots, t^{-1}, t^{-1} (t-\de_{1})(t-\de_{2})(t-t_{1})(t-t_{2})(t-\nu)),
\end{equation}
which is~\eqref{ASj} with $\be_{j,1}=\be_{j,2}=0$, $\ka_{j}=-1$, $r_{j}^{-}(t)=r_{j}^{+}(t)=1$ for $1\le j\le N-1$.
Thus the first diagonal entry gives $\be_{1,1}+\be_{1,2}+\ka_{1}=-1<0$, while, by $p_{j-1}\mid p_{j}$, the last entry of~\eqref{diff} contains all five zeros, so that, writing $(t-\de_{i})=t(1-\de_{i}t^{-1})$, we obtain $\be_{N,1}+\be_{N,2}+\ka_{N}=3$.

Consequently, the hypothesis $\be_{N,1}+\cdots+\be_{N,r}+\ka_{N}<0$ of Theorem~\ref{th:G1} is not satisfied and since $\ka_{1}=-1<0$, the hypothesis $\ka_{1}\ge 0$ of Theorem~\ref{th:3.9} fails as well. The failure of both criteria for exponential growth is consistent with the polynomially growing behavior observed numerically, but a proof of the at most polynomial growth in this situation is beyond the techniques developed here and remains open.
\end{remark}

The scalar case $N=1$ allows a full characterization of stability.

\begin{theorem}
Let $A\in R_{1\times 1}(\bT)$ be of the form $A(t)=t^{-k}\SR(t)$ with $\SR(t)$ as in~\eqref{eq:R}.
Then the condition numbers $\ka(T_{n}(A))$ are uniformly bounded as $n\to\nf$ if and only if $\SR$ is of type $(k,0,m)$.
\end{theorem}

\begin{proof}
Let $\SR$ be of type $(s,0,\ell)$.
In the case $N=1$, Theorem~\ref{th:Fredholm} states that $T(A)$ is invertible if and only if its partial index $\ka_{1}$ is zero, which is equivalent to the equality $\wind(A)=0$.
As established in the proof of Theorem~\ref{cr:R}, $\wind(A)=s-k$.
Thus, $T(A)$ is invertible if and only if $s=k$ (and thus $\ell=m$).
Since in the scalar case, $T(\ti{A})$ is invertible if and only if $T(A)$ is invertible, Theorem~\ref{th:low} gives the assertion.
\end{proof}

\section{Stability of space-time finite element methods}\label{sec:application}

In this section, we apply the results from the previous section to study the stability of certain space-time finite element methods based on spline discretizations in time.
We first consider the wave equation and then the Schrödinger equation.

For the temporal discretization, we consider splines over the interval $[0,\TF]$ for a given $\TF>0$.
We fix some common notation.
For $N_{\mesh}\in\bN$, define $h=\TF /N_{\mesh}$ and $t_{\ell}=\ell h$ for $\ell=0,\ldots,N_{\mesh}$.
Let $p\ge 1$ be a prescribed polynomial degree and $0 \le r \le p-1$ the regularity.
We consider the knot vector composed of $p+1$ repetitions of $t_{0}$, $p-r$ repetitions of each $t_{1},\ldots, t_{N_{\mesh}-1}$, and $p+1$ repetitions of $t_{N_{\mesh}}$.
Let $\{\xi^{(p,r)}_{j}\}_{j}$ denote this knot vector, whose length is
\[
2(p+1)+(N_{\mesh}-1)(p-r)=N_{\mesh}(p-r)+r+p+2.
\]

According to the Cox--de Boor recursion formula~\cite{dB01}, the B-splines on the knot vector $\{\xi_{j}^{(p,r)}\}_{j}$ are defined recursively in $k$ as
\begin{equation}\label{eq:41}
\ph_{j}^{(k,r)}(t)=
\begin{cases}
\ds\frac{t-\xi_{j}^{(p,r)}}{\xi^{(p,r)}_{j+k}-\xi^{(p,r)}_{j}}\ph_{j}^{(k-1,r)}(t)
+\ds\frac{\xi^{(p,r)}_{j+k+1}-t}{\xi^{(p,r)}_{j+k+1}-\xi^{(p,r)}_{j+1}}\ph_{j+1}^{(k-1,r)}(t)
&\text{if } t\in [\xi_{j}^{(p,r)},\xi^{(p,r)}_{j+k+1}),\\[3ex]
0 &\text{otherwise},
\end{cases}
\end{equation}
for $j=0,\ldots,N_{\mesh}(p-r)+r$, with $\ph_{j}^{(0,r)}(t)=1$ if $t\in [\xi_{j}^{(p,r)},\xi_{j+1}^{(p,r)})$, and $\ph_{j}^{(0,r)}(t)=0$ otherwise.
The space of splines generated by $\{\ph_{j}^{(p,r)}\}_{j=0}^{N_{\mesh}(p-r)+r}$ is denoted by $S_{h}^{(p,r)}(0,\TF)$.
We also define
\begin{equation}\label{eq:42}
\begin{aligned}
S_{h,0,\bullet}^{(p,r)}(0,\TF)
&=S_{h}^{(p,r)}(0,\TF )\cap H_{0,\bullet}^{1}(0,\TF)
=\spn\{\ph_{j}^{(p,r)}\}_{j=1}^{N_{\mesh}(p-r)+r},\\[2ex]
S_{h,\bullet,0}^{(p,r)}(0,\TF)
&=S_{h}^{(p,r)}(0,\TF )\cap H_{\bullet,0}^{1}(0,\TF)
=\spn\{\ph_{j}^{(p,r)}\}_{j=0}^{N_{\mesh}(p-r)+r-1},
\end{aligned}
\end{equation}
where we set
\begin{align*}
H^{1}_{0,\bullet}(0,\TF )=\{w\in H^{1}(0,\TF ):\ w(0)=0\},\quad
H^{1}_{\bullet,0}(0,\TF )=\{w\in H^{1}(0,\TF ):\ w(\TF)=0\}.
\end{align*}

Let us define the following matrices in $\bR^{n\times n}$ with $n=N_{\mesh}(p-r)+r$:
\begin{equation}\label{eq:43}
\begin{aligned}
\mathbf{M}^{(p,r)}_{h}[\ell,j]
&=\int_{0}^{\TF}\ph_{j}^{(p,r)}(s)\,\ph_{\ell-1}^{(p,r)}(s)\,\d s,\\
\mathbf{B}^{(p,r)}_{h}[\ell,j]
&=\int_{0}^{\TF}\partial_{s}\ph_{j}^{(p,r)}(s)\,\partial_{s}\ph_{\ell-1}^{(p,r)}(s)\,\d s,\\
\mathbf{D}^{(p,r)}_{h}[\ell,j]
&=\int_{0}^{\TF}\partial_{s}^{r+1}\ph_{j}^{(p,r)}(s)\,\partial_{s}^{r+1}\ph_{\ell-1}^{(p,r)}(s)\,\d s,\\
\mathbf{C}^{(p,r)}_{h}[\ell,j]
&=\int_{0}^{\TF}\partial_{s}\ph_{j}^{(p,r)}(s)\,\ph_{\ell-1}^{(p,r)}(s)\,\d s,
\end{aligned}
\end{equation}
for $\ell,j=1,\ldots,n$.
The entries of $\mathbf{M}_{h}^{(p,r)}$ scale like $h$, those of $\mathbf{B}_{h}^{(p,r)}$ like $1/h$, and those of $\mathbf{D}_{h}^{(p,r)}$ like $1/h^{2r+1}$, whereas those of $\mathbf{C}_{h}^{(p,r)}$ do not depend on $h$. This scaling follows from the local support of the B-splines of width $\mathcal{O}(h)$, together with $\partial_{s}\sim h^{-1}$. Therefore, it is natural to define the scaled matrices
\begin{equation}\label{eq:44}
\begin{aligned}
\mathbf{M}_{n}^{(p,r)} &=\frac{1}{h}\,\mathbf{M}_{h}^{(p,r)},\\
\mathbf{B}_{n}^{(p,r)} &=h\,\mathbf{B}_{h}^{(p,r)},\\
\mathbf{D}_{n}^{(p,r)} &=h^{2r+1}\,\mathbf{D}_{h}^{(p,r)},\\
\mathbf{C}_{n}^{(p,r)} &=\mathbf{C}_{h}^{(p,r)},
\end{aligned}
\end{equation}
whose entries are independent of $h$.

\subsection{Wave equation}

Consider the linear acoustic wave equation (\ref{eq:45}).

After multiplying by a test function and integrating by parts in space and time, a space-time variational formulation of~\eqref{eq:45} reads as follows: find
\[
U\in L^{2}(0,\TF ;H^{1}_{0}(\Om))\cap H^{1}_{0,\bullet}(0,\TF ;L^{2}(\Om))
\]
such that
\begin{equation}\label{eq:46}
-(\partial_{t} U,\partial_{t} V)_{L^{2}(Q_{\TF})}
+(\nabla_{\bx} U,\nabla_{\bx} V)_{L^{2}(Q_{\TF})}
=(F, V)_{L^{2}(Q_{\TF})},
\end{equation}
for all
\[
V\in L^{2}(0,\TF ;H^{1}_{0}(\Om))\cap H^{1}_{\bullet,0}(0,\TF ;L^{2}(\Om)).
\]
Here, $(\cdot,\cdot)_{L^{2}(Q_{\TF})}$ denotes the scalar product in $L^{2}(Q_{\TF})$.

We use the standard notation
\[
H^{1}_{0}(\Om)=\{w\in H^{1}(\Om):\ w|_{\partial\Om}=0\},
\]
together with Bochner spaces.
Note that in~\eqref{eq:46} the initial condition $U(\bx,0)=0$ and the homogeneous Dirichlet boundary condition are imposed strongly in the trial space, while the condition $\partial_{t}U(\bx,0)=0$ is incorporated into the variational formulation.

By exploiting the Fourier expansion of the trial and test functions, one can show (see, e.g.,~\cite[\S 5]{StZa20} and references therein) that problem~\eqref{eq:46} admits a unique solution.
More precisely, let $\{\la_{j}\}_{j}$ be the eigenfunctions of the Laplacian operator with homogeneous Dirichlet boundary conditions in $\Om$, orthonormal in $L^{2}(\Om)$.
Any $U\in L^{2}(0,\TF ;H^{1}_{0}(\Om))\cap H^{1}_{0,\bullet}(0,\TF ;L^{2}(\Om))$ admits the representation
\[
U(\bx,t)=\sum_{j=0}^{\nf} u_{j}(t)\la_{j}(\bx)\quad\text{with}\quad u_{j}(t)=\int_{\Om} U(\bx,t)\la_{j}(\bx)\d\bx.
\]
Choosing in~\eqref{eq:46} a test function of the form $v(t)\la_{j}(\bx)$ with $v\in H^{1}_{\bullet,0}(0,\TF)$, it follows that the variational problem~\eqref{eq:46} is equivalent to finding the coefficient functions $u_{j}\in H^{1}_{0,\bullet}(0,\TF)$ such that
\[
-(\partial_{t} u_{j},\partial_{t} v)_{L^{2}(0,\TF )}+\mu_{j}(u_{j}, v)_{L^{2}(0,\TF)}=(f_{j}, v)_{L^{2}(0,\TF)}\quad\text{for all}\quad v\in H^{1}_{\bullet,0}(0,\TF),
\]
where $\{\mu_{j}\}_{j}$ is the non-decreasing, positive, and divergent sequence of eigenvalues of the Dirichlet Laplacian in $\Om$, and $f_{j}(t)=(F(\cdot,t),\la_{j})_{L^{2}(\Om)}$.
This motivates the study of the finite element discretization for the initial value problem with a parameter $\mu>0$: find $u\in H_{0,\bullet}^{1}(0,\TF )$ such that
\begin{equation}\label{eq:47}
a_{\mu}(u,v)=(f,v)_{L^{2}(0,\TF )}\quad\text{for all}\quad v\in H^{1}_{\bullet,0}(0,\TF),
\end{equation}
with $f\in L^{2}(0,\TF)$ and the bilinear form $a_{\mu}: H^{1}_{0,\bullet}(0,\TF)\times H^{1}_{\bullet,0}(0,\TF)\to\bR$ given by
\begin{equation}\label{eq:48}
a_{\mu}(u,v)=-(\partial_{t} u,\partial_{t} v)_{L^{2}(0,\TF)}+\mu (u,v)_{L^{2}(0,\TF)}\quad\text{for}\quad\mu>0.
\end{equation}
If a discretization of~\eqref{eq:48} is stable independently of both the mesh parameter characterizing the finite element subspace and $\mu$, then we expect unconditional stability for the corresponding space-time discretization of the wave problem~\eqref{eq:46}.

\begin{remark}
The linear wave equation considered here, and the Schr\"odinger equation we will embark on in Section~\ref{sec:schro}, serve as prototypes: the analysis does not depend on the specific form of the spatial operator but only on the existence of a Weyl-type asymptotics for its eigenvalues, and it therefore applies equally to other equations sharing this feature.
\end{remark}

It has been shown in~\cite{FeFr26} that a discretization of~\eqref{eq:47} with maximal regularity splines of degree $p\ge 1$ and regularity $C^{p-1}$ on a uniform mesh with mesh size $h$ is stable if and only if
\begin{equation} \label{rho}
\mu h^{2}\le 4\pi^{2}\frac{(2^{2p}-1)}{(2^{2(p+1)}-1)}\frac{\ze(2p)}{\ze(2(p+1))},
\end{equation}
where $\ze$ is the Riemann zeta function.
This mesh condition turns out to be a Courant--Friedrichs--Lewy (CFL) condition of the form $h_{t}<C_{\Om} h_{\bx}$ for the associated space-time variational formulation of the wave equation, with $C_{\Om}>0$ depending on the domain $\Om$, and $h_{t}$ and $h_{\bx}$ denoting the temporal and spatial mesh parameters, respectively.

The discrete counterpart of~\eqref{eq:47} reads as follows: find $u_{h}^{(p,r)}\in S_{h,0,\bullet}^{(p,r)}(0,\TF)$ such that
\begin{equation}\label{eq:49}
-(\partial_{t} u_{h}^{(p,r)},\partial_{t} v_{h}^{(p,r)})_{L^{2}(0,\TF)}+\mu (u_{h}^{(p,r)},v_{h}^{(p,r)})_{L^{2}(0,\TF)}=(f,v_{h}^{(p,r)})_{L^{2}(0,\TF)}
\end{equation}
for all $v_{h}^{(p,r)}\in S_{h,\bullet,0}^{(p,r)}(0,\TF )$.
Here, the discrete spaces $S_{h,0,\bullet}^{(p,r)}(0,\TF)$ and $S_{h,\bullet,0}^{(p,r)}(0,\TF)$ are defined in~\eqref{eq:42}.

The system matrix associated with~\eqref{eq:49}, with respect to the basis introduced in~\eqref{eq:41}, reads
\[
\mathbf{K}_{h,\mu}^{(p,r)}=-\mathbf{B}_{h}^{(p,r)}+\mu\mathbf{M}_{h}^{(p,r)},
\]
with $\mathbf{M}_{h}^{(p,r)}$ and $\mathbf{B}_{h}^{(p,r)}$ as in~\eqref{eq:43}.
Let us define the quantity $\rho=\mu h^{2}$.
The entries of the scaled matrix $\mathbf{K}_{n}^{(p,r)}(\rho)\in\bR^{n\times n}$ given by
\begin{equation} \label{eq:knp}
\mathbf{K}_{n}^{(p,r)}(\rho)=h\mathbf{K}_{h,\mu}^{(p,r)}=-\mathbf{B}_{n}^{(p,r)}+\rho\mathbf{M}_{n}^{(p,r)}
\end{equation}
depend on $\mu$ and $h$ only through $\rho$.
Here, $\mathbf{M}_{n}^{(p,r)}$ and $\mathbf{B}_{n}^{(p,r)}$ are defined in~\eqref{eq:44}.
We are interested in the behavior of the condition number of the family of matrices $\{\mathbf{K}_{n}^{(p,r)}(\rho)\}_{n}$ as $n$ increases, by varying $\rho$.
Employing the results obtained in Section \ref{sec:connection}, we study these behaviors for the cases $(p,r)=(2,0)$, $(p,r)=(3,0)$, and $(p,r)=(3,1)$.
\medskip

\noindent
\textbf{Case} $\boldsymbol{p=2,}$ $\boldsymbol{r=0.}$
The matrices $\mathbf{M}^{(2,0)}_{n},\mathbf{B}^{(2,0)}_{n}\in\bR^{2N_{\mesh}\times 2N_{\mesh}}$ are
\[
\mathbf{M}_{n}^{(2,0)}=\frac{1}{30}
\begin{pmatrix}
\newcommand{\cell}[1]{\makebox[1em][r]{$#1$}}
\newcommand{\lr}[1]{\multicolumn{1}{|r}{\cell{#1}}}
\newcommand{\rr}[1]{\multicolumn{1}{r|}{\cell{#1}}}
\;
\begin{array}{@{}*{10}{c}@{}}
\cline{1-2}
\lr{3} & \rr{1} &&&&&&&& \\
\lr{4} & \rr{3} &&&&&&&& \\
\cline{1-2}\cline{3-4}
\lr{3} & \rr{12} & \cell{3} & \rr{1} &&&&&& \\
\lr{0} & \rr{3} & \cell{4} & \rr{3} &&&&&& \\
\cline{1-2}\cline{3-4}\cline{5-6}
\lr{0} & \rr{1} & \cell{3} & \rr{12} & \cell{3} & \rr{1} &&&& \\
\lr{0} & \rr{0} & \cell{0} & \rr{3} & \cell{4} & \rr{3} &&&& \\
\cline{1-2}\cline{3-4}\cline{5-6}
\multicolumn{2}{c}{\ddots} & \multicolumn{2}{c}{\ddots} & \multicolumn{2}{c}{\ddots} & \multicolumn{2}{c}{\ddots} & \\
\cline{3-4}\cline{5-6}\cline{7-8}
&& \lr{0} & \cell{1} & \lr{3} & \cell{12} & \lr{3} & \rr{1} && \\
&& \lr{0} & \cell{0} & \lr{0} & \cell{3} & \lr{4} & \rr{3} && \\
\cline{3-4}\cline{5-6}\cline{7-8}\cline{9-10}
&&&& \lr{0} & \cell{1} & \lr{3} & \rr{12} & \cell{3} & \rr{1}\\
&&&& \lr{0} & \cell{0} & \lr{0} & \rr{3} & \cell{4} & \rr{3}\\
\cline{5-6}\cline{7-8}\cline{9-10}
\end{array}
\;\;
\end{pmatrix},
\]

\begin{equation}\label{eq:411}
\mathbf{B}_{n}^{(2,0)}=\frac{2}{3}
\begin{pmatrix}
\newcommand{\cell}[1]{\makebox[1em][r]{$#1$}}
\newcommand{\lr}[1]{\multicolumn{1}{|r}{\cell{#1}}}
\newcommand{\rr}[1]{\multicolumn{1}{r|}{\cell{#1}}}
\;
\begin{array}{@{}*{10}{c}@{}}
\cline{1-2}
\lr{-1} & \rr{-1} &&&&&&&& \\
\lr{2} & \rr{-1} &&&&&&&& \\
\cline{1-2}\cline{3-4}
\lr{-1} & \rr{4} & \cell{-1} & \rr{-1} &&&&&& \\
\lr{0} & \rr{-1} & \cell{2} & \rr{-1} &&&&&& \\
\cline{1-2}\cline{3-4}\cline{5-6}
\lr{0} & \rr{-1} & \cell{-1} & \rr{4} & \cell{-1} & \rr{-1} &&&& \\
\lr{0} & \rr{0} & \cell{0} & \rr{-1} & \cell{2} & \rr{-1} &&&& \\
\cline{1-2}\cline{3-4}\cline{5-6}
\multicolumn{2}{c}{\ddots} & \multicolumn{2}{c}{\ddots} & \multicolumn{2}{c}{\ddots} & \multicolumn{2}{c}{\ddots} \\
\cline{3-4}\cline{5-6}\cline{7-8}
&& \lr{0} & \cell{-1} & \lr{-1} & \cell{4} & \lr{-1} & \rr{-1} && \\
&& \lr{0} & \cell{0} & \lr{0} & \cell{-1} & \lr{2} & \rr{-1} && \\
\cline{3-4}\cline{5-6}\cline{7-8}\cline{9-10}
&&&& \lr{0} & \cell{-1} & \lr{-1} & \rr{4} & \cell{-1} & \rr{-1}\\
&&&& \lr{0} & \cell{0} & \lr{0} & \rr{-1} & \cell{2} & \rr{-1}\\
\cline{5-6}\cline{7-8}\cline{9-10}
\end{array}
\;\;
\end{pmatrix}.
\end{equation}

We associate with these matrices the three $2\times 2$ matrices that define the Toeplitz structure
\begin{align}\label{eq:412}
\mathbf{M}_{n}^{(2,0)} &\to
\overbrace{\frac{1}{30}\begin{pmatrix}[r]
0 & 1
\\ 0 & 0
\end{pmatrix}}^{2}
\quad
\overbrace{\frac{1}{30}\begin{pmatrix}[r]
3 & 12
\\ 0 & 3
\end{pmatrix}}^{1}
\quad
\overbrace{\frac{1}{30}\begin{pmatrix}[r]
3 & 1
\\ 4 & 3
\end{pmatrix}}^{0}
\\
\mathbf{B}_{n}^{(2,0)} &\to
\overbrace{\frac{2}{3}\begin{pmatrix}[r]
0 &-1
\\ 0 & 0
\end{pmatrix}}^{2}
\quad
\overbrace{\frac{2}{3}\begin{pmatrix}[r]
-1 & 4
\\ 0 &-1
\end{pmatrix}}^{1}
\quad
\overbrace{\frac{2}{3}\begin{pmatrix}[r]
-1 &-1
\\ 2 &-1
\end{pmatrix}}^{0}
\label{eq:413}
\end{align}
and similarly for $\mathbf{K}_{n}^{(2,0)}(\rho)=-\mathbf{B}_{n}^{(2,0)}+\rho\mathbf{M}_{n}^{(2,0)}$.
Note that, in the notation of Theorem~\ref{cr:R}, we have $N=2$, $m=2$, and $k=0$.
Let $\SR_{\rho}^{(2,0)}$ be the polynomial associated with the matrices $\{\mathbf{K}_{n}^{(2,0)}(\rho)\}_{n}$ as in~\eqref{eq:R}.

\begin{proposition} \label{prop:51}
The polynomial $\SR_{\rho}^{(2,0)}$ is of type
\begin{align*}
&(1,0,1)\quad\text{if and only if}\quad\rho\in (10,12)\cup (60,+\nf),\\
&(0,2,0)\quad\text{if and only if}\quad\rho\in [0,10]\cup [12,60].
\end{align*}
\end{proposition}

\begin{proof}
Recalling~\eqref{eq:412} and~\eqref{eq:413}, we explicitly compute
\[
\det \SR_{\rho}^{(2,0)}(t)=\frac{1}{180}\big(a_{\rho} t^{2}+b_{\rho} t+a_{\rho}\big).
\]
Here,
\[
a_{\rho}=\rho^{2}+16\rho+240,\quad b_{\rho}=-6\rho^{2}+208\rho-480.
\]

To determine the type of $\SR_{\rho}^{(2,0)}$, we analyze the roots of $\det \SR_{\rho}^{(2,0)}$.
The product of the two roots $t_{1},t_{2}$ satisfies $t_{1}t_{2}=1$.
Thus, the location of these roots with respect to the unit circle depends on the sign of the discriminant
\[
b_{\rho}^{2}-4a_{\rho}^{2}=32\rho(\rho-10)(\rho-12)(\rho-60).
\]
We distinguish two cases.
\smallskip

\noindent
\textbf{Case} $\boldsymbol{b_{\rho}^{2}-4a_{\rho}^{2}>0.}$
This occurs when $\rho\in (10,12)\cup (60,+\nf)$.
The quadratic has two distinct real roots.
Since their product is $1$, one root lies strictly inside the unit circle and the other strictly outside.
Thus, the polynomial is of type $(1,0,1)$.
\smallskip

\noindent
\textbf{Case} $\boldsymbol{b_{\rho}^{2}-4a_{\rho}^{2}\le 0.}$
This happens in case $\rho\in [0,10]\cup [12,60]$.
The roots are complex conjugates.
Since they are conjugates and their product is $1$, both roots lie on the unit circle.
Consequently, the polynomial is of type $(0,2,0)$.
\end{proof}

\begin{corollary} \label{cor:52}
If $\rho\in (10,12)\cup (60,+\nf)$, then the condition numbers of $\mathbf{K}_{n}^{(2,0)}(\rho)$ grow at least exponentially, and if $\rho\in [0,10]\cup [12,60]$ then the growth is at most polynomial.
\end{corollary}

\begin{proof}
This follows from Theorem~\ref{th:types} and Proposition \ref{prop:51}.
\end{proof}

In Figure \ref{fig:3}, we demonstrate numerically that Corollary \ref{cor:52} is sharp. Peaks of the condition number are obtained exactly when $\rho \in (10,12) \cup (60,+\infty)$. We remark that these results are consistent with those obtained in \cite[\S 3]{Ha24}.

\begin{figure}
    \centering
    \includegraphics[width=0.65\linewidth]{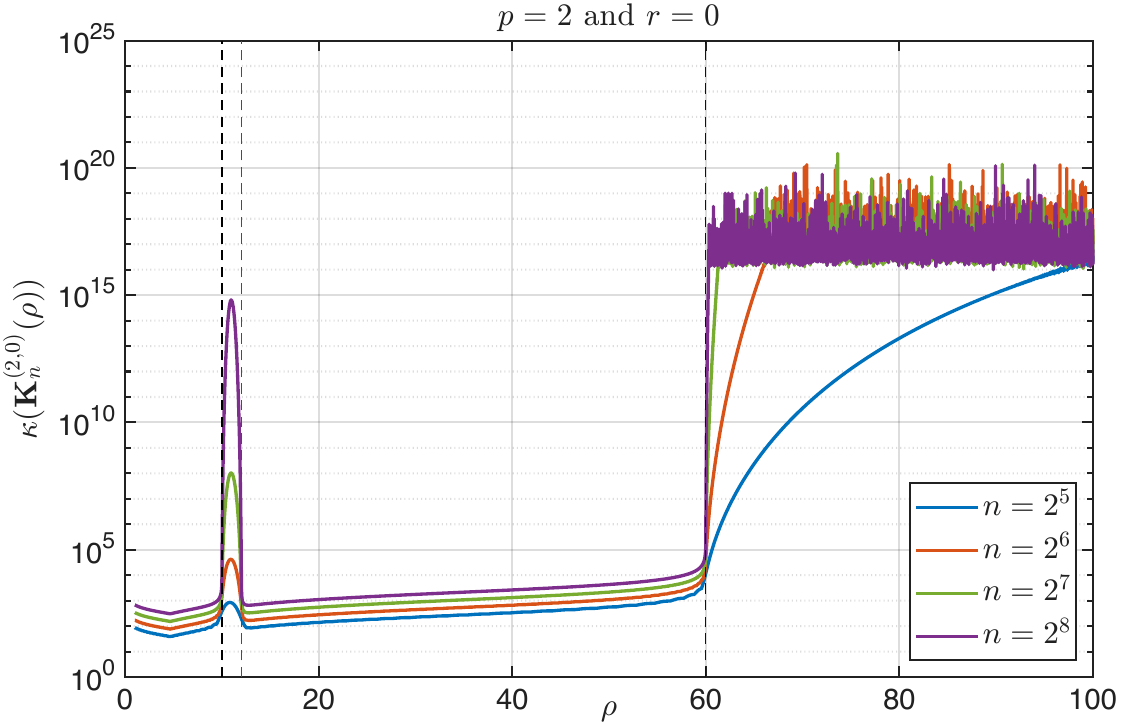}
    \caption{Condition numbers of the matrices $\mathbf{K}_n^{(2,0)}(\rho)$ defined in \eqref{eq:knp} by varying $\rho$ and~$n$. Thresholds correspond to $\rho \in \{10,12,60\}$.}
    \label{fig:3}
\end{figure}

\medskip
\noindent
\textbf{Case} $\boldsymbol{p=3,}$ $\boldsymbol{r=0.}$
The matrices $\mathbf{M}^{(3,0)}_{n},\mathbf{B}^{(3,0)}_{n}\in\bR^{3N_{\mesh}\times 3N_{\mesh}}$ are

{\small\[
\mathbf{M}_{n}^{(3,0)}=\frac{1}{140}
\begin{pmatrix}
\newcommand{\cell}[1]{\makebox[0.9em][r]{$#1$}}
\newcommand{\lr}[1]{\multicolumn{1}{|r}{\cell{#1}}}
\newcommand{\rr}[1]{\multicolumn{1}{r|}{\cell{#1}}}
\;
\begin{array}{@{}*{15}{c}@{}}
\cline{1-3}
\lr{10} & \cell{4} & \rr{1} &&&&&&&&&&&& \\
\lr{12} & \cell{9} & \rr{4} &&&&&&&&&&&& \\
\lr{9} & \cell{12} & \rr{10} &&&&&&&&&&&& \\
\cline{1-3}\cline{4-6}
\lr{4} & \cell{10} & \rr{40} & \cell{10} & \cell{4} & \rr{1} &&&&&&&&& \\
\lr{0} & \cell{0} & \rr{10} & \cell{12} & \cell{9} & \rr{4} &&&&&&&&& \\
\lr{0} & \cell{0} & \rr{4} & \cell{9} & \cell{12} & \rr{10} &&&&&&&&& \\
\cline{1-3}\cline{4-6}\cline{7-9}
\lr{0} & \cell{0} & \rr{1} & \cell{4} & \cell{10} & \rr{40} & \cell{10} & \cell{4} & \rr{1} &&&&&& \\
\lr{0} & \cell{0} & \rr{0} & \cell{0} & \cell{0} & \rr{10} & \cell{12} & \cell{9} & \rr{4} &&&&&& \\
\lr{0} & \cell{0} & \rr{0} & \cell{0} & \cell{0} & \rr{4} & \cell{9} & \cell{12} & \rr{10} &&&&&& \\
\cline{1-3}\cline{4-6}\cline{7-9}
\multicolumn{3}{c}{\ddots} & \multicolumn{3}{c}{\ddots} & \multicolumn{3}{c}{\ddots} & \multicolumn{3}{c}{\ddots} & \\
\cline{4-6}\cline{7-9}\cline{10-12}
&&& \lr{0} & \cell{0} & \cell{1} & \lr{4} & \cell{10} & \cell{40} & \lr{10} & \cell{4} & \rr{1} &&& \\
&&& \lr{0} & \cell{0} & \cell{0} & \lr{0} & \cell{0} & \cell{10} & \lr{12} & \cell{9} & \rr{4} &&& \\
&&& \lr{0} & \cell{0} & \cell{0} & \lr{0} & \cell{0} & \cell{4} & \lr{9} & \cell{12} & \rr{10} &&& \\
\cline{4-6}\cline{7-9}\cline{10-12}\cline{13-15}
&&&&&& \lr{0} & \cell{0} & \cell{1} & \lr{4} & \cell{10} & \rr{40} & \cell{10} & \cell{4} & \rr{1}\\
&&&&&& \lr{0} & \cell{0} & \cell{0} & \lr{0} & \cell{0} & \rr{10} & \cell{12} & \cell{9} & \rr{4}\\
&&&&&& \lr{0} & \cell{0} & \cell{0} & \lr{0} & \cell{0} & \rr{4} & \cell{9} & \cell{12} & \rr{10}\\
\cline{7-9}\cline{10-12}\cline{13-15}
\end{array}
\;\;
\end{pmatrix},
\]}

{\small\begin{equation}\label{eq:415}
\mathbf{B}_{n}^{(3,0)}=\frac{3}{10}
\begin{pmatrix}
\newcommand{\cell}[1]{\makebox[0.9em][r]{$#1$}}
\newcommand{\lr}[1]{\multicolumn{1}{|r}{\cell{#1}}}
\newcommand{\rr}[1]{\multicolumn{1}{r|}{\cell{#1}}}
\;
\begin{array}{@{}*{15}{c}@{}}
\cline{1-3}
\lr{-3} & \cell{-2} & \rr{-1} &&&&&&&&&&&& \\
\lr{4} & \cell{1} & \rr{-2} &&&&&&&&&&&& \\
\lr{1} & \cell{4} & \rr{-3} &&&&&&&&&&&& \\
\cline{1-3}\cline{4-6}
\lr{-2} & \cell{-3} & \rr{12} & \cell{-3} & \cell{-2} & \rr{-1} &&&&&&&&& \\
\lr{0} & \cell{0} & \rr{-3} & \cell{4} & \cell{1} & \rr{-2} &&&&&&&&& \\
\lr{0} & \cell{0} & \rr{-2} & \cell{1} & \cell{4} & \rr{-3} &&&&&&&&& \\
\cline{1-3}\cline{4-6}\cline{7-9}
\lr{0} & \cell{0} & \rr{-1} & \cell{-2} & \cell{-3} & \rr{12} & \cell{-3} & \cell{-2} & \rr{-1} &&&&&& \\
\lr{0} & \cell{0} & \rr{0} & \cell{0} & \cell{0} & \rr{-3} & \cell{4} & \cell{1} & \rr{-2} &&&&&& \\
\lr{0} & \cell{0} & \rr{0} & \cell{0} & \cell{0} & \rr{-2} & \cell{1} & \cell{4} & \rr{-3} &&&&&& \\
\cline{1-3}\cline{4-6}\cline{7-9}
\multicolumn{3}{c}{\ddots} & \multicolumn{3}{c}{\ddots} & \multicolumn{3}{c}{\ddots} & \multicolumn{3}{c}{\ddots} & \\
\cline{4-6}\cline{7-9}\cline{10-12}
&&& \lr{0} & \cell{0} & \cell{-1} & \lr{-2} & \cell{-3} & \cell{12} & \lr{-3} & \cell{-2} & \rr{-1} &&& \\
&&& \lr{0} & \cell{0} & \cell{0} & \lr{0} & \cell{0} & \cell{-3} & \lr{4} & \cell{1} & \rr{-2} &&& \\
&&& \lr{0} & \cell{0} & \cell{0} & \lr{0} & \cell{0} & \cell{-2} & \lr{1} & \cell{4} & \rr{-3} &&& \\
\cline{4-6}\cline{7-9}\cline{10-12}\cline{13-15}
&&&&&& \lr{0} & \cell{0} & \cell{-1} & \lr{-2} & \cell{-3} & \rr{12} & \cell{-3} & \cell{-2} & \rr{-1}\\
&&&&&& \lr{0} & \cell{0} & \cell{0} & \lr{0} & \cell{0} & \rr{-3} & \cell{4} & \cell{1} & \rr{-2}\\
&&&&&& \lr{0} & \cell{0} & \cell{0} & \lr{0} & \cell{0} & \rr{-2} & \cell{1} & \cell{4} & \rr{-3}\\
\cline{7-9}\cline{10-12}\cline{13-15}
\end{array}
\;\;
\end{pmatrix}.
\end{equation}}

We associate with these matrices the three $3\times 3$ matrices that define the Toeplitz structure
\begin{align}\label{eq:416}
\mathbf{M}_{n}^{(3,0)} &\to
\overbrace{\frac{1}{140}\begin{pmatrix}[r]
0 & 0 & 1
\\ 0 & 0 & 0
\\ 0 & 0 & 0
\end{pmatrix}}^{2}
\overbrace{\frac{1}{140}\begin{pmatrix}[r]
4 & 10 & 40
\\ 0 & 0 & 10
\\ 0 & 0 & 4
\end{pmatrix}}^{1}
\quad
\overbrace{\frac{1}{140}\begin{pmatrix}[r]
10 & 4 & 1
\\ 12 & 9 & 4
\\ 9 & 12 & 10
\end{pmatrix}}^{0}
\\
\mathbf{B}_{n}^{(3,0)} &\to
\overbrace{\frac{3}{10}\begin{pmatrix}[r]
0 & 0 &-1
\\ 0 & 0 & 0
\\ 0 & 0 & 0
\end{pmatrix}}^{2}
\quad
\overbrace{\frac{3}{10}\begin{pmatrix}[r]
-2 &-3 & 12
\\ 0 & 0 &-3
\\ 0 & 0 &-2
\end{pmatrix}}^{1}
\quad
\overbrace{\frac{3}{10}\begin{pmatrix}[r]
-3 &-2 &-1
\\ 4 & 1 &-2
\\ 1 & 4 &-3
\end{pmatrix}}^{0}
\label{eq:417}
\end{align}
and similarly for $\mathbf{K}_{n}^{(3,0)}(\rho)=-\mathbf{B}_{n}^{(3,0)}+\rho\mathbf{M}_{n}^{(3,0)}$.
Here, in the notation of Theorem~\ref{cr:R}, we have $N=3$, $m=3$, and $k=0$.

\begin{proposition} \label{prop:53}
The polynomial $\SR_{\rho}^{(3,0)}$ is of type
\begin{align*}
&(1,0,1)\quad\text{if and only if}\quad\rho\in (90-2\sqrt{1605},10)\cup (42,60)\cup (90+2\sqrt{1605},+\nf),\\
&(0,2,0)\quad\text{if and only if}\quad\rho\in [0,90-2\sqrt{1605}]\cup [10,42]\cup [60,90+2\sqrt{1605}].
\end{align*}
\end{proposition}

\begin{proof}
We explicitly compute, using~\eqref{eq:416} and~\eqref{eq:417},
\[
\det \SR_{\rho}^{(3,0)}(t)=\frac{3}{56000}\big(a_{\rho} t^{2}+b_{\rho} t+a_{\rho}\big).
\]
Here,
\[
a_{\rho}=\rho^{3}+30\rho^{2}+1080\rho+25200,\quad b_{\rho}=8\rho^{3}-1080\rho^{2}+23040\rho-50400.
\]
The product of the two roots is equal to $1$.
The location of these roots with respect to the unit circle depends on the sign of the discriminant
\[
b_{\rho}^{2}-4a_{\rho}^{2}=60\rho(\rho-10)(\rho-42)(\rho-60)(\rho-90+2\sqrt{1605})(\rho-90-2\sqrt{1605}).
\]
Noting that $90-2\sqrt{1605}<10$ and $60<90+2\sqrt{1605}$, an analysis of the sign of $b_{\rho}^{2}-4a_{\rho}^{2}$ yields two distinct cases.
\smallskip

\noindent
\textbf{Case} $\boldsymbol{b_{\rho}^{2}-4a_{\rho}^{2}>0.}$
This occurs when the number of negative factors is even, which corresponds to the intervals $\rho\in (90-2\sqrt{1605},10)\cup (42,60)\cup (90+2\sqrt{1605},+\nf)$.
The quadratic has two distinct real roots.
Since their product is $1$, one root lies strictly inside the unit circle and the other strictly outside.
Thus, $\SR_{\rho}^{(3,0)}$ is of type $(1,0,1)$.
\smallskip

\noindent
\textbf{Case} $\boldsymbol{b_{\rho}^{2}-4a_{\rho}^{2}\le 0.}$
This is the case for $\rho\in [0,90-2\sqrt{1605}]\cup [10,42]\cup [60,90+2\sqrt{1605}]$.
The roots are complex conjugates.
As their product is $1$, both lie on the unit circle and hence $\SR_{\rho}^{(3,0)}$ is of type $(0,2,0)$.
\end{proof}

\begin{corollary} \label{cor:54}
If
\[
\rho\in (90-2\sqrt{1605},10)\cup (42,60)\cup (90+2\sqrt{1605},+\nf),
\]
then the condition numbers of $\mathbf{K}_{n}^{(3,0)}(\rho)$ grow at least exponentially, whereas for
\[
\rho\in (0,90-2\sqrt{1605}]\cup [10,42]\cup [60,90+2\sqrt{1605}]
\]
their growth is at most polynomial.
\end{corollary}

\begin{proof}
This is again immediate from Proposition \ref{prop:53} and Theorem~\ref{th:types}.
\end{proof}

In Figure \ref{fig:4}, we show the numerically computed condition numbers of $\mathbf{K}_n^{(3,0)}(\rho)$ for various $\rho$ and $n$, demonstrating that Corollary \ref{cor:54} is sharp. Peaks of the condition number are obtained close to $\rho \approx 10$, for $\rho \in (42,60)$ and for $\rho > 90 + 2\sqrt{1605} \approx 170$.

\begin{figure}
    \centering
        \centering
        \includegraphics[width=0.7\linewidth]{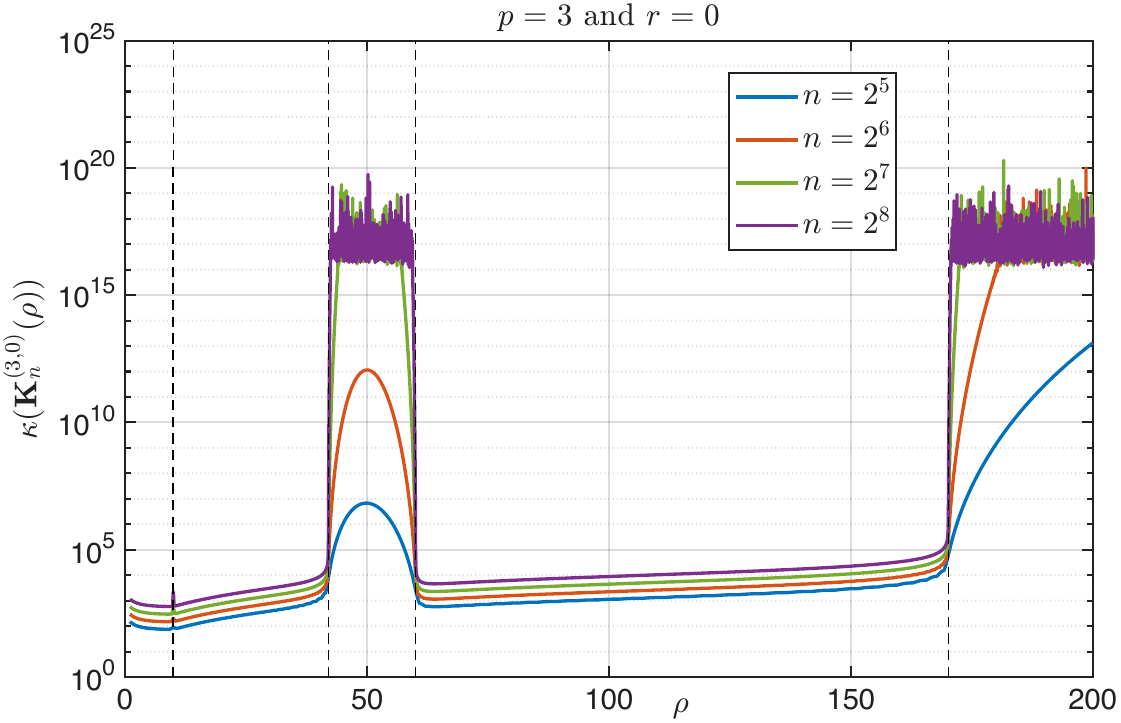}
    \caption{Condition numbers of the matrices $\mathbf{K}_n^{(3,0)}(\rho)$ defined in \eqref{eq:knp} by varying $\rho$ and~$n$. Thresholds correspond to $\rho \in \{90-2\sqrt{1605}, 10, 42, 60, 90+2\sqrt{1605}\}$.}
    \label{fig:4}
\end{figure}

\medskip
\noindent
\textbf{Case} $\boldsymbol{p=3,}$ $\boldsymbol{r=1.}$
The matrices $\mathbf{M}_{n}^{(3,1)},\mathbf{B}_{n}^{(3,1)}\in\bR^{(2N_{\mesh}+1)\times (2N_{\mesh}+1)}$ are

\begin{equation}\label{eq:418}
\mathbf{M}_{n}^{(3,1)}=\frac{1}{560}
\begin{pmatrix}
\newcommand{\cell}[1]{\makebox[1.4em][r]{$#1$}}
\newcommand{\lr}[1]{\multicolumn{1}{|r}{\cell{#1}}}
\newcommand{\rr}[1]{\multicolumn{1}{r|}{\cell{#1}}}
\;
\begin{array}{@{}*{11}{c}@{}}
\cline{1-2}\cline{3-4}
\lr{40} & \cell{18} & \lr{2} & \rr{0} &&&&&&& \\
\lr{48} & \cell{44} & \lr{8} & \rr{0} &&&&&&& \\
\cline{1-2}\cline{3-4}\cline{5-6}
\lr{44} & \cell{128} & \lr{80} & \rr{9} & \cell{1} & \rr{0} &&&&& \\
\lr{8} & \cell{80} & \lr{128} & \rr{53} & \cell{9} & \rr{0} &&&&& \\
\cline{1-2}\cline{3-4}\cline{5-6}\cline{7-8}
\lr{0} & \cell{9} & \lr{53} & \rr{128} & \cell{80} & \rr{9} & \cell{1} & \rr{0} &&& \\
\lr{0} & \cell{1} & \lr{9} & \rr{80} & \cell{128} & \rr{53} & \cell{9} & \rr{0} &&& \\
\cline{1-2}\cline{3-4}\cline{5-6}\cline{7-8}
\multicolumn{2}{c}{\ddots} & \multicolumn{2}{c}{\ddots} & \multicolumn{2}{c}{\ddots} & \multicolumn{2}{c}{\ddots} & \multicolumn{2}{c}{\ddots} & \\
\cline{3-4}\cline{5-6}\cline{7-8}\cline{9-10}
&& \lr{0} & \cell{9} & \lr{53} & \cell{128} & \lr{80} & \cell{9} & \lr{1} & \rr{0} & \\
&& \lr{0} & \cell{1} & \lr{9} & \cell{80} & \lr{128} & \cell{53} & \lr{9} & \rr{0} & \\
\cline{3-4}\cline{5-6}\cline{7-8}\cline{9-10}
&&&& \lr{0} & \cell{9} & \lr{53} & \cell{128} & \lr{80} & \rr{8} & \cell{2} \\
&&&& \lr{0} & \cell{1} & \lr{9} & \cell{80} & \lr{128} & \rr{44} & \cell{18} \\
\cline{5-6}\cline{7-8}\cline{9-10}
&&&&&& \cell{0} & \cell{8} & \cell{44} & \cell{48} & \cell{40} \\
\end{array}
\;\;
\end{pmatrix},
\end{equation}

\begin{equation}\label{eq:419}
\mathbf{B}_{n}^{(3,1)}=\frac{3}{40}
\begin{pmatrix}
\newcommand{\cell}[1]{\makebox[1.4em][r]{$#1$}}
\newcommand{\lr}[1]{\multicolumn{1}{|r}{\cell{#1}}}
\newcommand{\rr}[1]{\multicolumn{1}{r|}{\cell{#1}}}
\;
\begin{array}{@{}*{11}{c}@{}}
\cline{1-2}\cline{3-4}
\lr{-12} & \cell{-10} & \lr{-2} & \rr{0} &&&&&&& \\
\lr{16} & \cell{0} & \lr{-4} & \rr{0} &&&&&&& \\
\cline{1-2}\cline{3-4}\cline{5-6}
\lr{0} & \cell{16} & \lr{0} & \rr{-5} & \cell{-1} & \rr{0} &&&&& \\
\lr{-4} & \cell{0} & \lr{16} & \rr{-5} & \cell{-5} & \rr{0} &&&&& \\
\cline{1-2}\cline{3-4}\cline{5-6}\cline{7-8}
\lr{0} & \cell{-5} & \lr{-5} & \rr{16} & \cell{0} & \rr{-5} & \cell{-1} & \rr{0} &&& \\
\lr{0} & \cell{-1} & \lr{-5} & \rr{0} & \cell{16} & \rr{-5} & \cell{-5} & \rr{0} &&& \\
\cline{1-2}\cline{3-4}\cline{5-6}\cline{7-8}
\multicolumn{2}{c}{\ddots} & \multicolumn{2}{c}{\ddots} & \multicolumn{2}{c}{\ddots} & \multicolumn{2}{c}{\ddots} & \multicolumn{2}{c}{\ddots} & \\
\cline{3-4}\cline{5-6}\cline{7-8}\cline{9-10}
&& \lr{0} & \cell{-5} & \lr{-5} & \cell{16} & \lr{0} & \cell{-5} & \lr{-1} & \rr{0} & \\
&& \lr{0} & \cell{-1} & \lr{-5} & \cell{0} & \lr{16} & \cell{-5} & \lr{-5} & \rr{0} & \\
\cline{3-4}\cline{5-6}\cline{7-8}\cline{9-10}
&&&& \lr{0} & \cell{-5} & \lr{-5} & \cell{16} & \lr{0} & \rr{-4} & \cell{-2} \\
&&&& \lr{0} & \cell{-1} & \lr{-5} & \cell{0} & \lr{16} & \rr{0} & \cell{-10} \\
\cline{5-6}\cline{7-8}\cline{9-10}
&&&&&& \cell{0} & \cell{-4} & \cell{0} & \cell{16} & \cell{-12} \\
\end{array}
\;\;
\end{pmatrix}.
\end{equation}

Note that these matrices do not strictly fit the framework of Theorem~\ref{cr:R} due to perturbations in the top-left and bottom-right corners, as well as an additional spurious row and column.
Nevertheless, we can analyze their pure block Toeplitz band extensions.
We consider the matrices $\widetilde{\mathbf{M}}_{n}^{(3,1)},\widetilde{\mathbf{B}}_{n}^{(3,1)}\in\bR^{2N_{\mesh}\times 2N_{\mesh}}$
\begin{equation}\label{eq:420}
\widetilde{\mathbf{M}}_{n}^{(3,1)}=\frac{1}{560}
\begin{pmatrix}
\newcommand{\cell}[1]{\makebox[1.4em][r]{$#1$}}
\newcommand{\lr}[1]{\multicolumn{1}{|r}{\cell{#1}}}
\newcommand{\rr}[1]{\multicolumn{1}{r|}{\cell{#1}}}
\;
\begin{array}{@{}*{10}{c}@{}}
\cline{1-2}\cline{3-4}
\lr{80} & \cell{9} & \lr{1} & \rr{0} &&&&&& \\
\lr{128} & \cell{53} & \lr{9} & \rr{0} &&&&&& \\
\cline{1-2}\cline{3-4}\cline{5-6}
\lr{53} & \cell{128} & \lr{80} & \rr{9} & \cell{1} & \rr{0} &&&& \\
\lr{9} & \cell{80} & \lr{128} & \rr{53} & \cell{9} & \rr{0} &&&& \\
\cline{1-2}\cline{3-4}\cline{5-6}\cline{7-8}
\lr{0} & \cell{9} & \lr{53} & \rr{128} & \cell{80} & \rr{9} & \cell{1} & \rr{0} && \\
\lr{0} & \cell{1} & \lr{9} & \rr{80} & \cell{128} & \rr{53} & \cell{9} & \rr{0} && \\
\cline{1-2}\cline{3-4}\cline{5-6}\cline{7-8}
\multicolumn{2}{c}{\ddots} & \multicolumn{2}{c}{\ddots} & \multicolumn{2}{c}{\ddots} & \multicolumn{2}{c}{\ddots} & \multicolumn{2}{c}{\ddots}\\
\cline{3-4}\cline{5-6}\cline{7-8}\cline{9-10}
&& \lr{0} & \cell{9} & \lr{53} & \cell{128} & \lr{80} & \cell{9} & \lr{1} & \rr{0}\\
&& \lr{0} & \cell{1} & \lr{9} & \cell{80} & \lr{128} & \cell{53} & \lr{9} & \rr{0}\\
\cline{3-4}\cline{5-6}\cline{7-8}\cline{9-10}
&&&& \lr{0} & \cell{9} & \lr{53} & \cell{128} & \lr{80} & \rr{9}\\
&&&& \lr{0} & \cell{1} & \lr{9} & \cell{80} & \lr{128} & \rr{53}\\
\cline{5-6}\cline{7-8}\cline{9-10}
\end{array}
\;\;
\end{pmatrix},
\end{equation}

\begin{equation}\label{eq:421}
\widetilde{\mathbf{B}}_{n}^{(3,1)}=\frac{3}{40}
\begin{pmatrix}
\newcommand{\cell}[1]{\makebox[1.4em][r]{$#1$}}
\newcommand{\lr}[1]{\multicolumn{1}{|r}{\cell{#1}}}
\newcommand{\rr}[1]{\multicolumn{1}{r|}{\cell{#1}}}
\;
\begin{array}{@{}*{11}{c}@{}}
\cline{1-2}\cline{3-4}
\lr{0} & \cell{-5} & \lr{-1} & \rr{0} &&&&&&& \\
\lr{16} & \cell{-5} & \lr{-5} & \rr{0} &&&&&&& \\
\cline{1-2}\cline{3-4}\cline{5-6}
\lr{-5} & \cell{16} & \lr{0} & \rr{-5} & \cell{-1} & \rr{0} &&&&& \\
\lr{-5} & \cell{0} & \lr{16} & \rr{-5} & \cell{-5} & \rr{0} &&&&& \\
\cline{1-2}\cline{3-4}\cline{5-6}\cline{7-8}
\lr{0} & \cell{-5} & \lr{-5} & \rr{16} & \cell{0} & \rr{-5} & \cell{-1} & \rr{0} &&& \\
\lr{0} & \cell{-1} & \lr{-5} & \rr{0} & \cell{16} & \rr{-5} & \cell{-5} & \rr{0} &&& \\
\cline{1-2}\cline{3-4}\cline{5-6}\cline{7-8}
\multicolumn{2}{c}{\ddots} & \multicolumn{2}{c}{\ddots} & \multicolumn{2}{c}{\ddots} & \multicolumn{2}{c}{\ddots} & \multicolumn{2}{c}{\ddots} & \\
\cline{3-4}\cline{5-6}\cline{7-8}\cline{9-10}
&& \lr{0} & \cell{-5} & \lr{-5} & \cell{16} & \lr{0} & \cell{-5} & \lr{-1} & \rr{0} & \\
&& \lr{0} & \cell{-1} & \lr{-5} & \cell{0} & \lr{16} & \cell{-5} & \lr{-5} & \rr{0} & \\
\cline{3-4}\cline{5-6}\cline{7-8}\cline{9-10}
&&&& \lr{0} & \cell{-5} & \lr{-5} & \cell{16} & \lr{0} & \rr{-5} & \\
&&&& \lr{0} & \cell{-1} & \lr{-5} & \cell{0} & \lr{16} & \rr{-5} & \\
\cline{5-6}\cline{7-8}\cline{9-10}
\end{array}
\;\;
\end{pmatrix}.
\end{equation}

We associate with these matrices the four $2\times 2$ matrices that define the Toeplitz structure
\begin{align}\label{eq:422}
\widetilde{\mathbf{M}}_{n}^{(3,1)} &\to
\overbrace{\frac{1}{560}\begin{pmatrix}[r]
0 & 9
\\ 0 & 1
\end{pmatrix}}^{2}
\quad
\overbrace{\frac{1}{560}\begin{pmatrix}[r]
53 & 128
\\ 9 & 80
\end{pmatrix}}^{1}
\quad
\overbrace{\frac{1}{560}\begin{pmatrix}[r]
80 & 9
\\ 128 & 53
\end{pmatrix}}^{0}
\quad
\overbrace{\frac{1}{560}\begin{pmatrix}[r]
1 & 0
\\ 9 & 0
\end{pmatrix}}^{-1}
\\
\widetilde{\mathbf{B}}_{n}^{(3,1)} &\to
\overbrace{\frac{3}{40}\begin{pmatrix}[r]
0 &-5
\\ 0 &-1
\end{pmatrix}}^{2}
\quad
\overbrace{\frac{3}{40}\begin{pmatrix}[r]
-5 & 16
\\-5 & 0
\end{pmatrix}}^{1}
\quad
\overbrace{\frac{3}{40}\begin{pmatrix}[r]
0 &-5
\\ 16 &-5
\end{pmatrix}}^{0}
\quad
\overbrace{\frac{3}{40}\begin{pmatrix}[r]
-1 & 0
\\-5 & 0
\end{pmatrix}}^{-1}
\label{eq:423}
\end{align}
and similarly for
\begin{equation} \label{eq:new1}
    \widetilde{\mathbf{K}}_{n}^{(3,1)}(\rho):=-\widetilde{\mathbf{B}}_{n}^{(3,1)}+\rho\widetilde{\mathbf{M}}_{n}^{(3,1)}.
\end{equation}
In the notation of Theorem~\ref{cr:R}, we have $N=2$, $m=2$, and $k=1$.

\begin{proposition}\label{prop:43}
The polynomial $\SR_{\rho}^{(3,1)}$ is of type
\begin{align*}
&(3,0,2)\quad\text{if and only if}\quad\rho\in (168/17,10)\cup (42,+\nf),\\
&(2,2,1)\quad\text{if and only if}\quad\rho\in [0,168/17]\cup [10,42].
\end{align*}
\end{proposition}

\begin{proof}
We compute explicitly, using~\eqref{eq:422} and~\eqref{eq:423},
\[
\det \SR_{\rho}^{(3,1)}(t)=\frac{1}{11200}\, t \big(a_{\rho} t^{4}+b_{\rho} t^{3}+c_{\rho} t^{2}+b_{\rho} t+a_{\rho}\big).
\]
Here,
\[
a_{\rho}=-\rho^{2}-48\rho-1260,\quad
b_{\rho}=72\rho^{2}-768\rho+10080,\quad
c_{\rho}=-262\rho^{2}+6672\rho-17640.
\]
We analyze the roots of the reciprocal polynomial
\[
p_{\rho}(t)=a_{\rho} t^{4}+b_{\rho} t^{3}+c_{\rho} t^{2}+b_{\rho} t+a_{\rho}.
\]
We apply the substitution $y=t+t^{-1}$, which yields $t^{2}+t^{-2}=y^{2}-2$.
This reduces the quartic equation to the quadratic equation
\[
f_{\rho}(y)=a_{\rho} y^{2}+b_{\rho} y+(c_{\rho}-2a_{\rho})=0.
\]
Observe that if $y$ is real and $|y|>2$, then the equation $t^{2}-y t+1=0$ has two real roots, one strictly inside and one strictly outside the unit circle.

Conversely, if $y$ is real and $|y|\le 2$, the roots $t$ are complex conjugates lying on the unit circle.
Hence, there is a direct correspondence between pairs of roots of $\det \SR_{\rho}^{(3,1)}$ on the unit circle and roots of $f_{\rho}$ in the interval $[-2,2]$.

We now analyze $f_{\rho}(y)$.
First, note that the leading coefficient satisfies $a_{\rho}<0$ for all $\rho>0$. Moreover, $f_{\rho}(4)=12(\rho^{2}+244\rho+420)>0$ for all $\rho>0$, so that, since $a_{\rho}<0$, the parabola $f_{\rho}$ has two distinct real roots $y_{1}\le y_{2}$, with $y_{2}>4>2$.
Next, we show that $-b_{\rho}/(2a_{\rho})>2$ for all $\rho>0$, which is equivalent to
\[
4a_{\rho}+b_{\rho}=68\rho^{2}-960\rho+5040>0.
\]
The discriminant of this quadratic is $-449280<0$, so the inequality holds for all $\rho>0$.
Therefore, the critical point $-b_{\rho}/(2a_{\rho})$ of $f_{\rho}$ lies strictly to the right of $y=2$.

Since the parabola opens downward and its maximum is attained at $-b_{\rho}/(2a_{\rho})>2$, both roots $y_{1},y_{2}$ are strictly greater than $2$ if and only if $f_{\rho}(2)<0$.
We compute
\[
f_{\rho}(2)=-120\rho(\rho-42).
\]
Thus, for $\rho\in (42,+\nf)$, the polynomial $f_{\rho}$ has two real roots greater than $2$, and hence $\SR_{\rho}^{(3,1)}$ is of type $(3,0,2)$.

We now consider the case $\rho\le 42$.
In this range, $f_{\rho}(2)\ge 0$, so that $y_{1}\le 2\le y_{2}$.

To locate the root $y_{1}$, we evaluate at $y=-2$:
\begin{align*}
f_{\rho}(-2)=-24(\rho-10)(17\rho-168).
\end{align*}
For $\rho\in [0,168/17]\cup [10,42]$, we have $f_{\rho}(-2)\le 0$, and hence $y_{1}\in [-2,2]$.
Therefore, $\SR_{\rho}^{(3,1)}$ is of type $(2,2,1)$.

For $\rho\in (168/17,10)$, we have $f_{\rho}(-2)>0$.
Since also $f_{\rho}(2)>0$ and the parabola opens downward, both roots lie outside the interval $[-2,2]$.
Thus, $\SR_{\rho}^{(3,1)}$ is of type $(3,0,2)$.
\end{proof}

\begin{corollary} \label{cor:57}
If
\begin{equation*}
	\rho \in (168/17,10) \cup (42,+\infty),
\end{equation*}
then the condition numbers of $\widetilde{\mathbf{K}}_{n}^{(3,1)}(\rho)$ grow at least exponentially.
\end{corollary}

\begin{proof}
This results from Theorem~\ref{cr:R} and Proposition \ref{prop:43}.
\end{proof}

\begin{figure}
    \centering
    \begin{subfigure}{0.48\linewidth}
        \centering
        \includegraphics[width=\linewidth]{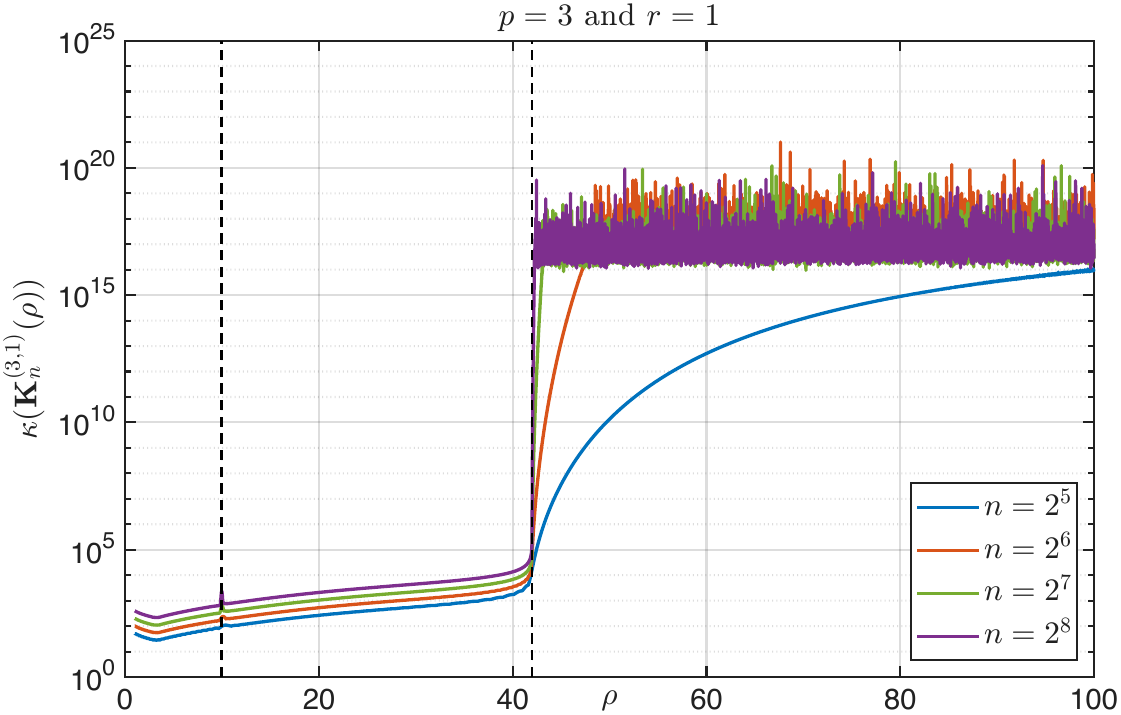}
    \end{subfigure}
    \hfill
    \begin{subfigure}{0.48\linewidth}
        \centering
        \includegraphics[width=\linewidth]{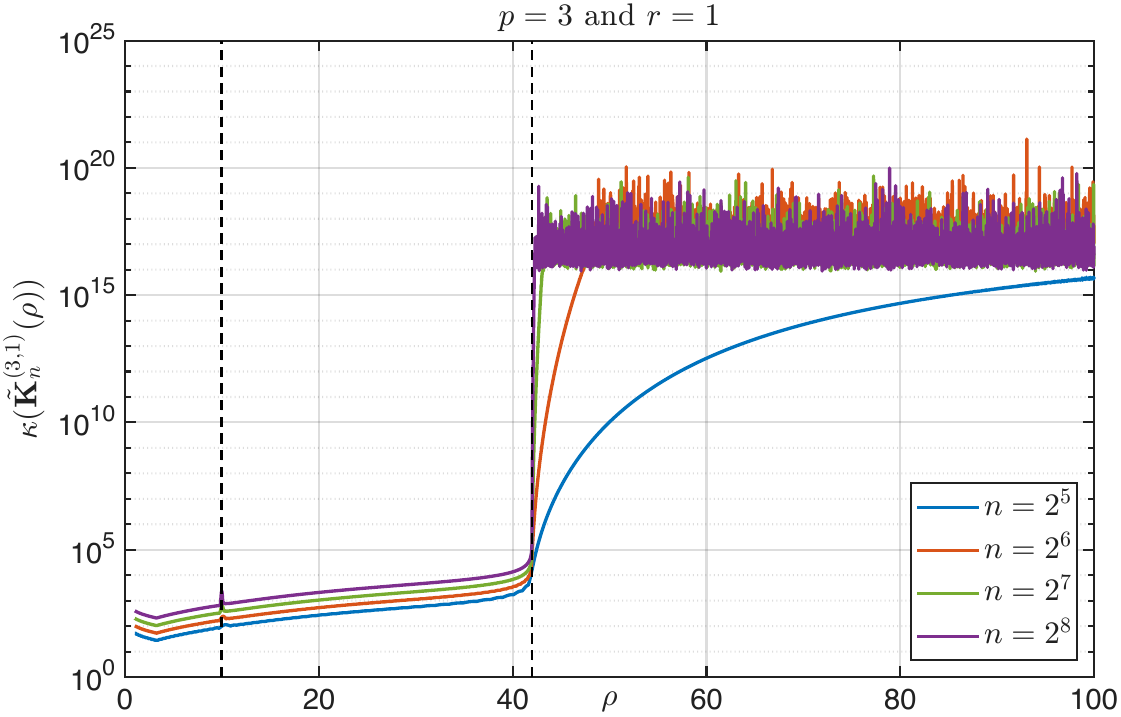}
    \end{subfigure}
    \caption{Condition numbers of the matrices $\mathbf{K}_n^{(3,1)}(\rho)$ (left), defined in~\eqref{eq:knp}, and $\widetilde{\mathbf{K}}_n^{(3,1)}(\rho)$ (right), defined in~\eqref{eq:new1}, as $\rho$ and~$n$ vary. Thresholds correspond to $\rho\in\{168/17,10,42\}$.}
    \label{fig:5}
\end{figure}

In Figure~\ref{fig:5} we report the numerically computed condition numbers of the matrices $\mathbf{K}_n^{(3,1)}(\rho)$ and $\widetilde{\mathbf{K}}_n^{(3,1)}(\rho)$ for
various values of $\rho$ and $n$. First, no differences are observed between the conditioning behaviors of the two families of matrices. Second, the numerical results
are in agreement with Corollary~\ref{cor:57} and suggest that the estimate therein is sharp: peaks of the condition number are observed close to $\rho\approx 10$ and for
$\rho>42$. Finally, for $\rho\in[0,168/17]\cup[10,42]$, where $\SR_{\rho}^{(3,1)}$ is of type $(2,2,1)$, the condition numbers appear to grow at most polynomially. This is the situation discussed in Remark~\ref{rm:221}, which is not covered by Theorems~\ref{th:G1} and~\ref{th:3.9}, and the numerical evidence is consistent with the at most polynomial growth expected there.

\subsection{Stabilized scheme}

To overcome the mesh condition, the bilinear form~\eqref{eq:48} is stabilized in~\cite{FrLo23,StZa19}.
Specifically, an unconditionally stable Petrov--Galerkin discretization is obtained by perturbing the bilinear form~\eqref{eq:48} with a suitable penalty term.
For maximal-regularity splines of degree $p\ge 1$ and regularity $C^{p-1}$, one considers the problem: find $u_{h}^{(p,p-1)}\in S_{h,0,\bullet}^{(p,p-1)}(0,\TF)$ such that
\[
a^{(p,p-1)}_{\mu,\de,h}(u_{h}^{(p,p-1)},v_{h}^{(p,p-1)})=(f, v_{h}^{(p,p-1)})_{L^{2}(0,\TF)},
\]
for all $v_{h}^{(p,p-1)}\in S_{h,\bullet,0}^{(p,p-1)}(0,\TF)$, where
\[
a^{(p,p-1)}_{\mu,\de,h}(u_{h}^{(p,p-1)},v_{h}^{(p,p-1)})
=a_{\mu}(u_{h}^{(p,p-1)},v_{h}^{(p,p-1)})
-\de\mu h^{2p} (\partial_{t}^{p} u_{h}^{(p,p-1)},\partial_{t}^{p} v_{h}^{(p,p-1)})_{L^{2}(0,\TF )},
\]
with $a_{\mu}$ as in~\eqref{eq:48} and a stabilization parameter $\de\ge 0$.

In~\cite{FeFr26}, it is shown that the method is stable for all $h$ and $\mu$ if
\begin{equation} \label{delta}
\de\ge\frac{(2^{2p+1}-1)}{2^{2p-1}\pi^{2(p+1)}}\ze(2(p+1)).
\end{equation}
We expect to stabilize problem~\eqref{eq:49} for arbitrary regularity $1\le r\le p-1$ using the bilinear form
\[
a_{\mu,\de,h}^{(p,r)}(u_{h}^{(p,r)},v_{h}^{(p,r)})
=a_{\mu}(u_{h}^{(p,r)},v_{h}^{(p,r)})
-\de\mu h^{2(r+1)} (\partial_{t}^{r+1} u_{h}^{(p,r)},\partial_{t}^{r+1} v_{h}^{(p,r)})_{L^{2}(0,\TF)},
\]
with $\de>0$ not too small.

It is important to note that the non-consistent term yields optimal convergence rates in the $L^{2}$-norm (expected to behave as $h^{p+1}$) only if
\begin{equation}\label{eq:424}
2(r+1)\ge p+1,\quad\text{that is,}\quad r\ge\lfloor (p-1)/2\rfloor.
\end{equation}

\begin{remark}
	When designing the perturbed formulation, attention must be paid to the balance between stability and consistency.
While the unstabilized problem is unstable unless a CFL condition is satisfied, the perturbed one should guarantee unconditional stability.
When $h\to 0$, the penalty term vanishes, ensuring that the discrete solution asymptotically approaches the solution of the continuous equation.
The condition~\eqref{eq:424} ensures that this ``consistency crime'' does not compromise the optimal accuracy of the scheme.
Although a variational argument establishing optimal convergence and stability is still missing, we aim to compute, via a matrix-based argument, the parameters for which the stabilized method is unconditionally stable.
\end{remark}

The system matrix associated with the stabilized problem, with respect to the basis in~\eqref{eq:41}, reads
\[
\mathbf{K}^{(p,r)}_{h,\mu}(\de)=-\mathbf{B}_{h}^{(p,r)}+\mu\mathbf{M}_{h}^{(p,r)}-\mu\de h^{2(r+1)}\mathbf{D}_{h}^{(p,r)},
\]
with $\mathbf{M}_{h}^{(p,r)}$, $\mathbf{B}_{h}^{(p,r)}$, and $\mathbf{D}_{h}^{(p,r)}$ as in~\eqref{eq:43}.
We define
\begin{equation} \label{eq:524b}
\mathbf{K}_{n}^{(p,r)}(\rho,\de)=h\mathbf{K}_{h,\mu}^{(p,r)}(\de)=-\mathbf{B}_{n}^{(p,r)}+\rho (\mathbf{M}_{n}^{(p,r)}-\de\mathbf{D}_{n}^{(p,r)}),
\end{equation}
whose entries depend on $\mu$ and $h$ only through $\rho=\mu h^{2}$.
Here, $\mathbf{M}_{n}^{(p,r)}$, $\mathbf{B}_{n}^{(p,r)}$, and $\mathbf{D}_{n}^{(p,r)}$ are defined in~\eqref{eq:44}.
We are interested in the conditioning behavior of the family of matrices $\{\mathbf{K}_{n}^{(p,r)}(\rho,\de)\}_{n}$ as $n$ increases.
In particular, we seek the smallest $\de$ such that the conditioning behavior is algebraic in $n$ for all $\rho> 0$.

\medskip
\noindent
\textbf{Case} $\boldsymbol{p=3,}$ $\boldsymbol{r=1.}$
This is the smallest case satisfying~\eqref{eq:424} that does not fit the framework of maximal regularity splines considered in~\cite{FeFr26}.

The matrices $\mathbf{M}_{n}^{(3,1)}$ and $\mathbf{B}_{n}^{(3,1)}$ are given in~\eqref{eq:418} and~\eqref{eq:419}, respectively.
The matrix $\mathbf{D}_{n}^{(3,1)}\in\bR^{(2N_{\mesh}+1)\times (2N_{\mesh}+1)}$ reads
\[
\mathbf{D}_{n}^{(3,1)}=\frac{3}{2}
\begin{pmatrix}
\newcommand{\cell}[1]{\makebox[1.4em][r]{$#1$}}
\newcommand{\lr}[1]{\multicolumn{1}{|r}{\cell{#1}}}
\newcommand{\rr}[1]{\multicolumn{1}{r|}{\cell{#1}}}
\;
\begin{array}{@{}*{11}{c}@{}}
\cline{1-2}\cline{3-4}
\lr{-12} & \cell{2} & \lr{2} & \rr{0} &&&&&&& \\
\lr{24} & \cell{-12} & \lr{0} & \rr{0} &&&&&&& \\
\cline{1-2}\cline{3-4}\cline{5-6}
\lr{-12} & \cell{16} & \lr{-8} & \rr{1} & \cell{1} & \rr{0} &&&&& \\
\lr{0} & \cell{-8} & \lr{16} & \rr{-11} & \cell{1} & \rr{0} &&&&& \\
\cline{1-2}\cline{3-4}\cline{5-6}\cline{7-8}
\lr{0} & \cell{1} & \lr{-11} & \rr{16} & \cell{-8} & \rr{1} & \cell{1} & \rr{0} &&& \\
\lr{0} & \cell{1} & \lr{1} & \rr{-8} & \cell{16} & \rr{-11} & \cell{1} & \rr{0} &&& \\
\cline{1-2}\cline{3-4}\cline{5-6}\cline{7-8}
\multicolumn{2}{c}{\ddots} & \multicolumn{2}{c}{\ddots} & \multicolumn{2}{c}{\ddots} & \multicolumn{2}{c}{\ddots} & \multicolumn{2}{c}{\ddots} & \\
\cline{3-4}\cline{5-6}\cline{7-8}\cline{9-10}
&& \lr{0} & \cell{1} & \lr{-11} & \cell{16} & \lr{-8} & \cell{1} & \lr{1} & \rr{0} & \\
&& \lr{0} & \cell{1} & \lr{1} & \cell{-8} & \lr{16} & \cell{-11} & \lr{1} & \rr{0} & \\
\cline{3-4}\cline{5-6}\cline{7-8}\cline{9-10}
&&&& \lr{0} & \cell{1} & \lr{-11} & \cell{16} & \lr{-8} & \rr{0} & \cell{2} \\
&&&& \lr{0} & \cell{1} & \lr{1} & \cell{-8} & \lr{16} & \rr{-12} & \cell{2} \\
\cline{5-6}\cline{7-8}\cline{9-10}
&&&&&& \cell{0} & \cell{0} & \cell{-12} & \cell{24} & \cell{-12} \\
\end{array}
\;\;
\end{pmatrix}.
\]

This matrix does not fit the framework of Theorem~\ref{cr:R} due to perturbations in the top-left and bottom-right corners, as well as an additional spurious row and column.
As already done for $\mathbf{M}_{n}^{(3,1)}$ and $\mathbf{B}_{n}^{(3,1)}$ in~\eqref{eq:420} and~\eqref{eq:421}, respectively, we consider the pure block Toeplitz band extension $\widetilde{\mathbf{D}}_{n}^{(3,1)}\in\bR^{2N_{\mesh}\times 2N_{\mesh}}$ given by
\[
\widetilde{\mathbf{D}}_{n}^{(3,1)}=\frac{3}{2}
\begin{pmatrix}
\newcommand{\cell}[1]{\makebox[1.4em][r]{$#1$}}
\newcommand{\lr}[1]{\multicolumn{1}{|r}{\cell{#1}}}
\newcommand{\rr}[1]{\multicolumn{1}{r|}{\cell{#1}}}
\;
\begin{array}{@{}*{11}{c}@{}}
\cline{1-2}\cline{3-4}
\lr{-8} & \cell{1} & \lr{1} & \rr{0} &&&&&&& \\
\lr{16} & \cell{-11} & \lr{1} & \rr{0} &&&&&&& \\
\cline{1-2}\cline{3-4}\cline{5-6}
\lr{-11} & \cell{16} & \lr{-8} & \rr{1} & \cell{1} & \rr{0} &&&&& \\
\lr{1} & \cell{-8} & \lr{16} & \rr{-11} & \cell{1} & \rr{0} &&&&& \\
\cline{1-2}\cline{3-4}\cline{5-6}\cline{7-8}
\lr{0} & \cell{1} & \lr{-11} & \rr{16} & \cell{-8} & \rr{1} & \cell{1} & \rr{0} &&& \\
\lr{0} & \cell{1} & \lr{1} & \rr{-8} & \cell{16} & \rr{-11} & \cell{1} & \rr{0} &&& \\
\cline{1-2}\cline{3-4}\cline{5-6}\cline{7-8}
\multicolumn{2}{c}{\ddots} & \multicolumn{2}{c}{\ddots} & \multicolumn{2}{c}{\ddots} & \multicolumn{2}{c}{\ddots} & \multicolumn{2}{c}{\ddots} & \\
\cline{3-4}\cline{5-6}\cline{7-8}\cline{9-10}
&& \lr{0} & \cell{1} & \lr{-11} & \rr{16} & \cell{-8} & \rr{1} & \cell{1} & \rr{0} & \\
&& \lr{0} & \cell{1} & \lr{1} & \rr{-8} & \cell{16} & \rr{-11} & \cell{1} & \rr{0} & \\
\cline{3-4}\cline{5-6}\cline{7-8}\cline{9-10}
&&&& \lr{0} & \cell{1} & \lr{-11} & \rr{16} & \cell{-8} & \rr{1} \\
&&&& \lr{0} & \cell{1} & \lr{1} & \rr{-8} & \cell{16} & \rr{-11} \\
\cline{5-6}\cline{7-8}\cline{9-10}
\end{array}
\;\;
\end{pmatrix}
\]

and the matrices
\[
\widetilde{\mathbf{K}}_{n}^{(3,1)}(\rho,\de)=-\widetilde{\mathbf{B}}_{n}^{(3,1)}+\rho (\widetilde{\mathbf{M}}_{n}^{(3,1)}-\de\widetilde{\mathbf{D}}_{n}^{(3,1)}).
\]

We also associate with $\widetilde{\mathbf{D}}_{n}^{(3,1)}$ the four $2\times 2$ matrices that define its Toeplitz structure
\begin{align}\label{eq:425}
\widetilde{\mathbf{D}}_{n}^{(3,1)} &\to
\overbrace{\frac{3}{2}\begin{pmatrix}[r]
0 & 1
\\ 0 & 1
\end{pmatrix}}^{2}
\quad
\overbrace{\frac{3}{2}\begin{pmatrix}[r]
-11 & 16
\\ 1 &-8
\end{pmatrix}}^{1}
\quad
\overbrace{\frac{3}{2}\begin{pmatrix}[r]
-8 & 1
\\ 16 &-11
\end{pmatrix}}^{0}
\quad
\overbrace{\frac{3}{2}\begin{pmatrix}[r]
1 & 0
\\ 1 & 0
\end{pmatrix}}^{-1}.
\end{align}

\begin{proposition} \label{prop:59}
Let $\de\ge 17/1680$.
Then, for all $\rho>0$, the polynomial $\SR_{\rho,\de}^{(3,1)}$ associated with $\widetilde{\mathbf{K}}_{n}^{(3,1)}(\rho,\delta)$ is of type $(2,2,1)$.
\end{proposition}

\begin{proof}
We compute explicitly, using~\eqref{eq:422},~\eqref{eq:423}, and~\eqref{eq:425},
\[
\det \SR_{\rho,\de}^{(3,1)}(t)=\frac{1}{11200}\, t \big(a_{\rho,\de} t^{4}+b_{\rho,\de} t^{3}+c_{\rho,\de} t^{2}+b_{\rho,\de} t+a_{\rho,\de}\big),
\]
with coefficients
\begin{align*}
a_{\rho,\de} &=-302400\de^{2}\rho^{2}-720\de\rho^{2}+20160\de\rho-\rho^{2}-48\rho-1260,
\\
b_{\rho,\de} &=1209600\de^{2}\rho^{2}+53280\de\rho^{2}+221760\de\rho+72\rho^{2}-768\rho+10080,
\\
c_{\rho,\de} &=-1814400\de^{2}\rho^{2}+197280\de\rho^{2}-483840\de\rho-262\rho^{2}+6672\rho-17640.
\end{align*}

We analyze the roots of the polynomial
\[
p_{\rho,\de}(t)=a_{\rho,\de} t^{4}+b_{\rho,\de} t^{3}+c_{\rho,\de} t^{2}+b_{\rho,\de} t+a_{\rho,\de}.
\]

As in the proof of Proposition~\ref{prop:43}, we associate with $p_{\rho,\de}$ the polynomial
\[
f_{\rho,\de}(y)=a_{\rho,\de} y^{2}+b_{\rho,\de} y+(c_{\rho,\de}-2a_{\rho,\de}),
\]
via the transformation $y=t+t^{-1}$.

We now show that if $\de\ge 17/1680$, then $f_{\rho,\de}$ has exactly one real root in the interval $[-2,2]$ and the other outside.
This corresponds to one pair of complex conjugate roots of $p_{\rho,\de}$ on the unit circle and one pair with one root strictly inside and the other strictly outside.
Hence, $\SR_{\rho,\de}^{(3,1)}$ is of type $(2,2,1)$.

We note that for all $\rho>0$ and $\de\ge 0$ the following properties hold:
\[
\begin{aligned}
& a_{\rho,\de}<0,\quad-\frac{b_{\rho,\de}}{2a_{\rho,\de}}>2,\\
& f_{\rho,\de}(2)=120\rho^{2}(2520\de-1)+5040\rho,\\[1ex]
& f_{\rho,\de}(-2)=-24\big((1680\de-17)\rho+168\big)\big((120\de-1)\rho+10\big).
\end{aligned}
\]
These properties follow from direct computations.
Thus, the parabola $f_{\rho,\de}$ opens downward and its maximum is attained at $-b_{\rho,\de}/(2a_{\rho,\de})>2$.

To ensure exactly one real root in $[-2,2]$ for all $\rho>0$, we require $f_{\rho,\de}(2)\ge 0$ and $f_{\rho,\de}(-2)\le 0$.
This leads to the conditions
\begin{align*}
2520\de-1\ge 0,\quad 1680\de-17\ge 0,
\end{align*}
which are equivalent to $\de\ge 17/1680$.
\end{proof}

\begin{remark} The threshold $\de=17/1680$ in Proposition~\ref{prop:59} is sharp: for every $\de\in[0,17/1680)$ there exists $\rho>0$ for which
$\SR_{\rho,\de}^{(3,1)}$ is of type $(3,0,2)$ instead of $(2,2,1)$.
\end{remark}

\begin{figure}
    \centering
        \centering
        \includegraphics[width=0.7\linewidth]{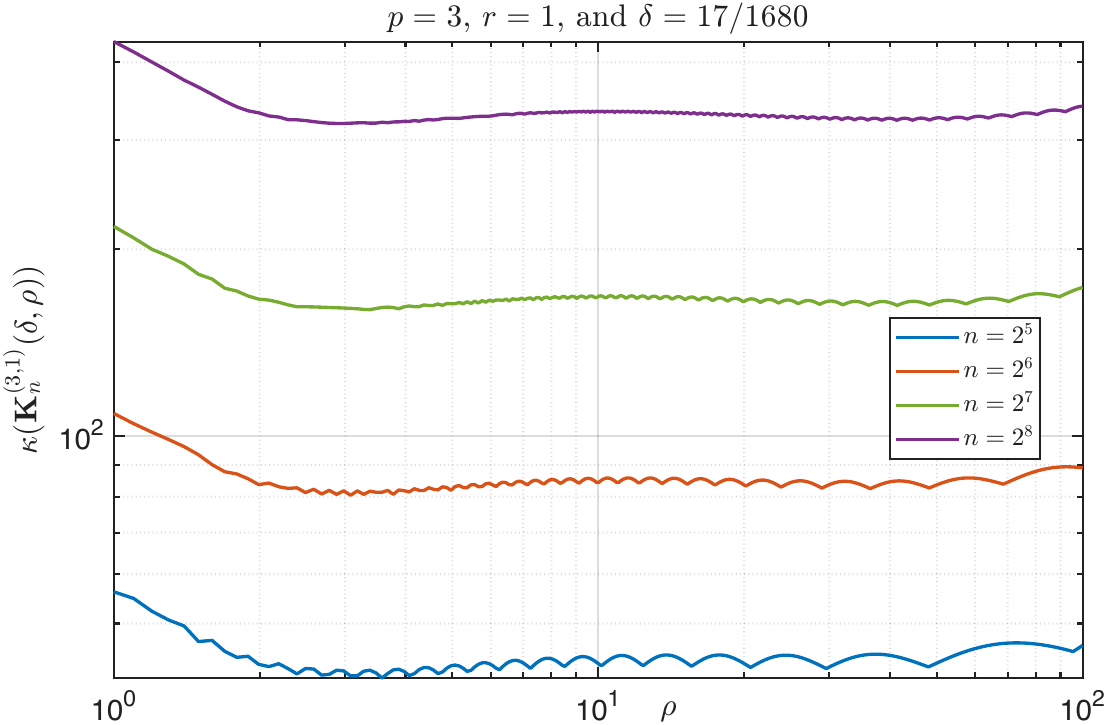}
    \caption{Condition numbers of the matrices $\mathbf{K}_n^{(3,1)}(\rho,\delta)$ defined in \eqref{eq:524b} by varying $\rho$ and~$n$ for $\delta=17/1680$.}
    \label{fig:6}
\end{figure}

Figure~\ref{fig:6} reports the numerically computed condition numbers of $\mathbf{K}_n^{(3,1)}(\rho,\delta)$ as $n$ and $\rho$ vary, for the smallest $\delta$ ($\delta=17/1680$) that guarantees type $(2,2,1)$ for all $\rho>0$. This is the regime discussed in Remark~\ref{rm:221}, which is not covered by Theorems~\ref{th:G1} and~\ref{th:3.9}. The numerically observed at most algebraic growth is consistent with the behavior expected there. Figure~\ref{fig:7} instead fixes $\rho$ and varies $\delta$: on the left, for small $\rho$, the condition numbers exhibit a peak at some $\delta\in(1/2520,17/1680)$; on the right, for large $\rho$, the value $\delta=17/1680$ is sharp, in the sense that it is the smallest $\delta$ for which the method is numerically stable. Both observations are consistent with the proof of Proposition~\ref{prop:59}. The same experiments performed on $\widetilde{\mathbf{K}}_n^{(3,1)}(\rho,\delta)$ produce indistinguishable results.

\begin{figure}
    \centering
    \begin{subfigure}{0.49\linewidth}
        \centering
        \includegraphics[width=\linewidth]{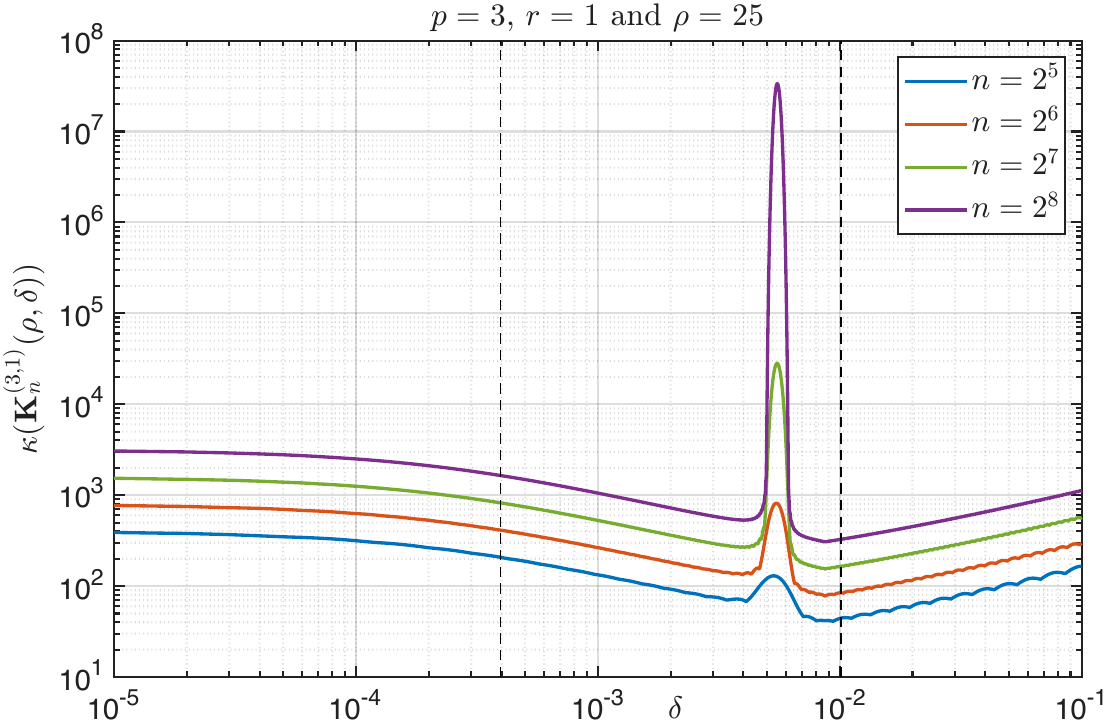}
    \end{subfigure}
    \hfill
    \begin{subfigure}{0.49\linewidth}
        \centering
        \includegraphics[width=\linewidth]{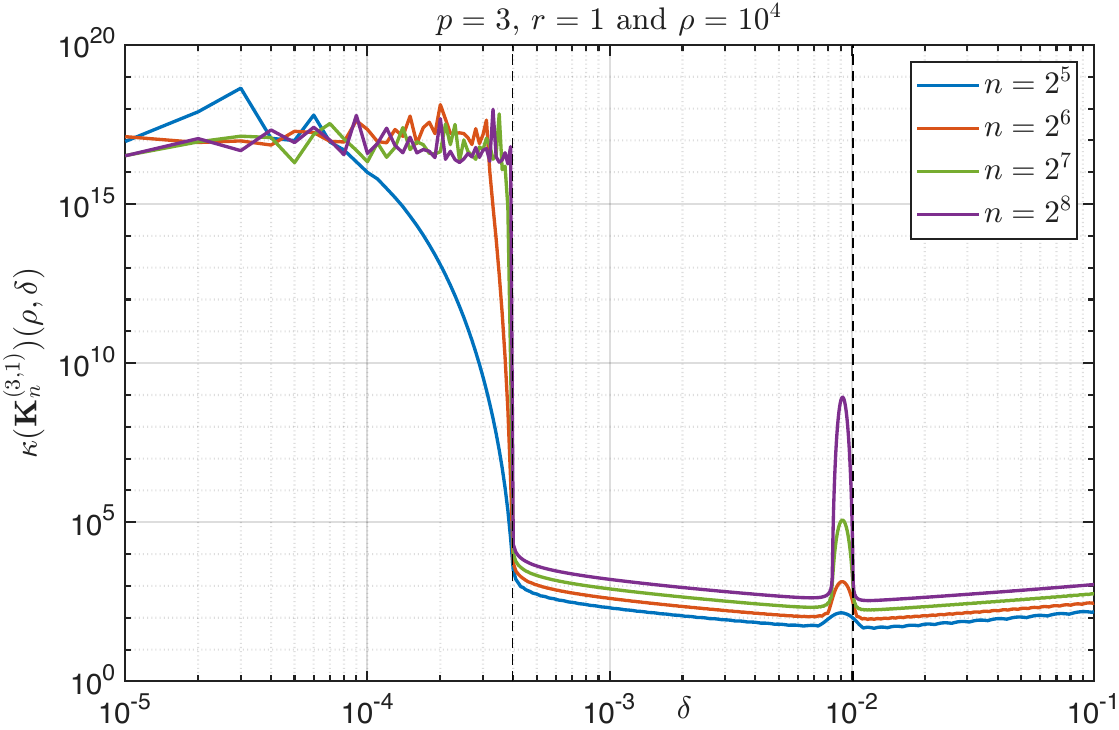}
    \end{subfigure}
    \caption{Condition numbers of the matrices $\mathbf{K}_n^{(3,1)}(\rho,\delta)$ defined in \eqref{eq:524b} by varying $n$ and~$\delta$ for $\rho=25$ (on the left) and $\rho=10^4$ (on the right). Thresholds correspond to $\delta \in \{1/2520, 17/1680\}$.}
    \label{fig:7}
\end{figure}

\subsection{Schrödinger equation} \label{sec:schro}

In the same setting as for the wave equation in~\eqref{eq:45}, we consider the linear Schrödinger equation, i.e., (\ref{eq:426}).

After multiplication by a test function and integration by parts in space, a space-time variational formulation of~\eqref{eq:426} reads as follows: find
\[
\psi\in L^{2}(0,\TF ;H^{1}_{0}(\Om))\cap H^{1}_{0,\bullet}(0,\TF ;[H_{0}^{1}(\Om)]')
\] such that
\begin{equation}\label{eq:427}
\mi\int_{0}^{\TF}\langle\partial_{t}\psi(\cdot,t),\Ph(\cdot,t)\rangle_{H^{1}_{0}(\Om)}\d t+(\nabla_{\bx}\psi,\nabla_{\bx}\Ph)_{L^{2}(Q_{\TF})}
=(F,\Ph)_{L^{2}(Q_{\TF})},
\end{equation}
for all $\Ph\in L^{2}(0,\TF;H^{1}_{0}(\Om))$, where $\langle\cdot,\cdot\rangle_{H^{1}_{0}(\Om)}$ denotes the duality pairing between $[H_{0}^{1}(\Om)]'$ and $H_{0}^{1}(\Om)$.

As in the case of the wave equation, by exploiting the Fourier expansion of the trial and test functions, the analysis of the unconditional stability of~\eqref{eq:427} reduces to studying the following parameter-dependent problem: given $\mu>0$, find $u\in H^{1}_{0,\bullet}(0,\TF)$ such that
\[
\mi (\partial_{t} u, w)_{L^{2}(0,\TF)}+\mu (u, w)_{L^{2}(0,\TF)}=(f, w)_{L^{2}(0,\TF)}\quad\text{for all}\quad w\in L^{2}(0,\TF),
\]
with $f\in L^{2}(0,\TF)$.
This problem is equivalent to determining $u\in H^{1}_{0,\bullet}(0,\TF)$ such that
\begin{equation}\label{eq:428}
b_{\mu}(u,v)=(f,\partial_{t} v)_{L^{2}(0,\TF)}\quad\text{for all}\quad v\in H^{1}_{\bullet,0}(0,\TF),
\end{equation}
where the bilinear form $b_{\mu}:H^{1}_{0,\bullet}(0,\TF)\times H^{1}_{\bullet,0}(0,\TF)\to\bC$ is given by
\[
b_{\mu}(u,v)=\mi (\partial_{t} u,\partial_{t} v)_{L^{2}(0,\TF)}-\mu (\partial_{t} u, v)_{L^{2}(0,\TF)}.
\]
The discrete counterpart of~\eqref{eq:428} amounts to looking for $u_{h}^{(p,r)}\in S_{h,0,\bullet}^{(p,r)}(0,\TF)$ such that
\[
\mi (\partial_{t} u_{h}^{(p,r)},\partial_{t} v_{h}^{(p,r)})_{L^{2}(0,\TF)}-\mu (\partial_{t} u_{h}^{(p,r)},v_{h}^{(p,r)})_{L^{2}(0,\TF)}=(f,\partial_{t} v_{h}^{(p,r)})_{L^{2}(0,\TF)},
\]
for all $v_{h}^{(p,r)}\in S_{h,\bullet,0}^{(p,r)}(0,\TF)$.

The system matrix associated with the bilinear form $b_{\mu}$, with respect to the basis introduced in~\eqref{eq:41}, reads
\[
\mathbf{S}_{h,\mu}^{(p,r)}=\mi\mathbf{B}_{h}^{(p,r)}-\mu\mathbf{C}_{h}^{(p,r)},
\]
with $\mathbf{B}_{h}^{(p,r)}$ and $\mathbf{C}_{h}^{(p,r)}$ as in~\eqref{eq:43}.
Let $\rho=\mu h$.
Then the entries of the scaled matrix
\begin{equation} \label{eq:Sn}
\mathbf{S}_{n}^{(p,r)}(\rho)=h\mathbf{S}_{h,\mu}^{(p,r)}=\mi\mathbf{B}_{n}^{(p,r)}-\rho\mathbf{C}_{n}^{(p,r)}
\end{equation}
depend on $\mu$ and $h$ only through $\rho$.
Here, $\mathbf{B}_{n}^{(p,r)}$ and $\mathbf{C}_{n}^{(p,r)}$ are defined in~\eqref{eq:44}.

We are interested in the conditioning behavior of the family of matrices $\{\mathbf{S}_{n}^{(p,r)}(\rho)\}_{n}$ as $n$ increases, by varying $\rho$.
It has been shown in~\cite{FeGo25} that for all $p\ge 1$ and $r=p-1$, the condition numbers of $\mathbf{S}_{n}^{(p,r)}(\rho)$ grow at most algebraically in $n$ for all $\rho>0$.
Here, we study the behavior for the cases $(p,r)=(2,0)$, $(p,r)=(3,0)$, and $(p,r)=(3,1)$.

\medskip
\noindent
\textbf{Case} $\boldsymbol{p=2,}$ $\boldsymbol{r=0.}$
The matrix $\mathbf{B}_{n}^{(2,0)}$ is reported in~\eqref{eq:411}, while the matrix $\mathbf{C}_{n}^{(2,0)}\in\bR^{2N_{\mesh}\times 2N_{\mesh}}$ reads

\[
\mathbf{C}_{n}^{(2,0)}=\frac{1}{6}
\begin{pmatrix}
\newcommand{\cell}[1]{\makebox[1em][r]{$#1$}}
\newcommand{\lr}[1]{\multicolumn{1}{|r}{\cell{#1}}}
\newcommand{\rr}[1]{\multicolumn{1}{r|}{\cell{#1}}}
\;
\begin{array}{@{}*{10}{c}@{}}
\cline{1-2}
\lr{2} & \rr{1} &&&&&&&& \\
\lr{0} & \rr{2} &&&&&&&& \\
\cline{1-2}\cline{3-4}
\lr{-2} & \rr{0} & \cell{2} & \rr{1} &&&&&& \\
\lr{0} & \rr{-2} & \cell{0} & \rr{2} &&&&&& \\
\cline{1-2}\cline{3-4}\cline{5-6}
\lr{0} & \rr{-1} & \cell{-2} & \rr{0} & \cell{2} & \rr{1} &&&& \\
\lr{0} & \rr{0} & \cell{0} & \rr{-2} & \cell{0} & \rr{2} &&&& \\
\cline{1-2}\cline{3-4}\cline{5-6}
\multicolumn{2}{c}{\ddots} & \multicolumn{2}{c}{\ddots} & \multicolumn{2}{c}{\ddots} & \multicolumn{2}{c}{\ddots} \\
\cline{3-4}\cline{5-6}\cline{7-8}
&& \lr{0} & \cell{-1} & \lr{-2} & \cell{0} & \lr{2} & \rr{1} && \\
&& \lr{0} & \cell{0} & \lr{0} & \cell{-2} & \lr{0} & \rr{2} && \\
\cline{3-4}\cline{5-6}\cline{7-8}\cline{9-10}
&&&& \lr{0} & \cell{-1} & \lr{-2} & \rr{0} & \cell{2} & \rr{1}\\
&&&& \lr{0} & \cell{0} & \lr{0} & \rr{-2} & \cell{0} & \rr{2}\\
\cline{5-6}\cline{7-8}\cline{9-10}
\end{array}
\;\;
\end{pmatrix}.
\]

We associate with this matrix the three $2\times 2$ matrices that define the Toeplitz structure:
\begin{equation}\label{eq:429}
\mathbf{C}_{n}^{(2,0)}\to
\overbrace{\frac{1}{6}\begin{pmatrix}[r]
0 &-1 \\
0 & 0
\end{pmatrix}}^{2}
\quad
\overbrace{\frac{1}{6}\begin{pmatrix}[r]
-2 & 0 \\
0 &-2
\end{pmatrix}}^{1}
\quad
\overbrace{\frac{1}{6}\begin{pmatrix}[r]
2 & 1 \\
0 & 2
\end{pmatrix}}^{0}.
\end{equation}

Similarly, we define $\mathbf{S}_{n}^{(2,0)}(\rho)=\mi\mathbf{B}_{n}^{(2,0)}-\rho\mathbf{C}_{n}^{(2,0)}$.

\begin{proposition}\label{P56}
The polynomial $\SR_{\rho}^{(2,0)}$ associated with $\mathbf{S}_{n}^{(2,0)}(\rho)$ is of type $(0,2,0)$ for all $\rho>0$.
\end{proposition}

\begin{proof}
Using~\eqref{eq:413} and~\eqref{eq:429}, we compute
\[
\det \SR_{\rho}^{(2,0)}(t)=\frac{1}{9}\big(a_{\rho} t^{2}+b_{\rho} t+\ol{a}_{\rho}\big),
\]
with coefficients $a_{\rho}=\rho^{2}+6\mi\rho-12$, $b_{\rho}=-2\rho^{2}+24$.
We compute the discriminant
\[
b_{\rho}^{2}-4|a_{\rho}|^{2}=-144\rho^{2},
\]
which is real and negative for all $\rho>0$.

The squared modulus of the roots is then
\[
\frac{\Big(-b_{\rho}\pm\mi\sqrt{-(b_{\rho}^{2}-4|a_{\rho}|^{2})}\Big)\Big(-b_{\rho}\mp\mi\sqrt{-(b_{\rho}^{2}-4|a_{\rho}|^{2})}\Big)}{4|a_{\rho}|^{2}}
=\frac{b_{\rho}^{2}+(-(b_{\rho}^{2}-4|a_{\rho}|^{2}))}{4|a_{\rho}|^{2}}=1.
\]
Thus, $\SR_{\rho}^{(2,0)}$ is of type $(0,2,0)$.
\end{proof}

\medskip
\noindent
\textbf{Case} $\boldsymbol{p=3,}$ $\boldsymbol{r=0.}$
The matrix $\mathbf{B}_{n}^{(3,0)}$ is shown in~\eqref{eq:415}, while the matrix $\mathbf{C}_{n}^{(3,0)}\in\bR^{3N_{\mesh}\times 3N_{\mesh}}$ reads
{\small\[
\mathbf{C}_{n}^{(3,0)}=\frac{1}{20}
\begin{pmatrix}
\newcommand{\cell}[1]{\makebox[0.9em][r]{$#1$}}
\newcommand{\lr}[1]{\multicolumn{1}{|r}{\cell{#1}}}
\newcommand{\rr}[1]{\multicolumn{1}{r|}{\cell{#1}}}
\;
\begin{array}{@{}*{15}{c}@{}}
\cline{1-3}
\lr{6} & \cell{3} & \rr{1} &&&&&&&&&&&& \\
\lr{0} & \cell{3} & \rr{3} &&&&&&&&&&&& \\
\lr{-3} & \cell{0} & \rr{6} &&&&&&&&&&&& \\
\cline{1-3}\cline{4-6}
\lr{-3} & \cell{-6} & \rr{0} & \cell{6} & \cell{3} & \rr{1} &&&&&&&&& \\
\lr{0} & \cell{0} & \rr{-6} & \cell{0} & \cell{3} & \rr{3} &&&&&&&&& \\
\lr{0} & \cell{0} & \rr{-3} & \cell{-3} & \cell{0} & \rr{6} &&&&&&&&& \\
\cline{1-3}\cline{4-6}\cline{7-9}
\lr{0} & \cell{0} & \rr{-1} & \cell{-3} & \cell{-6} & \rr{0} & \cell{6} & \cell{3} & \rr{1} &&&&&& \\
\lr{0} & \cell{0} & \rr{0} & \cell{0} & \cell{0} & \rr{-6} & \cell{0} & \cell{3} & \rr{3} &&&&&& \\
\lr{0} & \cell{0} & \rr{0} & \cell{0} & \cell{0} & \rr{-3} & \cell{-3} & \cell{0} & \rr{6} &&&&&& \\
\cline{1-3}\cline{4-6}\cline{7-9}
\multicolumn{3}{c}{\ddots} & \multicolumn{3}{c}{\ddots} & \multicolumn{3}{c}{\ddots} & \multicolumn{3}{c}{\ddots} & \\
\cline{4-6}\cline{7-9}\cline{10-12}
&&& \lr{0} & \cell{0} & \cell{-1} & \lr{-3} & \cell{-6} & \rr{0} & \cell{6} & \cell{3} & \rr{1} &&& \\
&&& \lr{0} & \cell{0} & \cell{0} & \lr{0} & \cell{0} & \rr{-6} & \cell{0} & \cell{3} & \rr{3} &&& \\
&&& \lr{0} & \cell{0} & \cell{0} & \lr{0} & \cell{0} & \rr{-3} & \cell{-3} & \cell{0} & \rr{6} &&& \\
\cline{4-6}\cline{7-9}\cline{10-12}\cline{13-15}
&&&&&& \lr{0} & \cell{0} & \rr{-1} & \cell{-3} & \cell{-6} & \rr{0} & \cell{6} & \cell{3} & \rr{1}\\
&&&&&& \lr{0} & \cell{0} & \rr{0} & \cell{0} & \cell{0} & \rr{-6} & \cell{0} & \cell{3} & \rr{3}\\
&&&&&& \lr{0} & \cell{0} & \rr{0} & \cell{0} & \cell{0} & \rr{-3} & \cell{-3} & \cell{0} & \rr{6}\\
\cline{7-9}\cline{10-12}\cline{13-15}
\end{array}
\;\;
\end{pmatrix}.
\]}

We associate with this matrix the three $3\times 3$ matrices that define the Toeplitz structure:
\[
\mathbf{C}_{n}^{(3,0)}\to
\overbrace{\frac{1}{20}\begin{pmatrix}[r]
0 & 0 &-1 \\
0 & 0 & 0 \\
0 & 0 & 0
\end{pmatrix}}^{2}
\quad
\overbrace{\frac{1}{20}\begin{pmatrix}[r]
-3 &-6 & 0 \\
0 & 0 &-6 \\
0 & 0 &-3
\end{pmatrix}}^{1}
\quad
\overbrace{\frac{1}{20}\begin{pmatrix}[r]
6 & 3 & 1 \\
0 & 3 & 3 \\
-3 & 0 & 6
\end{pmatrix}}^{0}.
\]
Similarly, we define $\mathbf{S}_{n}^{(3,0)}(\rho)=\mi\mathbf{B}_{n}^{(3,0)}-\rho\mathbf{C}_{n}^{(3,0)}$.

\begin{proposition}\label{P57}
The polynomial $\SR_{\rho}^{(3,0)}$ associated with $\mathbf{S}_{n}^{(3,0)}(\rho)$ is, for all $\rho>0$, of type $(0,2,0)$.
\end{proposition}

\begin{proof}
Computing the determinant gives
\[
\det \SR_{\rho}^{(3,0)}(t)=-\frac{9}{800}\big(a_{\rho} t^{2}+b_{\rho} t-\ol{a}_{\rho}\big),
\]
with $a_{\rho}=\rho^{3}+12\mi\rho^{2}-60\rho-120\mi$, $b_{\rho}=-24\mi\rho^{2}+240\mi$.
We evaluate the discriminant:
\begin{align*}
b_{\rho}^{2}+4|a_{\rho}|^{2}=4\rho^{2}(\rho^{2}-60)^{2},
\end{align*}
which is positive for all $\rho>0$.
The squared modulus of the roots is
\[
\frac{\Big(-b_{\rho}\pm\sqrt{b_{\rho}^{2}+4|a_{\rho}|^{2}}\Big)\Big(-b_{\rho}\mp\sqrt{b_{\rho}^{2}+4|a_{\rho}|^{2}}\Big)}{4|a_{\rho}|^{2}}
=\frac{-b_{\rho}^{2}+(b_{\rho}^{2}+4|a_{\rho}|^{2})}{4|a_{\rho}|^{2}}=1.
\]
Thus, $\SR_{\rho}^{(3,0)}$ is of type $(0,2,0)$ for all $\rho>0$.
\end{proof}

\begin{corollary} \label{cor:514}
The condition numbers of $\mathbf{S}_{n}^{(2,0)}(\rho)$ and $\mathbf{S}_{n}^{(3,0)}(\rho)$ grow at most polynomially for every $\rho>0$.
\end{corollary}

\begin{proof}
This is a direct consequence of Theorem~\ref{th:types} and Propositions \ref{P56} and \ref{P57}.
\end{proof}

Figure~\ref{fig:8} reports the condition numbers of $\mathbf{S}_n^{(2,0)}(\rho)$ and $\mathbf{S}_n^{(3,0)}(\rho)$. In contrast with Figures~\ref{fig:3} and~\ref{fig:4} for the wave equation, and in agreement with Corollary~\ref{cor:514}, these grow only algebraically in $n$ for every $\rho>0$.

\begin{figure}
    \centering
    \begin{subfigure}{0.49\linewidth}
        \centering
        \includegraphics[width=\linewidth]{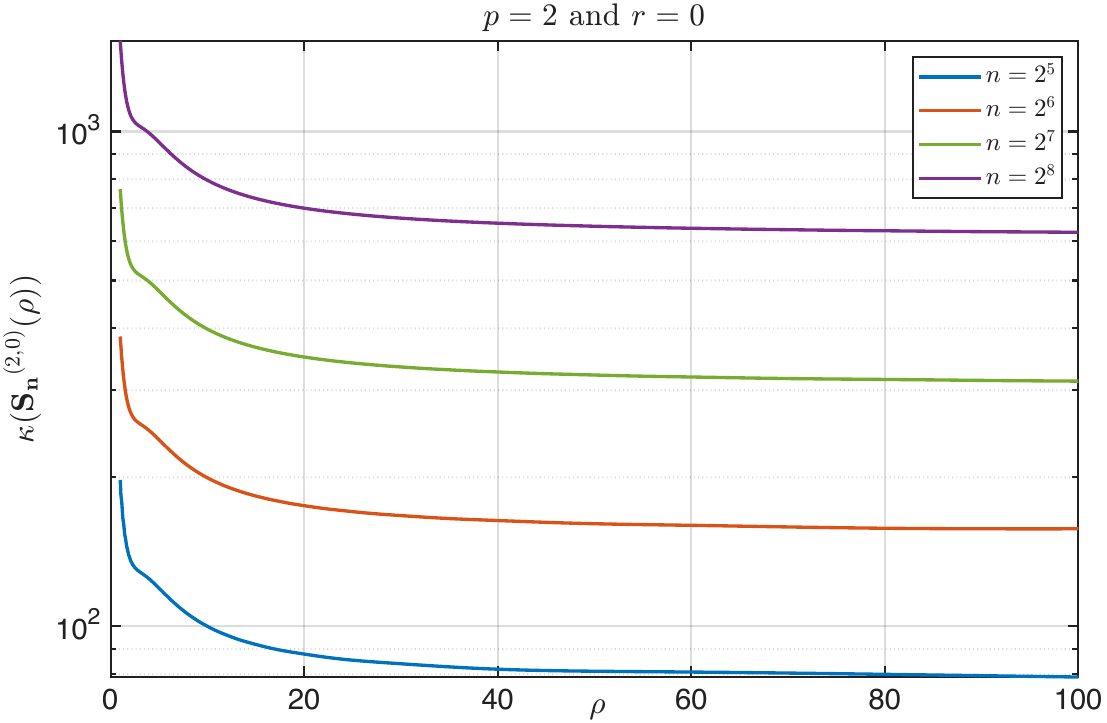}
    \end{subfigure}
    \hfill
    \begin{subfigure}{0.49\linewidth}
        \centering
        \includegraphics[width=\linewidth]{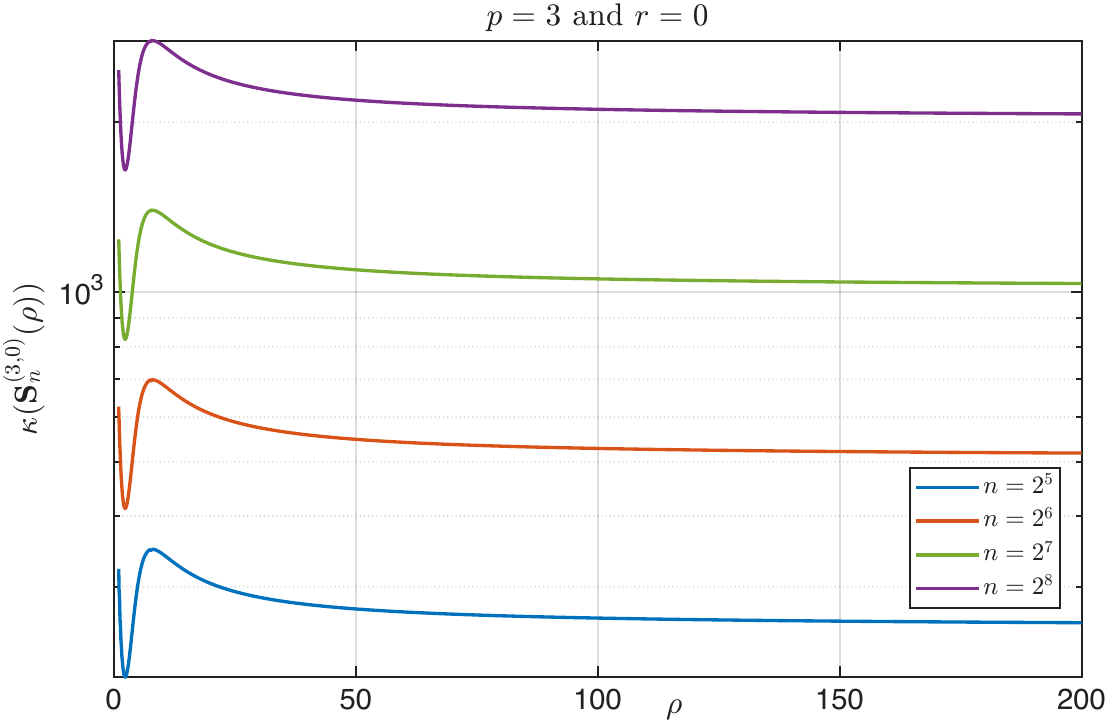}
    \end{subfigure}
    \caption{Condition numbers of the matrices $\mathbf{S}_n^{(2,0)}(\rho)$ (on the left) and $\mathbf{S}_n^{(3,0)}(\rho)$ (on the right) defined in \eqref{eq:Sn}, respectively, by varying $\rho$ and~$n$.}
    \label{fig:8}
\end{figure}

\medskip
\noindent
\textbf{Case} $\boldsymbol{p=3,}$ $\boldsymbol{r=1.}$
The matrix $\mathbf{B}_{n}^{(3,1)}$ is the one in~\eqref{eq:419}, and the matrix $\mathbf{C}_{n}^{(3,1)}\in\bR^{(2N_{\mesh}+1)\times (2N_{\mesh}+1)}$ is

\[
\mathbf{C}_{n}^{(3,1)}=\frac{1}{80}\begin{pmatrix}
\newcommand{\cell}[1]{\makebox[1.4em][r]{$#1$}}
\newcommand{\lr}[1]{\multicolumn{1}{|r}{\cell{#1}}}
\newcommand{\rr}[1]{\multicolumn{1}{r|}{\cell{#1}}}
\;
\begin{array}{@{}*{12}{c}@{}}
\cline{1-2}\cline{3-4}
\lr{24} & \cell{14} & \lr{2} & \rr{0} &&&&&&&& \\
\lr{0} & \cell{18} & \lr{6} & \rr{0} &&&&&&&& \\
\cline{1-2}\cline{3-4}\cline{5-6}
\lr{-18} & \cell{0} & \lr{24} & \rr{7} & \cell{1} & \rr{0} &&&&&& \\
\lr{-6} & \cell{-24} & \lr{0} & \rr{25} & \cell{7} & \rr{0} &&&&&& \\
\cline{1-2}\cline{3-4}\cline{5-6}\cline{7-8}
\lr{0} & \cell{-7} & \lr{-25} & \rr{0} & \cell{24} & \rr{7} & \cell{1} & \rr{0} &&&& \\
\lr{0} & \cell{-1} & \lr{-7} & \rr{-24} & \cell{0} & \rr{25} & \cell{7} & \rr{0} &&&& \\
\cline{1-2}\cline{3-4}\cline{5-6}\cline{7-8}
\multicolumn{2}{c}{\ddots} & \multicolumn{2}{c}{\ddots} & \multicolumn{2}{c}{\ddots} & \multicolumn{2}{c}{\ddots} & \multicolumn{2}{c}{\ddots}&& \\
\cline{3-4}\cline{5-6}\cline{7-8}\cline{9-10}
&& \lr{0} & \cell{-7} & \lr{-25} & \cell{0} & \lr{24} & \cell{7} & \lr{1} & \rr{0} && \\
&& \lr{0} & \cell{-1} & \lr{-7} & \cell{-24} & \lr{0} & \cell{25} & \lr{7} & \rr{0} && \\
\cline{3-4}\cline{5-6}\cline{7-8}\cline{9-10}
&&&& \lr{0} & \cell{-7} & \lr{-25} & \cell{0} & \lr{24} & \rr{6} & \cell{2} \\
&&&& \lr{0} & \cell{-1} & \lr{-7} & \cell{-24} & \lr{0} & \rr{18} & \cell{14} \\
\cline{5-6}\cline{7-8}\cline{9-10}
&&&&&& \cell{0} & \cell{-6} & \cell{-18} & \cell{0} & \cell{24} \end{array}
\;\;
\end{pmatrix}.
\]

Once again, this matrix does not fit the framework of Theorem~\ref{cr:R} due to perturbations in the top-left and bottom-right corners, as well as an additional spurious row and column.
We consider the pure block Toeplitz band extension $\widetilde{\mathbf{C}}_{n}^{(3,1)}\in\bR^{2N_{\mesh}\times 2N_{\mesh}}$:
\[
\widetilde{\mathbf{C}}_{n}^{(3,1)}=\frac{1}{80}\begin{pmatrix}
\newcommand{\cell}[1]{\makebox[1.4em][r]{$#1$}}
\newcommand{\lr}[1]{\multicolumn{1}{|r}{\cell{#1}}}
\newcommand{\rr}[1]{\multicolumn{1}{r|}{\cell{#1}}}
\;
\begin{array}{@{}*{10}{c}@{}}
\cline{1-2}\cline{3-4}
\lr{24} & \cell{7} & \lr{1} & \rr{0} &&&&&& \\
\lr{0} & \cell{25} & \lr{7} & \rr{0} &&&&&& \\
\cline{1-2}\cline{3-4}\cline{5-6}
\lr{-25} & \cell{0} & \lr{24} & \rr{7} & \cell{1} & \rr{0} &&&& \\
\lr{-7} & \cell{-24} & \lr{0} & \rr{25} & \cell{7} & \rr{0} &&&& \\
\cline{1-2}\cline{3-4}\cline{5-6}\cline{7-8}
\lr{0} & \cell{-7} & \lr{-25} & \rr{0} & \cell{24} & \rr{7} & \cell{1} & \rr{0} && \\
\lr{0} & \cell{-1} & \lr{-7} & \rr{-24} & \cell{0} & \rr{25} & \cell{7} & \rr{0} && \\
\cline{1-2}\cline{3-4}\cline{5-6}\cline{7-8}
\multicolumn{2}{c}{\ddots} & \multicolumn{2}{c}{\ddots} & \multicolumn{2}{c}{\ddots} & \multicolumn{2}{c}{\ddots} & \multicolumn{2}{c}{\ddots}\\
\cline{3-4}\cline{5-6}\cline{7-8}\cline{9-10}
&& \lr{0} & \cell{-7} & \lr{-25} & \cell{0} & \lr{24} & \rr{7} & \cell{1} & \rr{0}\\
&& \lr{0} & \cell{-1} & \lr{-7} & \cell{-24} & \lr{0} & \rr{25} & \cell{7} & \rr{0}\\
\cline{3-4}\cline{5-6}\cline{7-8}\cline{9-10}
&&&& \lr{0} & \cell{-7} & \lr{-25} & \rr{0} & \cell{24} & \rr{7}\\
&&&& \lr{0} & \cell{-1} & \lr{-7} & \rr{-24} & \cell{0} & \rr{25}\\
\cline{5-6}\cline{7-8}\cline{9-10}
\end{array}
\;\;
\end{pmatrix}.
\]

We then define
\begin{equation} \label{eq:Sntilde}
\widetilde{\mathbf{S}}_{n}^{(3,1)}(\rho)=\mi\widetilde{\mathbf{B}}_{n}^{(3,1)}-\rho\widetilde{\mathbf{C}}_{n}^{(3,1)}.
\end{equation}

We associate with $\widetilde{\mathbf{C}}_{n}^{(3,1)}$ the four $2\times 2$ matrices that define its Toeplitz structure:
\begin{align*}
\widetilde{\mathbf{C}}_{n}^{(3,1)}\to
\overbrace{\frac{1}{80}\begin{pmatrix}[r]
0 &-7 \\
0 &-1
\end{pmatrix}}^{2}
\quad
\overbrace{\frac{1}{80}\begin{pmatrix}[r]
-25 & 0 \\
-7 &-24
\end{pmatrix}}^{1}
\quad
\overbrace{\frac{1}{80}\begin{pmatrix}[r]
24 & 7 \\
0 & 25
\end{pmatrix}}^{0}
\quad
\overbrace{\frac{1}{80}\begin{pmatrix}[r]
1 & 0 \\
7 & 0
\end{pmatrix}}^{-1}.
\end{align*}

\begin{proposition} \label{prop:515}
For all $\rho>0$, the polynomial $\SR_{\rho}^{(3,1)}$ associated with $\widetilde{\mathbf{S}}_{n}^{(3,1)}(\rho)$ is of type $(2,2,1)$.
\end{proposition}

\begin{proof}
The determinant can be computed explicitly:
\[
\det \SR_{\rho}^{(3,1)}(t)=\frac{3}{800}t\big(a_{\rho} t^{4}+b_{\rho} t^{3}+c_{\rho} t^{2}+\ol{b}_{\rho} t+\ol{a}_{\rho}\big),
\]
where the coefficients are
\begin{align*}
a_{\rho}=-\rho^{2}-10\mi\rho+30,\quad
b_{\rho}=24\rho^{2}+80\mi\rho-240,\quad
c_{\rho}=-46\rho^{2}+420.
\end{align*}
We analyze the roots of the polynomial
\[
p_{\rho}(t)=a_{\rho} t^{4}+b_{\rho} t^{3}+c_{\rho} t^{2}+\ol{b}_{\rho} t+\ol{a}_{\rho}.
\]
Note that if $t_{0}$ is a root of $p_{\rho}$, then $\ol{t}_{0}^{-1}$ is also a root.
Moreover, $t=1$ is a root; hence, another root lies on the boundary of the unit circle.
We divide $p_{\rho}(t)$ by $t-1$ to obtain
\[
p_{\rho}(t)=(t-1)\big(A_{\rho} t^{3}+B_{\rho} t^{2}-\ol{B}_{\rho} t-\ol{A}_{\rho}\big),
\]
with $A_{\rho}=a_{\rho}$, $B_{\rho}=a_{\rho}+b_{\rho}$.
To count the number of zeros of
\[
q_{\rho}(t)=A_{\rho} t^{3}+B_{\rho} t^{2}-\ol{B}_{\rho} t-\ol{A}_{\rho}
\]
on the boundary of the unit circle, we apply the transformation $t=(1+\mi x)/(1-\mi x)$ with $x\in\bR$.
We obtain
\[
(1-\mi x)^{3} q_{\rho}\bigg(\frac{1+\mi x}{1-\mi x}\bigg)=8\mi R_{\rho}(x),
\]
with
\[
R_{\rho}(x)=6(\rho^{2}-10)x^{3}+25\rho x^{2}+5(\rho^{2}-6)x+15\rho.
\]
Real roots of $R_{\rho}$ correspond to roots of $q_{\rho}$ on the boundary of the unit circle.
We compute the discriminant of the cubic (in $x$) polynomial $R_{\rho}$:
\[
\begin{aligned}
\mathrm{Disc}(R_{\rho})&=-3000\rho^{8}+83425\rho^{6}-855000\rho^{4}-5269500\rho^{2}-6480000\\
&=-25\rho^{4}(120\rho^{4}-3337\rho^{2}+34200)-5269500\rho^{2}-6480000.
\end{aligned}
\]
We note that for all $\rho>0$, we have $\mathrm{Disc}(R_{\rho})<0$, since the quadratic polynomial $120y^{2}-3337y+34200$ assumes strictly positive values for all $y\in\bR$.
This implies that $R_{\rho}$ has exactly one real root.
Consequently, $q_{\rho}$ has exactly one root on the boundary of the unit circle, one strictly inside, and one strictly outside.
\end{proof}

\begin{figure}
    \centering
    \begin{subfigure}{0.49\linewidth}
        \centering
        \includegraphics[width=\linewidth]{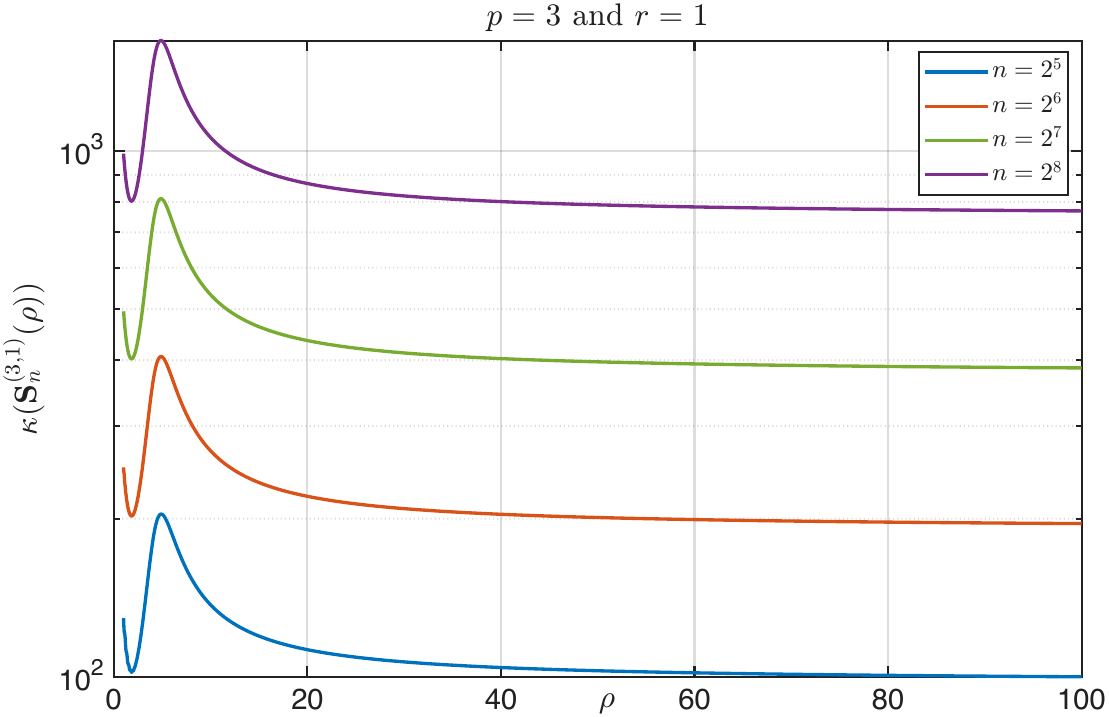}
    \end{subfigure}
    \hfill
    \begin{subfigure}{0.49\linewidth}
        \centering
        \includegraphics[width=\linewidth]{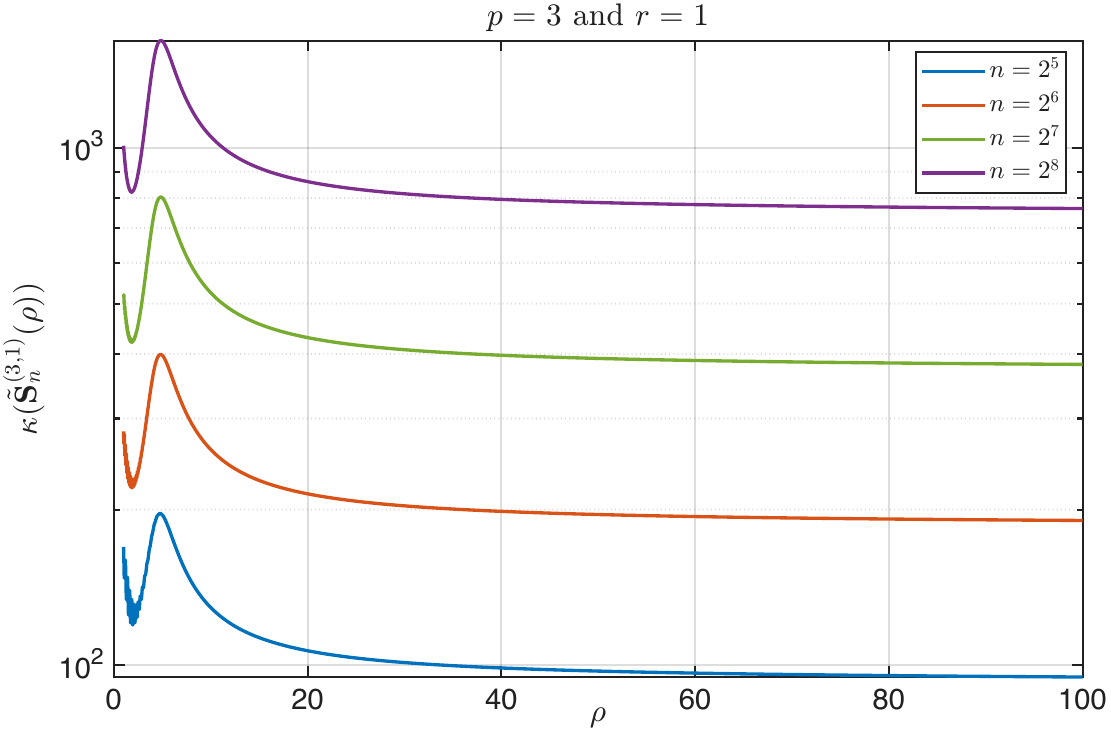}
    \end{subfigure}
    \caption{Condition numbers of the matrices $\mathbf{S}_n^{(3,1)}(\rho)$ (on the left) and $\widetilde{\mathbf{S}}_n^{(3,1)}(\rho)$ (on the right) defined in \eqref{eq:Sn} and \eqref{eq:Sntilde}, respectively, by varying $\rho$ and~$n$.}
    \label{fig:9}
\end{figure}

In Figure~\ref{fig:9}, we report the condition numbers of the matrices $\mathbf{S}_n^{(3,1)}(\rho)$ (on the left) and $\widetilde{\mathbf{S}}_n^{(3,1)}(\rho)$ (on the right). No differences are observed between the two families. Since $\SR_{\rho}^{(3,1)}$ is of type $(2,2,1)$ for all $\rho>0$ by Proposition~\ref{prop:515}, this is again the regime discussed in Remark~\ref{rm:221}, which is not covered by Theorems~\ref{th:G1} and~\ref{th:3.9}; the numerically observed at most algebraic growth is consistent with the behavior expected there.

\section{Conclusions}\label{sec:end}

In the present work we have obtained results in three different directions.
\begin{itemize}
\item We took a substantial step in generalizing the work in~\cite{AmoBru96,BoGr99} by obtaining useful bounds for the condition numbers of block Toeplitz matrix-sequences, by describing three regimes, uniformly bounded, polynomially unbounded, exponentially unbounded.
\item We have rewritten the results in the language of the work~\cite{AmoBru96}, results of which are currently employed in the numerical analysis community.
\item We have finally applied the theoretical results to stability issues of the matrix-sequences arising from the IgA approximation with intermediate regularity of the linear wave and linear Schr\"odinger equations, thus generalizing the results in~\cite{FeFr26,FeGo25}, which hold for the case of maximal regularity, that is, for $N=p-r=1$.
\end{itemize}

A first open problem is a systematic study of the symbol of the matrices associated with splines of intermediate regularity, like the one performed in~\cite{GaMa14,FeFr26} for the maximal regularity case. This would possibly extend the analysis started in this paper from some examples to general formulas, which are currently available only for the case of maximal regularity, like~\eqref{rho} and~\eqref{delta}. Another open problem is represented by an exhaustive study of the stability taking into account the low-rank perturbations which appear in concrete applications, such as the case $(p,r) = (3,1)$. Furthermore, it is highly desirable to have an analogue of Theorem \ref{th:types} for $t^{-1}S(t)$ in place of $S(t)$. Note that in particular the type $(2, 2, 1)$ is so far intractable unless the matrices are not block triangular.

As further future work, it would be nice to incorporate graded meshes and variable coefficients in the PDEs considered here.
From a discrete viewpoint, instead of block Toeplitz matrix-sequences we would end up with block generalized locally Toeplitz matrix-sequences for which spectral tools have been given in the quoted literature~\cite{BaGa20b,BaGa20a} and already used in connections with several approximation methods; see~\cite{BaGa20b,BaGa20a,GaMa14,GaSe18a,GaSe15,GaSp19} and references therein.
An ambitious goal would be the extension of the present work to this more involved setting, and this should indeed be the subject of future investigations.


\begin{thebibliography}{10}

\bibitem{AmoBru96}
P.~Amodio and L.~Brugnano.
\newblock The conditioning of {T}oeplitz band matrices.
\newblock {\em Math. Comput. Modelling}, 23(10):29--42, 1996.

\bibitem{BaGa20b}
G.~Barbarino, C.~Garoni, and S.~Serra-Capizzano.
\newblock Block generalized locally {T}oeplitz sequences: {T}heory and applications in the multidimensional case.
\newblock {\em Electron. Trans. Numer. Anal.}, 53:113--216, 2020.

\bibitem{BaGa20a}
G.~Barbarino, C.~Garoni, and S.~Serra-Capizzano.
\newblock Block generalized locally {T}oeplitz sequences: {T}heory and applications in the unidimensional case.
\newblock {\em Electron. Trans. Numer. Anal.}, 53:28--112, 2020.

\bibitem{BoDo09}
A.~B\"ottcher and P.~D\"orfler.
\newblock On the best constants in inequalities of the Markov and Wirtinger types for polynomials on the half-line.
\newblock {\em Linear Algebra Appl.}, 430:1057--1069, 2009.

\bibitem{BoGr99}
A.~B\"{o}ttcher and {\relax S.M}.~Grudsky.
\newblock {T}oeplitz band matrices with exponentially growing condition numbers.
\newblock {\em Electronic J. Linear Algebra}, 5(1):104--125, 1999.

\bibitem{BoGr05}
A.~B\"{o}ttcher and {\relax S.M}.~Grudsky.
\newblock {\em Spectral properties of banded Toeplitz matrices}.
\newblock SIAM, Philadelphia, 2005.

\bibitem{BoSi99}
A.~B\"{o}ttcher and B.~Silbermann.
\newblock {\em Introduction to large truncated {T}oeplitz matrices}.
\newblock Universitext. Springer-Verlag, New York, 1999.

\bibitem{BoWi06}
A.~B\"{o}ttcher and H.~Widom.
\newblock From Toeplitz eigenvalues through Green's kernels to higher-order Wirtinger-Sobolev inequalities.
\newblock {\em Operator Theory: Adv. and Appl.}, 171:73--87, 2006.

\bibitem{CoHu09}
J.~Cottrell, T.~Hughes, and Y.~Bazilevs.
\newblock {\em Isogeometric analysis: {T}oward integration of {CAD} and {FEA}}.
\newblock John Wiley \& Sons, 2009.

\bibitem{Da75a}
{\relax K.M}.~Day.
\newblock Measures associated with {T}oeplitz matrices generated by the {L}aurent expansion of rational functions.
\newblock {\em Trans. Amer. Math. Soc.}, 209:175--183, 1975.

\bibitem{Da75b}
{\relax K.M}.~Day.
\newblock Toeplitz matrices generated by the {L}aurent series expansion of an arbitrary rational function.
\newblock {\em Trans. Amer. Math. Soc.}, 206:224--245, 1975.

\bibitem{dB01}
C.~de~Boor.
\newblock {\em A practical guide to splines}, volume~27 of {\em Applied Mathematical Sciences}.
\newblock Springer-Verlag, New York, 2001.

\bibitem{DoNeySe12}
M.~Donatelli, M.~Neytcheva, and S.~Serra-Capizzano.
\newblock Canonical eigenvalue distribution of multilevel block {T}oeplitz sequences with non-{H}ermitian symbols.
\newblock {\em Oper. Theory Adv. Appl.}, 221:269--291, 2012.

\bibitem{Goh1}
H.~Dym and I.~Gohberg.
\newblock On unitary interpolants and {F}redholm infinite block {T}oeplitz matrices.
\newblock {\em Integral Equations Operator Theory}, 6(6):863--878, 1983.

\bibitem{FeFr26}
M.~Ferrari and S.~Fraschini.
\newblock Stability of conforming space–time isogeometric methods for the wave equation.
\newblock {\em Math. Comp.}, 95(1):683--719, 2026.

\bibitem{FeFrLoPe25}
M.~Ferrari, S.~Fraschini, G.~Loli, and I.~Perugia.
\newblock Unconditionally stable space-time isogeometric discretization for the wave equation in {H}amiltonian formulation.
\newblock {\em ESAIM: Math. Model. Numer. Anal.}, 59(5):2447--2490, 2025.

\bibitem{FeGo25}
M.~Ferrari and S.~Gómez.
\newblock Unconditionally stable space-time isogeometric method for the linear {S}chr\"odinger equation, 2025.

\bibitem{FrLo23}
S.~Fraschini, G.~Loli, A.~Moiola, and G.~Sangalli.
\newblock An unconditionally stable space–time isogeometric method for the acoustic wave equation.
\newblock {\em Comput. Math. Appl.}, 169:205--222, 2024.

\bibitem{GaMa14}
C.~Garoni, C.~Manni, F.~Pelosi, S.~Serra-Capizzano, and H.~Speleers.
\newblock On the spectrum of stiffness matrices arising from isogeometric analysis.
\newblock {\em Numer. Math.}, 127:751--799, 2014.

\bibitem{GaSe18a}
C.~Garoni and S.~Serra-Capizzano.
\newblock Generalized {L}ocally {T}oeplitz sequences: {A} spectral analysis tool for discretized differential equations.
\newblock In {\em Splines and PDEs: From Approximation Theory to Numerical Linear Algebra}, pages 161--236. Springer, Cham, 2018.

\bibitem{GaSe15}
C.~Garoni, S.~Serra-Capizzano, and D.~Sesana.
\newblock Spectral analysis and spectral symbol of $d$-variate $\mathbb{Q}_p$ {L}agrangian {FEM} stiffness matrices.
\newblock {\em SIAM J. Matrix Anal. Appl.}, 36(3):1100--1128, 2015.

\bibitem{GaSp19}
C.~Garoni, H.~Speleers, S.-E. Ekstr{\"{o}}m, A.~Reali, S.~Serra-Capizzano, and T.~Hughes.
\newblock Symbol-based analysis of finite element and isogeometric {B}-spline discretizations of eigenvalue problems: {E}xposition and review.
\newblock {\em Arch. Comput. Method. E.}, 26(5):1639--1690, 2019.

\bibitem{GoFe}
I.~Gohberg and {\relax I.A}.~Feldman.
\newblock {\em Convolution equations and projection methods for their solution}.
\newblock Amer. Math. Soc., Providence, RI, 1974.

\bibitem{Goh2}
I.~Gohberg and {\relax M.A}.~Kaashoek.
\newblock Block {T}oeplitz operators with rational symbols.
\newblock {\em Operator Theory Adv. Appl.}, 35:385--440, 1988.

\bibitem{Goh3}
I.~Gohberg and {\relax M.A}.~Kaashoek.
\newblock Projection method for block {T}oeplitz operators with operator-valued symbols.
\newblock {\em Operator Theory Adv. Appl.}, 71: 79--104, 1994.

\bibitem{Ha24}
{\relax J.I.M}.~Hauser.
\newblock Space-time {FEM} for the vectorial wave equation under consideration of {O}hm's law.
\newblock {\em Comput. Methods Appl. Math.}, 24(3):693--723, 2024.

\bibitem{HuCoBa05}
{\relax T.J.R}.~Hughes, {\relax J.A}.~Cottrell, and Y.~Bazilevs.
\newblock Isogeometric analysis: Cad, finite elements, nurbs, exact geometry and mesh refinement.
\newblock {\em Computer Methods in Applied Mechanics and Engineering}, 194(39--41):4135--4195, 2005.

\bibitem{LiSp87}
G.~Litvinchuk and {\relax I.M}.~Spitkovskii.
\newblock {\em Factorization of measurable matrix functions}, volume~25 of {\em Operator Theory: Advances and Applications}.
\newblock Birkh\"auser Verlag, Basel, 1987.

\bibitem{Pra96}
{\relax V.V}.~Prasolov.
\newblock {\em Problems and theorems in linear algebra}.
\newblock Amer. Math. Soc., Providence, RI, 1996.

\bibitem{Se99e}
S.~Serra-Capizzano.
\newblock Asymptotic results on the spectra of block {T}oeplitz preconditioned matrices.
\newblock {\em SIAM J. Matrix Anal. Appl.}, 20(1):31--44, 1999.

\bibitem{Se99f}
S.~Serra-Capizzano.
\newblock Spectral and computational analysis of block {T}oeplitz matrices having nonnegative definite matrix-valued generating functions.
\newblock {\em BIT}, 39(1):152--175, 1999.

\bibitem{SeTi99}
S.~Serra-Capizzano and P.~Tilli.
\newblock Extreme singular values and eigenvalues of non-{H}ermitian block {T}oeplitz matrices.
\newblock {\em J. Comput. Appl. Math.}, 108(1/2):113--130, 1999.

\bibitem{Si64}
{\relax I.B}.~Simonenko.
\newblock The {R}iemann boundary-value problem for $n$ pairs of functions with measurable coefficients and its application to the study of singular integrals in {$L_{p}$} spaces with weights.
\newblock {\em Izv. Akad. Nauk SSSR Ser. Mat.}, 28:277--306, 1964.

\bibitem{StZa19}
O.~Steinbach and M.~Zank.
\newblock {\em A stabilized space-time finite element method for the wave equation}, volume 128 of {\em Lect. Notes Comput. Sci. Eng.}
\newblock Springer, 2019.

\bibitem{StZa20}
O.~Steinbach and M.~Zank.
\newblock Coercive space-time finite element methods for boundary value problems.
\newblock {\em Electron. Trans. Numer. Anal.}, 52:154--194, 2020.

\bibitem{Ti98a}
P.~Tilli.
\newblock A note on the spectral distribution of {T}oeplitz matrices.
\newblock {\em Linear Multilin. Algebra}, 45(2-3):147--159, 1998.

\bibitem{WiBlo3}
H.~Widom.
\newblock Asymptotic behavior of block {T}oeplitz matrices and determinants.
\newblock {\em Advances in Math.}, 13:284--322, 1974.

\bibitem{WiBlo2}
H.~Widom.
\newblock On the limit of block {T}oeplitz determinants.
\newblock {\em Proc. Amer. Math. Soc.}, 50:167--173, 1975.

\bibitem{WiBlo1}
H.~Widom.
\newblock Asymptotic behavior of block {T}oeplitz matrices and determinants. {II}.
\newblock {\em Advances in Math.}, 21(1):1--29, 1976.

\end{thebibliography}
\end{document}